\documentclass[10pt,reqno]{amsart}
\usepackage{times}
\usepackage{color}
\usepackage{appendix}
\usepackage{psfrag,graphicx, 
 }
\usepackage{amsfonts,amssymb}
\usepackage{amsmath,amscd,amsxtra,latexsym,mathrsfs}

\newtheorem{theorem}{Theorem}[section]
\newtheorem{lemma}[theorem]{Lemma}
\newtheorem{proposition}[theorem]{Proposition}
\numberwithin{equation}{section}
\newtheorem{corollary}[theorem]{Corollary}

\theoremstyle{definition}		
\newtheorem{definition}[theorem]{Definition}

\newtheorem{remark}[theorem]{Remark}

\usepackage{tikz} 

\usetikzlibrary{matrix,arrows} 
\newfont{\Bb}{msbm10 scaled 1200}
\def\RR{{\mathbf R}}

\def\bc{\begin{center}}
\def\ec{\end{center}}

\def\R{{{\mathbf R}}}
\def\E{{{\mathbf E}}}
\def\F{{{\mathcal F}}}

\def\DD{{\mathbb D}}
\def\P{{\mathbb P}}

\def\N{{{\mathbf N}}}

\def\L{{{\mathcal L}}}
\def\mcS{{{\mathcal S}}}
\def\rms{{{\rm s}}}

\def\l{\ell}

\def\sgn{\mathop {\rm sign}}
\def\n{{   {   \vec n}}}
\def\Ric{{\mathop  {\rm Ric}}}
\def\ric{{\mathop  {\rm ric}}}

\def\<{\langle}
\def\>{\rangle}

\def\paral{/\kern-0.55ex/}
\def\parals_#1{/\kern-0.55ex/_{\!#1}}
\def\beq{\begin{equation}}
\def\eeq{\end{equation}}
\def\ben{\begin{enumerate}}
\def\een{\end{enumerate}}

\def\1{{\chi}}

\def\paral{/\kern-0.3em/}
\def\parals_#1{/\kern-0.3em/_{\!#1}}

\def\B{{\mathcal B}}

\def\di{\displaystyle}
\def\f{\frac}
\def\a{\alpha }

\def\D{\Delta }
\def\d{\delta }
\def\e{\varepsilon }

\def\na{\nabla }

\def\om{\omega }

\def\s{\sigma }

\def\trace{\mathop{\rm trace}}

\newcommand{\Label}[1]{\label{#1}}

\newcommand{\SF}{{\mathscr F}}
\newcommand{\SH}{{\mathscr H}}

\newcommand{\SR}{{\mathscr R}}
\newcommand{\scrS}{{\mathscr S}}

\begin{document}

\title[]{Reflected Brownian Motion: selection, approximation and Linearization }
\author[Marc Arnaudon]{Marc Arnaudon} \address{Institut de Math\'ematiques de Bordeaux \hfill\break\indent CNRS: UMR 5251\hfill\break\indent
  Universit\'e de Bordeaux  \hfill\break\indent
  F33405 TALENCE Cedex, France}
\email{marc.arnaudon@math.u-bordeaux.fr}
\author[Xue-Mei Li]{Xue-Mei Li} \address{Mathematics Institute \hfill\break\indent The University of Warwick\hfill\break\indent
 Coventry CV4 7AL, United Kingdom}
\email{xue-mei.li@warwick.ac.uk}

\begin{abstract}
   We construct a family of  SDEs whose solutions select a reflected Brownian flow as well as a stochastic damped 
   transport process $(W_t)$.  The latter gives a representation for the solutions to the heat equation for
   differential 1-forms with the absolute boundary conditions; it evolves pathwise by the Ricci curvature 
   in the interior,  by the shape operator on the boundary and driven by the boundary local time, and has 
   its normal part erased on the boundary. On the half line this construction selects the Skorohod solution
   (and its derivative with respect to initial points), not the Tanaka solution.    On the half space this agrees
   with the construction of N. Ikeda and S. Watanabe  \cite{Ikeda-Watanabe} by Poisson point  processes. 
 
   This leads also to an approximation for the boundary local time in the topology of uniform convergence; 
   not in the semi-martingale topology, indicating the difficulty for the convergence of solutions of a family
   of random ODE's, with nice coefficients, to the solution of an equation with jumps and driven by the local time.
   In addition, We note that $(W_t)$ is the weak derivative of a family of reflected Brownian motions with respect 
   to the starting point.

%
%

\end{abstract}

\maketitle

\section{Introduction}
\Label{Section1}  

{\bf A. }
Let $M$ be a  smooth connected $d$-dimensional Riemannian manifold with interior $M^o$ and boundary $\partial M$, and $\nu$ the unit inward pointing normal vector on $\partial M$.
Observe that the boundary is a smooth, not necessarily connected, manifold.
 Denote by $\Delta$ the Laplacian on functions. Let $\wedge $ denote the space of differential forms of all degrees, $d$ the exterior differential, $d^*$ its dual operator on $L^2(\wedge)$. Furthermore, let $\Delta^1=-(d^*d+dd^*)$ denote the restriction of the Hodge-Laplace-Beltrami operator, to differential 1-forms.     
 Throughout the article, $(\Omega, \F, \F_t, \P)$ is an underlying filtered probability space satisfying the usual assumptions.

 Let us consider the following initial value and Neumann boundary problem, describing  heat conduction in a perfect insulator,
\begin{equation}
\Label{heat-1}
\begin{split}
&\frac{\partial u}{ \partial t}=\frac12\Delta u \; \hbox{ in $M^o$}, \quad
\f{\partial u}{\partial \nu}|_{\partial M}=0, \quad u(0, \cdot)=f.
\end{split}
\end{equation}

It is well known that any bounded classical solution $u$, from $[0,T]\times M$ to $\R$,  is given by $u(t, x)=\E[ f(Y_t)]$ where $(Y_t)$ is a reflected Brownian motion (RBM) with the initial value $x$. Consider the following problems. (1) How would we simulate a RBM?  (2) Could we select a RBM so it gives a probabilistic representation for the solution to the following heat equation on differential 1-forms, with initial value and Dirichlet boundary?
\begin{equation}\begin{split}\label{form}
\frac{ \partial\phi}{  \partial t}=\frac{1}{ 2}\Delta^1 \phi \;\; \hbox{ in $M^o$}, \quad
\phi(t,\nu)|_{\partial M}=0,  \quad \phi(0, \cdot)=\phi_0.\end{split}
\end{equation}
If $u$ solves (\ref{heat-1}), its exterior differential $du$ solves (\ref{form}). 
In general (\ref{form}) is under determined, hence we impose the extra boundary condition: $d\phi(t,\nu)|_{\partial M}=0$. This is known as the absolute boundary condition which is automatically satisfied by $\phi=du$. Once the Dirichlet boundary condition $\phi(\nu)=0$ is given, it is equivalent to $\nabla_\nu \phi -\phi(\nabla \nu)=0$ on the tangent space of a point of the boundary.   

The existence of a RBM is well known,  we quote the following theorem, from the book of N. Ikeda and S. Watanabe, which treats a more general problem and implies in particular that a non-sticky SDE on a manifold has a unique local solution.
\begin{theorem}[{\cite[Thm 7.2]{Ikeda-Watanabe}}]
Let $D=\R_d^+=\{(x_1, \dots, x^d), x^d\ge 0\}$. For $i=1, \dots, m$, let $\sigma_k^i,b^i: D\to \R$ and $\delta: \partial D\to \R$ be bounded Lipschitz continuous functions. Suppose that for some number $C>0$, $\sum_{k=1}^m (\sigma_k^d)^2 \ge C$ and 
$\delta\ge C$.
Then there exists a pair of adapted stochastic processes $(Y_t, L_t)$,   where $(L_t)$ is a continuous and increasing process  with $L_0=0$ and they satisfy the relation that  $\int_0^t \chi_{Y_s\in \partial D} dL_s=L_t$ for all $t\ge 0$ a.s. and the equations below:
\begin{equation}\label{RBM}
\begin{split}
dY^i_t=&\sum_{k=1}^m \sigma_k^i (Y_t)\chi_{D^0}(Y_t)dB_t^k+b^i(Y_t)\chi_{Y_t\in D^0}dt,\quad 1\le i\le d-1\\
dY_t^d=&\sum_{k=1}^m  \sigma_k^d (Y_t)\chi_{D^0}(Y_t)dB_t^k+b^d(Y_t)\chi_{Y_t\in D^0}dt
+\delta(Y_t)dL_t.\end{split}
\end{equation}
Furthermore uniqueness in law holds.
\end{theorem}
 For the reflected problem in a local chart of a manifold, take $\delta\equiv 1$. 
By this theorem, other reflected stochastic processes, with not necessarily normal reflection, can also be constructed. All reasonable constructions should lead to the same essential quantity: the local time of the Brownian motion on the boundary. This philosophy follows from the uniqueness to the associated sub-martingale problem. 

To solve the heat equation on differential 1-forms with the absolute boundary condition, N. Ikeda and S. Watanabe sought a tangent space valued process $(W_t(v))$ satisfying $\phi(t,v)=\E \phi(0, W_t(v))$.  They constructed the solution for the half space, by the method of orthonormal frames and remarked that by working in charts this leads trivially to the existence of a solution for general manifolds \cite{Ikeda-Watanabe}.  An ansatz for $(W_t)$ is the solution to the random (covariant) ordinary differential equation along the sample paths of a given RBM:
$\frac{D}{ dt}W_t=-\frac{1}{ 2}\Ric^{\#}_{Y_t} (W_t)$. Here  ${D}$ denotes covariant differentiation of $(W_t(\omega))$ along the path $(Y_t(\omega))$, and $\Ric_x^{\#} $ is the
 linear map from $T_xM$ to $T_xM$ such that $\<\Ric^{\sharp}(u), v\>=\Ric(u,v)$ and $\Ric$ denotes the Ricci curvature.
We use $W_t(\omega)$ to stand both for 
a tangent vector $W_t(\omega) \in T_{Y_t(\omega)}M$ with initial value $W_0(\omega)$ and the linear map  $W_t(\omega): T_{Y_0}M \to T_{Y_t(\omega)}M$ (the latter translates a tangent vector at $Y_0$  `by parallel' to a tangent vector in the tangent space at $Y_t(\omega)$). The parallel equation is damped by the Ricci curvature and hence the solutions are called  `damped stochastic parallel transaltion/ transports'.

 This representation holds for a manifold without a boundary and is essentially true for the half space. The involvement of the Ricci curvature follows from the Weitzenb\"ock formula: $\Delta^1\phi={\rm trace} \nabla^2 \phi-\phi(\Ric^{\#} )$. As we will explain, it is also necessary to consider the shape of the boundary and
add the shape operator~${\mcS}: TM\to TM$:
\begin{equation}\label{3-damped}
{D}W_t=-\frac{1}{ 2}\Ric_{Y_t}^{\#} (W_t)\chi_{\{Y_t\in M^o\}}dt-{\mcS}_{Y_t}(W_t)dL_t,
\end{equation}
where $(L_t)$ is the boundary time of the RBM. The shape operator $\mcS_y(w)=\mcS(w)$ is the tangential part of
$-\nabla_w\nu$ for $w\in T_yM$ where $y\in \partial M$ and $
\nu$ is the inward normal vector field. It vanishes identically for the half space  in which case the local time is
not involved. For a general manifold, the shape operator compensates with the variation of $\phi$ in the normal
direction: if $\phi(\nu)\equiv 0$ then
$\nabla_w \phi(\nu)=\phi(\mcS_y(w))$, see (\ref{So}) for detail. There remains a term of the form
$\int_0^t \nabla\phi_{T-s}({\nu(Y_s)},W_s^\nu)ds$ where $W_t^\nu$ is the `normal part' of $W_t$. 
 To solve the absolute boundary problem one simply assume that the normal part of $(W_t)$  vanishes on the boundary.
  Such a process exists and has a number of constructions. The basic idea for our construction is as follows. Denote by $R_t$ the distance of $Y_t$ to the boundary.  We erase the normal part of the solution to (\ref{3-damped}) at the ends of  each excursion, of  size larger or equal to $\epsilon$, of the process $(R_t)$ into the interior $M^o$ and obtain a process $(W_t^\epsilon)$.  As $\epsilon$ approaches zero, $(W_t^\epsilon)$ converges to a process $(W_t)$ with the required properties. The Poisson point process of $(R_t)$ determines the frequency of the projection, the `local time' or the boundary time of the RBM is the local time of $(R_t)$ at $0$  and is essentially the $\delta$-measure on the boundary: $\int_0^t \delta_{{\partial M}} (Y_s) ds$. The point we wish to make is that our SDEs choose the above mentioned damped parallel process.

{\bf B.} 
A RBM on $M$ is a solution to an SDE of the form:
$dY_t=``dx_t"\1_{Y_t\in M^o}+ {\nu\, \delta_{\partial M}(Y_t)dt}$.
We seek an approximations of the form
$dY_t^a=``dx_t" +A^a(Y_t^a) dt$ with the additional properties: (1)  $L_t^a:=\int_0^t \|A^a(Y_s^a)\|ds$ approximates   $ \int_0^t \delta_{\partial M}(Y_s)ds$;
 (2) as $a\to 0$, the solutions~to the equations
\begin{equation}
\label{4-damped}
{D}W_t^a=-\frac{1}{ 2}\Ric_{Y_t^a}^{\#} (W_t^a)dt+\nabla_{W_t^a}A_t^a\,dt
\end{equation}
converge to  $(W_t)$. Here we have extended $\mcS$ to level sets of distance to boundary. Observe that $dL_t^a$ and $dL_t$ have mutually singular supports and it is clear that
not every choice of $``dx_t"$ yields an approximation.  

Our aim is to construct a nice family of  SDEs whose solutions $(Y_t^a)$ are smooth  and stay in the interior for all time (in particular there is no boundary time), and also the solutions together with their  damped parallel translations $(W_t^a)$ select a RBM $(Y_t)$ together with a damped stochastic parallel transport $(W_t)$. Note that, although $(Y_t^a)$ and $(W_t^a)$ are sample continuous, the selected stochastic process $(Y_t, W_t)$  has jumps on the boundary. In addition we observe that $(W_t^a)$ can also be obtained by an appropriate variation in the initial values of the stochastic flow $(Y_t^a)$.

This construction is motivated by the following well known facts for manifold without boundaries: (1) damped parallel translations can be obtained from perturbation to a Brownian system with respect to its  initial data and (2) a formula for the derivative of the heat semigroup follows naturally from differentiating a Brownian flow with respect to its initial value. 
 More precisely if $(dF_t)(v)$ is the derivative of a Brownian flow $(F_t(x))$ with respect to its initial data $x$ in the direction $v\in T_xM$, then $W_t(v)$ is the conditional expectation of  $(dF_t)(v)$ with respect to the filtration of the Brownian flow. From our construction it would be trivial to see the formula $du_t(Y_0) =\frac{1}{t} \E\left( f(Y_t) M_t\right)$ for manifolds with boundary where $(M_t)$ is a suitable local martingale.  
Our choice is rather natural which we will elaborate in Section \ref{outline} where we outline our main results.


{\bf C. Historical Remarks.}
Heat equations on differential forms on manifolds were already considered 60 years ago, \cite[P. E. Conner]{Conner},  they are related to the existence problem for harmonic differential forms and were studied  by  P. Malliavin \cite{Malliavin:74} and H. Airault \cite{Airault:76} using  perturbations to a boundary operator. 
The space of harmonic differential forms are related to the de Rham cohomology groups, the latter is
the quotient of closed smooth differential forms by exact differential forms and is obtained by $L^2$ Hodge decomposition theorems. De Rham cohomologies on compact manifolds are topological invariants; the alternative sum of the dimensions of the de Rham cohomology groups is the Euler characteristic.
See A. M\'eritet \cite{Meritet} for the vanishing of the first de Rham cohomology on manifolds with boundary.   See also I. Shigekawa, N. Ueki, and S. Watanabe \cite{Shigekawa-Naomasa-Watanabe} who gave a probabilistic proof for the Gauss-Bonnet-Chern Theorem, on manifolds with boundaries,  a formula relating the Euler characteristics to the  Gauss curvature.

The study of reflected Brownian motions, in conjunction with heat equation on functions, goes back to the late 50's and early 60's, see
 \cite[F. Dario]{Dario},  \cite[N. Ikeda]{Ikeda61},
 \cite[A. V. Skorohod]{Skorohod61}, and \cite[N. Ikeda, T. Ueno, H. Tanaka and K.  Sat\^o]{Ikeda-Ueno-Tanaka-Sato}.
A weak solution for smooth domains are given in \cite[D. W. Stroock and S. R. S. Varadhan
]{Stroock-Varadhan71}, see also the book \cite[A. Bensoussan and J. L. Lions ]{Bensoussan-Lions82}.  On convex Euclidean domains this was studied by H. Tanaka \cite{Tanaka}.
A comprehensive study can be found in \cite[P. L. Lions and A. S. Sznitman]{Lions-Sznitman84}, see also \cite[S. R. S. Varadhan and R. J. Williams]{Varadhan-Willimans}.
In terms of Dirichlet forms, see 
\cite[Z.-Q. Chen, P.J. Fitzsimmons, and R. Song]{Chen-Fitzsimmons-Song}, and \cite[Burdzy-Chen-Jones]{Burdzy-Chen-Jones}.

 The study of stochastic damped parallel transports on manifolds with boundary began with the upper half plane preluding which, the construction of the reflecting  Brownian motion $(Y_t)$ by a canonical horizontal stochastic differential equation (SDE) on the orthonormal frame bundle (OM) with drift given by the horizontal lift of the `reflecting' vector field, see \cite[N. Ikeda]{Ikeda:60},  \cite[Watanabe]{Watanabe:76}, and  \cite[N. Ikeda and S. Watanabe]{Ikeda-Watanabe:78}. The same SDE on OM defines also the  stochastic parallel transport process $\parals_t(Y)$, which is a sample continuous stochastic process. Their construction of the damped parallel translation is  more difficult. 
 The former reduces to solving a `reflected' stochastic differential equation while the latter is a stochastic process with seemingly arbitrarily induced jumps, which are not at all arbitrary as can be seen  by the argument of the current article. 
 
 Since a manifold with boundary can be transformed to the upper half plane by local charts, c.f. \cite[N. Ikeda and the S. Watanabe]{Ikeda-Watanabe:78}, this implies the local existence of damped parallel transport on a general manifold.  
We note that the Dirichlet Neumann problems were studied in \cite[H. Airault]{Airault:74,Airault:76} using  multiplicative functionals, which was followed up in \cite[E. Hsu]{Hsu} to give a neat treatment for the reflected SDE on the orthonormal frame bundle.  
  In Appendix \ref{section:damped}, we extend and explain Ikeda-Watanabe's construction \cite{Ikeda-Watanabe} for damped parallel transport, the formulation involves an additional curvature term, the shape operator of the normal 
 vector $\nu$.  For the upper half plane, the shape operator vanishes identically.

In case of $M$ having no boundary, the concept of stochastic parallel translation goes back to K. It\^o \cite{Ito-parallel},
 J. Eells and D. Elworthy \cite{Eells-Elworthy},  P. Malliavin \cite{Malliavin:74}, H. Airault \cite{Airault:76} and N. Ikeda and S. Watanabe \cite{Ikeda-Watanabe:78}.  The damped parallel translations along a Brownian motion $(x_t)$ are constructed using stochastic parallel translation. Together with $(x_t)$ it is an ${\mathbb L}(TM,TM)$-valued diffusion process along $(x_t)$ with Markov generator $\frac12\Delta^1$. It is well understood that  $\E df(W_t)$, where $f\in BC^2(M)$, solves the heat equation on differnetial 1-forms. See e.g.  \cite[D. Elworthy]{Elworthy-book},  \cite[J.-M. Bismut]{Bismut:82},  \cite[N. Ikeda and  S. Watanabe]{Ikeda-Watanabe:78}, and   \cite[P.-A. Meyer]{Meyer:81}.  Furthermore if $P_t$ is the heat semi-group,  the process $ dP_{T-t}f(W_t)$ is a local martingale.  See  \cite{Li:thesis} for a systematic study of probabilistic representations of heat semigroups on differential forms using both the damped parallel translation and the derivative flow to a Brownian system. It is also well known that $(W_t)$ can be obtained from the derivative process  of a Brownian flow by conditioning, from which information on the derivative of the heat semigroup on functions can also be obtained. This method was used by M. Arnaudon, B. Driver, K. D. Elworthy,  Y. LeJan, Xue-Mei Li, A. Thalmaier  and Feng-Yu Wang \cite{Li:thesis, Elworthy-Li:94-2,  Elworthy-LeJan-Li96, Elworthy-Li:98,Thalmaier-Wang:98, Arnaudon-Driver-Thalmaier:07}, and many other related works. This plays also a role in the study of the regularity of finely harmonic maps, c.f. \cite[M. Arnaudon, Xue-Mei  Li and A. Thalmaier]{Arnaudon-Li-Thalmaier:99}.   The stochastic Jacobi field point of view is developed in  \cite[M. Arnaudon and A. Thalmaier ]{Arnaudon-Thalmaier:98-1}, leading to estimates for the derivatives of harmonic maps between manifolds and to Liouville-type theorems.

\section{Outline and Main Theorems}
\label{outline}
The paper is organised as follows. We first construct, in \S\ref{section3}, a family of SDEs whose solution flows $(Y_t^a, a>0)$ exist and have a limit $(Y_t)$ which are RBM's.  The convergence of their stochastic parallel translations,  $\parals_s(Y_t^a)$, is given in \S\ref{section4}, while the more difficult  convergence theorem for their damped stochastic parallel translations $(W_t^a)$ is stated in \S\ref{section5} together with a detailed description for the limiting process $(W_t)$ and with a differentiation formula for the heat semi-group with Neunman boundary condition.  The proof for the main theorem of \S\ref{section5} is given in \S\ref{section7}. In \S\ref{section6} we study a variation $(Y_t^a(u), u\in [0,1])$ of $(Y_t^a)$ and discuss the convergence of their derivatives with respect to $u$
to the damped stochastic parallel translation process $(W_t)$. In the appendices we study the half line example, 
explaining notions of convergence of stochastic processes on manifolds, and the existence of Ikeda and Watanabe's stochastic damped parallel translation on manifolds with a boundary. Their proof was explicit for the half plane and uses Poisson point processes.

Close to the boundary, $(Y_t^a, a>0)$ are Brownian motions with drift  $A^a$, where $A^a$ is the gradient of $\ln \tanh\left(\f{R}{a}\right)$ and $R(x)$ is the distance from $x$ to the boundary.  This construction selects a special reflected Brownian motion.  For $M=\R_+$, the solution to the Skorohod problem is selected while the solution given by Tanaka's formula is not, see Appendix \ref{section-half-line} for detail where we illustrate the construction. 
Furthermore they satisfy the following properties.
 
 {\bf Theorem \ref{T2}.}{\it 
 \begin{enumerate}
 \item [(a)]For each $a>0$,  $(Y_t^a)$ remains in $M^0$;
 \item [(b)] As $a\to 0$, $(Y_t^a)$ converges to the RBM process  $(Y_t)$,
    in the topology of uniform convergence in probability (UCP).  
  \end{enumerate}   }
Let $\rho$ denote the Riemannian distance in $M$ and $p\ge 1$. If  $\displaystyle \E \left[\sup_{0\le t\le T} \rho(Y_t^a,Y_t)^p\right]\to~0$ we say that the family of processes 
$(Y_t^a)$ converges to $(Y_t)$ in $\scrS_p([0,T])$. One can also use an embedding of $M$ into an Euclidean space. Contrarily to UCP convergence,
$\scrS_p([0,T])$ convergence in the target space equipped with the Euclidean distance depends on the embedding. 
 However if $M$ is compact,
  convergence in $\scrS_p([0,T])$ in the manifold  is equivalent to convergence in $\scrS_p([0,T])$
  in the target space independently of the choice of the embedding, and it is also equivalent to UCP convergence in the target space. 
 The notation concerning the convergence of stochastic processes on manifolds is introduced in
 Appendix \ref{preliminary} where we also summarise the relevant estimates for stochastic integrals in the $\scrS_p$ and $\SH_p$ norm, as well as relating different notions of convergences.

The main idea behind the construction is to ensure that $R(Y_t^a)$ involves
one single real valued Brownian motion for all parameters $a$, and hence
pathwise analysis is possible. The construction involves a tubular neighbourhood of the boundary in which the product metric  is used for the estimation. 
By the construction, the convergence of $\{Y_t^a, a>0\}$ in the `tangential directions' of the product tubular neighbourhood is trivial.
 Outside of the tubular neighbourhood the drift $A^a(x)$ vanishes, while inside it converges to zero exponentially fast as $a\to 0$
for every $x\in M^o$.

There also exists a family of stochastic processes $\{L_t^a, a>0\}$ which converges to the local time $L_t$ in $\scrS_p([0,T])$, but not in the $\SH_p([0,T])$ topology. The former convergence is a key for the convergence of the damped stochastic parallel translations, while the lack of convergence in the latter topology makes it difficult to follow the standard methods for proving the convergence of solutions of SDEs with a parameter.

 {\bf Corollary \ref {convLat} and Lemma \ref{L1} } {\it Let $M$ be compact. Then
 for any $p>1$ and $T>0$,
  $$\lim_{a\to 0} \|L_t^a-L_t\|_{\scrS_p([0,T])}=0.$$ }
  
Despite of the above mentioned convergence, $dL_t$ and $dL_t^a$ are mutually singular measures. In particular, as $a$ approaches zero, the total variation norm of the their difference converges to a non-zero measure: $ |d(L_t-L^a_t)|\to 2|dL_t|$.
 We also observe that  $|L_t^a|_{\SH_p([0,T])}=|L_t^a|_{\scrS_p[0,T]}$, $|L_t|_{\SH_p[0,T]}=|L_t|_{\scrS_p[0,T]}$ and
  $$ |L_t-L^a_t|_{\SH_p([0,T])}=|L_t|_{\SH_p([0,T])}+|L_t^a|_{\SH_p([0,T])}.$$

Below we state our main theorem, whose proof is the content of Section~\ref{section7}, a preliminary and easier result on the convergence of the stochastic parallel transport is given in \S \ref{section4}.  For
the RBM $(Y_t)$, constructed in Theorem \ref{T2},  there exists a damped stochastic parallel transport process $(W_t)$ along $(Y_t)$ such that the tangential parts of $W_t^{a}$ converge
to the tangential part of $(W_t)$  in $\scrS_p([0,T])$ and the normal parts of $W_t^a$ converge 
 to the normal part of $W_t$ in $L^p([0,T]\times \Omega)$.

If $v\in T_yM$, where $y\in \partial M$, we denote $v^T$ and $v^\nu$ respectively its tangential and normal component. Thus $(W_t^{a,T})$ is the tangential part of $W_t^a$, the latter solves (\ref{4-damped}),
and $$W_t^a=W_t^{a,T}+f_a(t)\nu(Y_t^a) .$$
  \medbreak

{\bf Lemma ~\ref{L4}, Lemma \ref{LpProd}, Theorem ~\ref{T4} and Corollary ~\ref{T4bis}.}
Let $M$ be a compact Riemannian manifold and let $(W_t)$ be the solution to (\ref{3-damped}).
 Let $p\in[1,\infty)$.
 Then 
\begin{enumerate}
\item $\lim_{a\to 0} W_\cdot^{a,T}=W_\cdot^T$ in UCP.
\item Write $W_t=W_t^T+f(t)\nu_{Y_t}$. As $a\to 0$,
$$
 \E\left[\int_0^T|f_a(t)-f(t)|^p\,dt\right]\to 0.$$

\item For any  $C^2$ differential $1$-form  $\phi$ such that $ \phi(\nu)=0$ on boundary,
$$\lim_{a\to 0}\sup_{s\le t} \E\left |  \phi(W_s^a)- \phi(W_s)\right|^p=0.$$
\end{enumerate}

The family of sample continuous stochastic processes $\{(W_t^a), a>0\}$ cannot converge to the stochastic process
$(W_t)$ with jumps, in the topology of uniform convergence in probability,
 nor in the $\scrS_p$ norm nor in the Skorohod topology.
The uniform distance between a continuous path and a path with jumps is at least half of the size of the largest jump. In Lemma \ref{L2} and Corollary \ref{C2}  we see the simple ideas which are key to the proof.  We show that there exists a real continuous process $c_a(u)$ such that $df_a(t)=-c_a(t)f_a(t)\,dt+\hbox{\it other interacting terms}$,
and $-\int_0^t c_a(u) f_a(u)\,du$ is the only term that contributes to creating jumps as $a$ approaches $0$. More precisely, 
if $\alpha_t=\sup_{s\le t} \{s\le t: Y_s \in \partial M\}$ and $t\not \in \SR(\omega)$,
then for all $s,t\in [0,S]$ satisfying $s<t$, \begin{equation*}
\begin{split}
\lim_{a\to 0} e^{-\int_s^tc_a(u)\,du}= 1&\quad \hbox{if}\quad s>\a_t\\
\lim_{a\to 0}  e^{-\int_s^tc_a(u)\,du}=0&\quad \hbox{if}\quad s<\a_t.
\end{split}
\end{equation*}
   This particular construction and the formula for the normal part of $(W_t^a)$ are given in section \ref{section5}. Let us follow Appendix C and explain the jumps in the limiting process $W_t$.
Denote by $\rho(\cdot ,\partial M)$ the distance function to the boundary and set $R_t=\rho(Y_t, \partial M)$.
 Viewed by the distance process $R_t$, the RBM reaches the boundary at a stopping time $\zeta$ where it makes
excursions into the interior. By an equivalent change of probability measures, these excursions are similar to those made by a real valued RBM. The complement to the set of boundary points of $R_t$ are disjoint open intervals.
Let $\SR(\om)$ denote the set of the `right most points' of the excursions of the distance function $R_t$ and $\SR_\epsilon(\om)$ its subset coming from excursions of length at least $\epsilon$. 
 
The stochastic process $W_t\in L(T_{Y_0}M; T_{Y_t}M)$ is  the unique c\`adl\`ag process along the RBM $(Y_t)$ which satisfies the following equation during an excursion to the interior:
\begin{equation}\begin{split}W_t(v)=&v-{\f12}\parals_t\int_0^t \parals_s^{-1}\Ric^{\sharp}_{Y_{s-}}(W_{s-}(v))ds-
\parals_t\int_0^t \parals_s^{-1}\mcS_{Y_{s-}}(W_{s-}(v)) dL_s\\
&-\sum_{s\in \SR(\om)\cap [0,t]} \<W_{s-}(v), \nu({Y_s)}) \>
\parals_{s,t}\nu({Y_s}),\end{split}\end{equation}
where $\parals_{s,t} =\parals_{s,t}(Y_\cdot)$ denotes the parallel transport process along $(Y_\cdot)$ from $T_{Y_s}M$
to $T_{Y_t}M$. We abbreviate $\parals_{0,t}$ to $\parals_t$.
If we remove only the normal part of $W_t$ on excursion intervals equal or exceeding size $\e$, the above stated procedure makes sense and the resulting processes $(W_t^\e)$ has a limit. The above equation is understood in this limiting sense. 
   More precisely, let
 $W_t^\e$ the solution to
\begin{equation}\Label{We-100}
\begin{split}
 DW_t^\e&=-\f12\Ric^\sharp(W_t^\e)\,dt-\mcS(W_t^\e)\,dL_t
-\1_{\{t\in \SR_\e(\om)\}}\langle W_{t-}^\e,\nu_{Y_{t}}\rangle \nu_{Y_t},\\ \quad W_0^\e&={\rm Id}_{T_{Y_0}M}.
\end{split}
\end{equation}
There exists an  adapted right continuous stochastic process $W_t$ such that 
$\lim_{\epsilon \to 0} W_t^\epsilon= W_t$ in UCP, and in $\scrS_p $ where  $p\ge 1$  for $M$ compact. Furthermore for any $\alpha\in {\mathcal T}$, the set of stochastic processes with values in bounded  $1$-forms above $Y_t$ vanishing outside some compact set, c.f. (\ref{T}),
 \begin{align*}&\lim_{\e\to 0}\left(\int_0^{\cdot\wedge \tau_D}  \alpha_s (DW_s^\e) \right)\stackrel{\scrS_2} =\left(\int_0^{\cdot\wedge \tau_D} \alpha_s (DW_s) \right).\end{align*}
   This construction agrees with the derivative flow of the Skorohod reflected Brownian motion in the half line,  extending stochastic damped parallel translation 
   of Ikeda and Watanabe to general manifolds with boundary. This was further elaborated in Section \ref {section6} where we prove damped parallel translation is a weak derivative of the reflected Brownian flow.
See also  \cite[E. Hsu]{Hsu} and  \cite[F. Wang]{Wang:15} for two other constructions for damped parallel translations. In Proposition~\ref{P2} we explain that removing the normal part 
of the damped parallel translation at the beginning of excursions leads to the same object in the limit, using the fact that the beginnings of excursions are left limits of the ends of excursions.  In  {\it Theorem~\ref{T3bis}} we observe that  $(W_t)$ has the  local martingale property when composed with the differential of a solution to the heat equation with Neumann boundary condition, c.f. \cite[N. Ikeda and S. Watanabe]{Ikeda-Watanabe}.
 A stochastic representation for the semigroup on differential one forms with absolute boundary conditions follows.

{\bf Theorem \ref{T6.5} }
 {\it 
Suppose that the tubular neighbourhood of the boundary has positive radius, the curvatures $\Ric^{\#}$ and $\mcS $
are bounded from below. If $\phi(t, \cdot)$ is a solution to the heat equation with absolute boundary conditions, then
$\phi(t, \cdot)=\E \phi(0,W_t(\cdot))$.}

%
%
%
%
%

\medbreak

Finally we consider the problem of obtaining $(W_t^a)$ by varying  the initial value in a family of stochastic flows.
 We have previously obtained a family of processes $(Y_t^a, 0<a\le 1)$ which converges as $a\to 0$ to a RBM in $L^p$. For $a$ fixed we consider a variation 
 $(Y_t^a(u), u\in [0,1])$ satisfying  $Y_0^a(u)=\gamma(u)$ where $\gamma(u)$ is a $C^1$ curve and $\partial_u Y_t^a(u)=W_t^a(u)\dot\gamma(u)$, see~\cite[M. Arnaudon, K.A. Coulibaly and A. Thalmaier]{Arnaudon-Coulibaly-Thalmaier:06} for the construction. \\

{\bf Proposition~\ref{Tightness}, Theorem~\ref{weakflow} and Theorem~\ref{Weakderivative}}
{\it If $M$ is compact, then the family of two parameter stochastic processes $\{(Y_t^a(u))_{t\in [0,T], u\in[0,1]}, \;  a\in (0,1]\}$,  indexed by $a$, is tight in the weak topology. Furthermore for each sequence $\{a_k\}$ of numbers of converging to $0$ and such that  $(Y_t^{a_k}(u))_{t\in [0,T], u\in[0,1]}$ converges weakly to $(Y_t(u))_{t\in [0,T], u\in[0,1]}$,  the following statements hold.
 \begin{itemize}
\item[(1)] for every $u\in[0,1]$, $Y_t(u)$ is a RBM on $M$ with initial value $\gamma(u)$;
\item[(2)] for every $p\in [1,\infty)$ there exists a number $C'(p,T)$ s.t. for all $0\le u_1, u_2\le 1$,
\begin{equation}
\Label{rhopY}
\sup_{0\le t\le T}\E\left[\rho^p(Y_t(u_1), Y_t(u_2))\right]\le C'(p,T)\|\dot\gamma\|_\infty |u_1-u_2|^p.
\end{equation}
\item[(3)]
For all $f\in C^2(M)$ with $df|_{\partial M}(\nu)=0$, 
\begin{equation*}
\E\left[f(Y_t(u_2))-f(Y_t(u_1))-\int_{u_1}^{u_2} df\big(W_t(u)\dot\gamma(u)\big)\,du\right]=0.
\end{equation*} 
\end{itemize}
}
The last identity implies that if $\dot \gamma(u)=v$, then $\f\partial {\partial u} Y_t(\gamma(u))=W_t(u)(v)$. Thus the damped parallel translation can be interpreted as
a derivative of the RBM with respect to the initial point, making connection with the study of   K. Burdzy \cite{Burdzy:09}, S. Anders \cite{Andres:11}, and J.-D. Deuschel and L. Zambotti \cite{Deuschel-Zambotti}
for Euclidean domains. 
It would be interesting to prove that $(W_t)$ can be constructed from the derivative of the flow in strong sense.

 \section{A reflected Brownian Flow}
\Label{section3}
\setcounter{equation}0
 
 We define  a RBM as a solution to the Skorohod problem: it is a sample continuous strong Markov process satisfying the following properties: 
(1) its Markov generator $\L$ restricted to $C_K^2(M^0)$ is  ${\f12}\Delta $;
 (2)   it spends almost all the time in the interior of the manifold (with respect to Lebesgue measure on time set); and
 (3)  its drift on the boundary is colinear with $\nu$,  the inward-pointing unit normal vector field on $\partial M$. 
By inward-pointing we mean the sign of $\nu$ is chosen so that the exponential map, for sufficiently small $t>0$, 
 $\exp(t\nu)$ belongs to $M^0$.  
 If $ \{\sigma_1, \dots\sigma_m\}$ is a family of vector fields spanning the tangent space at each point,
we associate to it a bundle map $\sigma: M\times \R^m \to TM$ given by
$\sigma(x)(e)=\sum_{k=1}^m \sigma_k(x) \<e_k, e\>$ where $\{e_i\}$ is an orthonormal basis of $\R^m$. 
Let $\sigma_0$ be a smooth vector field. Let $(B_t)$ be an $\R^m$-valued Brownian motion.
 A Brownian system is a stochastic differential equations (SDE) of the form $dx_t= \sigma(x_t)\circ dB_t +\sigma_0(x_t)dt$
 with infinitesimal generator $\frac{1}{2}\Delta$ where $\circ$ denotes Stratonovich integration. 
 Such equations always exist and are not unique, for example if $\sigma_i$ are the gradient vector fields obtained  by an isometric embedding of $M$ into $\R^m$, the SDE is a gradient Brownian system.

  \begin{definition}\Label{DRSF}
   A stochastic flow is a Riemannian RBM if it solves the Skorohod problem   \begin{equation}\Label{RSP} dY_t= \sigma (Y_t)\circ dB_t+\sigma_0(Y_t) dt+A(Y_t)dL_t,
   \end{equation}
   where $A$ is a smooth vector field extending the inward normal vector field  and $L_t$ is the local time of $Y_t$ at $\partial M$,  a non-decreasing process satisfying 
   $$\int_0^t \1_{(\partial M)^c} (Y_s) dL_s=0.$$
   \end{definition}
  Since $dL_t$ is supported on the boundary of $M$,  $Y_t$ behaves exactly like a  Brownian motion in the interior.
  \medbreak

The aim of the section is to construct a family of Brownian systems on $M$ with large drift $A^a$ pushing away from the boundary,
approximating a reflected Brownian motion in the topology of uniform convergence in probability. Furthermore the following properties are desired:  the convergence is `uniform' and the limiting process is continuous with respect to the initial data.

\subsection{}
 
Let $\rho$ denote the Riemannian distance function on $M$ and $R$ the distance function to the boundary,
$R(x)=\inf\{ \rho(x,y): y \in \partial M\}$.
 By the tubular neighbourhood theorem, there exists a continuous function $\delta: \partial M \to (0, \infty)$ such that for
$E_0=\{(x, t): x\in \partial  M, \  0\le t<\delta(x)\}$, the map
 \begin{equation} 
  \begin{split}\Label{Phi}
   \Phi : E_0&\to  F_0:=\Phi(E_0)\subset M\\
    (x,t)&\mapsto \exp_x(t\nu_x)
  \end{split}
 \end{equation}
 is a diffeomorphism such that $\Psi(y)=(\pi(y),R(y))$ on $F_0$ where $\pi(y)$ is the boundary point given by
  $\rho(y, \pi(y))=R(y)$ and $\Psi=\Phi^{-1}$.
In other words, on $F_0$, $R\Big(\Phi(x, t)\Big)=t$ and the distance function $R$ is smooth. 
 For $0\le c\le 1$, define $$E_c=\{(x, t): 0\le t<(1-c)\delta(x)\},\quad F_c=\Phi(E_c).$$

  Since $\partial M$ is a Riemannian manifold of its own right, we may represent 
  $\Delta_{\partial M}=\sum_{j=2}^m \bar \sigma_j$ 
 as the sum of squares of vector fields, for example by taking 
  $\{\bar \sigma_j, j=2,\dots, m\}$ to be a gradient system on $\partial M$.

In the tubular neighbourhood around a  relatively compact set $U$ of the boundary, the width $\delta(x)$ can be taken to be a positive constant $3\delta_0$.
Let $R$ be a real valued $1$-Lipschitz smooth function on $M$ which on $F_{1/3}$ agrees with the distance function to the boundary and such that $R\ge \d_0$ on $F_{1/3}^c$. This can be obtained by modifying the distance function to the boundary.
 
 \begin{proposition} \Label{construction} 
 Let $r\ge 1$ and $c\in (0,1)$. Let $\{\bar\s_0, \  \bar\sigma_j, j=2,\dots, m\}$ be a family of $C^r$
 vector fields on $\partial M$ with the property that 
 $$\frac12\sum_{j=2}^m L_{ \bar\sigma_j}L_{\bar\sigma_j}+ L_{\bar\sigma_0}=\f12\Delta_{\partial M}.$$
  
Suppose that $E_0$ has strictly positive radius, i.e. $\displaystyle \inf_{\partial M} \delta>0$.
Let $\d_0$ be a positive number such that   $\delta \ge 3\d_0$. Then there exist a finite number of  $C^r$ vector fields  $\{ \sigma_j, j=0, \dots,N\}$  on $M$ such that  
\begin{enumerate}
\item [(1)] $\frac12\sum_{j=1}^N L_{ \sigma_j}L_{ \sigma_j}+L_{\s_0}=\f12\Delta$,  
\item[(2)]   $\sigma_1=\nabla R$ on $F_{2/3}$,
\item[(3)] For $2\le j\le m$,  $ \sigma_j$ extends $ \bar\sigma_j$,
\item[(4)] For all
 $ p\in \Phi(\partial M\times [0,\d_0])$, for all $j\ge 2$:
$\di  \< \sigma_1,  \sigma_j\>_p=0.$
\end{enumerate}
In the sequel we  denote $A= \sigma_1$. 
\end{proposition}
\begin{proof}
We first assume that there exist a family of  vector fields $\sigma_j, 1\le j\le m$ defined in the tubular neighbourhood $F_0$. Denote $\sigma:M\times \R^m\to TM$ the corresponding bundle map, as indicated earlier.
Let us extend the construction to $M$ and then return to the local construction.

Let $\tilde \sigma_i, m+1\le i\le m+m'$ be a family of vector fields in $M$ such that 
$$\sum_{i=m+1}^{m+m'}L_{\tilde \sigma_i}L_{\tilde \sigma_i}=\Delta.$$ Let $\tilde\sigma$ be the corresponding  bundle map from $M\times \R^{m'}\to TM$ and let $N=m+m'$. We take a real valued function $\beta\in C^\infty(M; [0,1])$ with the property that $\sqrt \beta$ and $\sqrt{1-\beta}$ are  smooth (this is implied by the other assumptions), $\beta|_{F_{2/3}}=1$, and $\beta|_{M \setminus F_{1/3}}=0$.
Let us define a new bundle map $\hat \sigma: \R^{m}\times \R^{m'}\to TM$ as following:
$$\hat\sigma(e_1, e_2)=\sqrt\beta\,\sigma(e_1)+\sqrt{1-\beta}\, \tilde\sigma(e_2),$$
and prove that it is a surjection. Let
 $$(\hat \sigma)^*(v)=(\sqrt\beta \sigma^*(v),\sqrt{1-\beta}(\tilde \sigma)^*(v)).$$
Then $(\hat \sigma)^*: T_xM \to \R^N$ is the right inverse to $\hat \sigma$. Indeed  $$\hat \sigma \hat \sigma^*=\beta \sigma\sigma^*+(1-\beta)\tilde\sigma (\tilde \sigma)^*=id,$$
and $\hat\sigma$ induces the Laplace-Beltrami operator. It is the desired map.

Let us now construct $\s$ on $F_0$. If $y\in m$, let $\pi(y)$ denote the projection of $y$ to $M$,
$\gamma_y$ the geodesic from $\pi(y)$ to $y$,  and $\paral(\gamma_y)$
the parallel transport along $\gamma_y$ which is a linear map from $T_{\pi(y)}M$ to $T_yM$.
 For $j\not =1$, along a geodesic normal to the boundary, we may extend $ \bar\sigma_j$ by parallel transport 
 along the geodesic in the normal vector direction:
 $$ \sigma_j(y)=\paral (\gamma_y) \bar\sigma_j(\pi(y)).$$
In other words it is constant along the geodesic. It is clear that $\< \sigma_j(x), \sigma_1(x)\>=0$ for all $j>1$ and $x\in F_0$.

From the assumption, $\frac12\sum_{k=2}^m\nabla^\partial_{ \bar\sigma _k} \bar\sigma_k+\bar\sigma_0=0$ on  $\partial M$ and
$\sum_{j=2}^m\bar\s_j(x)\bar\s_j^\flat(x)={\rm Id}_{T_x\partial M}$. Here $\bar\s_j^\flat(x)$
denotes the $1$-form $\langle \s_j(x),\cdot\rangle$.  Let us prove that for all
$x\in F_0$, $$\sum_{j=1}^m\s_j(x)\s_j^\flat(x)={\rm Id}_{T_x M}$$
where $\s_j^\flat(x)$ is defined in a similar way.
 Since the vectors $\{\s_k(x)\}$ generate $T_xM$, it is sufficient to prove that for all $k=1,\ldots, m$, 
$$
\left(\sum_{j=1}^m\s_j(x)\s_j^\flat(x)\right)(\s_k(x))=\s_k(x).
$$
For $k=1$:
\begin{align*}
 \left(\sum_{j=1}^m\s_j(x)\s_j^\flat(x)\right)(\s_1(x))&=\s_1(x)\langle \s_1(x),\s_1(x)\rangle=\s_1(x).
\end{align*}
For $k\ge 2$: 
\begin{align*}
 \left(\sum_{j=1}^m\s_j(x)\s_j^\flat(x)\right)(\s_k(x))&=\sum_{j=2}^m\s_j(x)\langle\s_j(x),\s_k(x)\rangle\\
 &=\sum_{j=2}^m\paral (\gamma_x) \bar\sigma_j(\pi(x))\langle \bar \s_j(\pi(x)),\bar\s_k(\pi(x))\rangle\\
 &=\paral (\gamma_x) \bar\sigma_k(\pi(x))=\s_k(x).
\end{align*}
To complete the construction, we take the drift vector with the following property:
$$
\s_0(x)=-\f12\sum_{j=1}^m\nabla_{\s_j(x)}\s_j=-\f12\sum_{j=2}^m\nabla_{\s_j(x)}\s_j, \qquad x\in F_0.
$$
This completes the proof.
\end{proof}

The vector field $A= \sigma_1$, constructed in Proposition \ref{construction} extends the unit inward normal vector field, defined
on $\partial M$, and coincides with $\nabla R$ on $F_{2/3}$.  
Off the cut locus of the $R$,  $\nabla R$ exists almost everywhere. For the Skorohod problem, 
we will only need the information of $A$ on $F:=F_{2/3}$ and in particular we 
   do not need to worry the effect of the cut locus.

 Next we take a family of additional drift vector fields converging to $0$ in the interior of $M$ and to the local time
 on the boundary. We divide the manifold $M$ into three regions: inner tubular neighbourhood, the middle region
  and the outer region. 
The inner region, a subset of $E_0$ with the product metric  is quasi isometric to its image, 
i.e. there is a  constant $C>0$ such that for all $x,y\in \partial M$, for all $s,t\in [0,\delta_0]$, denoting $\bar\rho$ the distance in $\partial M$,
$$\frac {1}{C}\left( \bar\rho(x,y)+|s-r|\right)\le \rho\left(\Phi(x, s), \Phi(y,r)\right)\le C\left (\bar\rho(x,y)+|s-r|\right).$$
Outside of the tubular neighbourhood the drifts will be chosen to be uniformly bounded and to converge to zero
uniformly. If $x_n$ is a sequence of points in the outer region with limit $x_0$, we need to assume that the 
solution with initial value $x_n$ converges in some sense.
In the first region we have  convergence in probability and in the second we will need a control on the  
rate of convergence that induces the property of a flow.

By the convergence of the manifold valued stochastic process $(Y_t^a) $ to $(Y_t)$
we mean that $H(Y_t^a)$ converges to $H(Y_t)$ where  $H:M\to \R^k$ is an embedding,
with the same notion of convergence. 
We recall that $(Y_t^a)$ converges to $Y_t$ in UCP implies that for any $\epsilon>0$,  the explosion times $\xi^a$ of $Y_t^a$ 
and its exit times from relatively compact sets, for sufficiently small $a$, are bounded below by the corresponding ones for $Y_t=Y_t^0$
minus $\epsilon$ (in other words $\liminf_{a\to 0}\xi^a\ge \xi^0$). 
Also, if we assume sufficient growth control on the curvatures and the shape of the tubular neighbourhood, the convergence will be in $\scrS_p$. We only discuss this aspect for a compact manifold.

Let $R: M\to \R$ be a smooth function such that $R|F$ is the distance to the boundary,
and $R|_{F^c}\ge \tilde \delta$ for some number $\tilde \delta$, and
$R(x)$ is a constant on the complement of $F^c$.

 \begin{theorem}\Label{T2} 
Let $a$ be a positive number. For $\sigma_k, A$ satisfying  properties stated in Proposition \ref{construction},
  let $Y_t^a$ and $Y_t$ denote respectively the maximal
 solution to the equations, with initial value $x$,
 \begin{eqnarray}\label{Yta}
 dY_t^a&=& \sum_{k=1}^m\sigma_k (Y_t^a)\circ dB_t^k
 +\sigma_0(Y_t^a)dt +\nabla \ln \left( \tanh \left(\frac{R(Y_t^a) }{ a}\right)\right)\,dt,\\
 dY_t&=&\sum_{k=1}^m\ \sigma_k (Y_t)\circ dB_t^k+\sigma_0(Y_t)dt+ A(Y_t)dL_t.
 \label{Yt}
 \end{eqnarray}
  Suppose that 
 $\rho(Y_0^a, Y_0)$ converges to $0$ in probability.
 \begin{enumerate}
 \item Then  $\lim_{a\to 0} Y^a=Y$, in the topology of uniform convergence in probability. 
 \item If $M$ is compact, then the SDEs do not explode and
 for all $p\in[1,\infty)$ and for all $T>0$,  $Y^a$ converges to $Y$ in $\scrS_p([0,T])$, i.e.
 $$\lim_{a\to 0} \E \sup_{0\le s\le T} \rho(Y_s^a, Y_s)^p=0.$$ \end{enumerate}
  \end{theorem}
\begin{proof} 

Since UCP convergence is local and is implied by local convergence in $\scrS_p$, (1) is a consequence of (2). 
See Corollary \ref{C3}. So we assume that $M$ is compact and choose a constant $\d_0>0$
such that the function $\delta$ is bounded below by $ 3\d_0$. Then we replace $\delta$ by $3\delta_0$
in the definition of $E_0$, $F_0$, $E_c$, $F_c$.

We define  $h^a(x)=\ln \Big( \tanh \left(\frac{R (x)}{ a}\right)\Big)$ and
 $$A^a(x)=\nabla h^a(x)=\frac{2\nabla R(x)}{ a \sinh \left(\frac{2 R(x) }{ a}\right)}.$$
 This is an approximation for a vector field that vanishes on $M^0$ and 
 exerts an `infinity' force in the direction of $\nabla R=A$ on the boundary.

  Let $R_t^a=R(Y_t^a)$.
 Then $$R_t^a=R_0^a+ \sum_k\int_0^t \langle d R,\sigma_k(Y_s^a)\rangle dB_s^k
 +\frac{1}{2}\int_0^t \Delta R(Y_s^a) ds + \int_0^t \frac{2 }{ a \sinh (\frac{2 R_s^a }{ a})}ds. $$
Let us denote by $\beta_t^a$ the stochastic term:
$$R_t^a=R_0^a+\beta_t^a+\frac{1}{2}\int_0^t \Delta R(Y_s^a) ds + \int_0^t \frac{2 }{ a \sinh (\frac{2 R_s^a }{ a})}ds. $$
For $Y_t^a\in F_{2/3}$ the tubular neighbourhood of $\partial M$, we have by Proposition~\ref{construction} (2) that $d\beta_t^a=dB_t^1$ is
independent of $a$ and of $Y_t^a$, which will be crucial for the sequel:
\begin{equation}
 \Label{ERat}
 dR_t^a=dB_t^1+\frac{1}{2}\Delta R(Y_t^a) dt +  \frac{2 }{ a \sinh (\frac{2 R_t^a }{ a})}dt.
\end{equation}


  Since we assumed that $M$ is compact,  $|\Delta R|$ is bounded, so the drift is essentially 
    $ \di\frac{2 }{ a \sinh (\frac{2 R_s^a }{ a})}$ 
  and $R_t^a$ never touches the boundary and the equation is well defined.

 Recall that $\pi$ is the map that sends a point $x\in M$ to the nearest point on $\partial M$,
it is defined on $F_0$.   The tubular neighbourhood map $\Psi :  F_0\to E_0 $ 
splits into two parts,  $\Psi(x)=(\pi(x), R (x))$. Since $\Psi$ is a diffeomorphism onto its image, on $\{Y_t\in F_0\}$,
the processes $Y_t^a$ converges to $Y_t$ in the Riemannian metric on $M$ if and only if 
$\Psi(Y_t^a)$ converges to $\Psi(Y_t)$ in the product metric of 
$\partial M\times [0, \d]$.  

On any subset of $M$ not intersecting the tubular neighbourhood that is distance $c\delta$ from $\partial M$
for some $c<1$, the functions $|\nabla A^a|$ are uniformly bounded in $a$
and converge to zero as $a\to 0$. The local time does not charge any real time if $Y_t$ is not on the boundary. For a $C^3$ embedding $\Phi: M\to \R$,
\begin{align*}
\Phi (Y_t^a)-\Phi (Y_t)
=&\Phi (Y_0^a)-\Phi (Y_0)+\int_0^t\left\langle (\s^\ast\nabla \Phi )(Y_s^a)
-(\s^\ast\nabla \Phi )(Y_s),dB_s\right\rangle\\
&+\f12\int_0^t\left(\D \Phi (Y_s^a)-\D \Phi (Y_s)\right)\,ds.
\end{align*}
By standard estimates,  if $ Y_0^a \to Y_0$ in probability,  the processes $\Phi(Y_t^a)$  started outside the closed tubular set $F_{\frac{1}{3}}$ and stopped at the first entrance time of  $F_{\f23}$  converge to $\Phi(Y_t)$ in UCP. In particular this holds for
isometric embeddings and since the intrinsic Riemannian distance is controlled by the extrinsic distance function, we see that the stochastic process $\rho(Y_t^a, Y_t)$ converges in UCP.

 Splitting in a proper way the times, for the UCP topology it is enough to prove that the processes $(Y_t^a)$ started  inside the open
 set $F_{\frac23}$ and stopped at exiting $F_{\f13}$ converge to  $(Y_t)$ whenever $Y_0^a\to Y_0$.
 
 So we assume that $Y_0^a$ and $Y_0$ belong to $F_{2/3}$ and
 $\rho(Y_0^a, Y_0)$ converges to $0$ in probability. We let 
 $$
 \tau=\inf\{t\ge 0,\ R(Y_t)=2\d_0\},\quad \tau^a=\inf\{t\ge 0,\ R(Y_t^a)=2\d_0\}.
 $$
 We first prove that for all $T>0$,
 \begin{equation}\Label{E9}
 \forall T>0,\ \ \lim_{a\to 0}\E\left[\sup_{t\le \tau^a\wedge \tau \wedge T}\rho^2(Y_t^a,Y_t)\right]=0.
 \end{equation}
Notice  if \eqref{E9} holds,  $\sup_{s\le t} R(Y_s) <2 \delta_0$ implies that $\sup_{s\le t} R(Y_s^a)<2 \delta_0$ for sufficiently small $a$, consequently,
  \begin{equation}\Label{E8}
 \lim_{a\to 0}\E\left[\sup_{t\le \tau\wedge T}\rho^2(Y_t^a,Y_t)\right]=0.
 \end{equation}
 This in turn shows that
\begin{equation}\Label{E7}
 \liminf_{a\to\infty}\tau^a\wedge T\ge \tau \wedge T
 \end{equation}
 and the convergence of $Y_t^a$ to $Y_t$ in the UCP topology follows.
 
 Let $R_t^a=R(Y_t^a)$ and $R_t=R(Y_t)$. Denote $\bar \rho$ the Riemannian distance on $\partial M$.
 Using the tubular neighbourhood map, proving \eqref{E9} will be equivalent to prove the following two limits:

 \begin{equation}
 \Label{E5-E6}
  \lim_{a\to 0}\E\left[\sup_{t\le \tau^a\wedge\tau\wedge T}(R_t^a-R_t)^2\right]=0, \quad
  \lim_{a\to 0}\E\left[\sup_{t\le \tau^a\wedge \tau\wedge T}(\bar \rho)^2(\pi(Y_t^a),\pi(Y_t))\right]=0.
 \end{equation}
 For $t\le \tau^a\wedge \tau$ we have by~\eqref{ERat} and \eqref{RSP},
 \begin{equation}\Label{E9.1} R_t^a-R_t=R_0^a-R_0+\int_0^t\f{ds}{a\sinh\left(\f{2R_s^a}{a}\right)}-L_t+\f12\int_0^t\left(\D R(Y_s^a)-\D R(Y_s)\right)\,ds.
 \end{equation}
We remark that in the above equation there is no martingale part. 
 Let $\e>0$, and 
 $$\tilde L_t^a=\int_0^t\f{ds}{a\sinh\left(\f{2R_s^a}{a}\right)}-L_t.$$
We apply It\^o-Tanaka  formula to the convex function
 $\max(y,\e)$ to obtain:
 \begin{align*}
 \e\vee|R_t^a-R_t|=& |R_0^a-R_0|\vee\e +\int_0^t \1_{\{R_s^a-R_s >\e\}}\;d\tilde L_s^a
 -\int_0^t \1_{\{R_s^a-R_s <-\e\}}\;d\tilde L_s^a\\
 &+{\frac12}\int_0^t  \1_{\{|R_s^a-R_s| >\e\}}\left(\D R(Y_s^a)-\D R(Y_s)\right)\,ds.
\end{align*}
It is vital to remark that $L_s >0$ if and only if $R_s^a-R_s=R_s^a$. Also $R_s^a-R_s <-\e$,
if only if $R_s\not =0$, and so $- d\tilde L_s^a$ is a negative measure. We may ignore the third term on the right hand side of the identity.
For each $\a>0$ and $\e>0$, there exists a number $a(\e,\a)>0$ such that for all $a\le a(\e,\a)$ and $r\ge \e$, 
$\f1{a\sinh\left(\f{2r}{a}\right)}<\a$. Hence,
\begin{align*}
&\int_0^t \chi_{ \{R_s^a-R_s>\e, R_s=0\}}d\tilde L_s^a\\
&\le  \int_0^t \f1{a\sinh\left(\f{2(R_s^a-R_s)}{a}\right)}\chi_{ \{R_s^a-R_s>\e, R_s=0\}}ds-L_t\le
 \a t,\\
&\int_0^t \chi_{ \{R_s^a-R_s>\e, R_s \not =0\}}d\tilde L_s^a
= \int_0^t \f1{a\sinh\left(\f{2R_s^a}{a}\right)}\chi_{ \{R_s^a-R_s>\e, R_s \not =0\}}ds\\&\le 
\int_0^t \f1{a\sinh\left(\f{2\e}{a}\right)}ds\le
\a t,
\end{align*}
It follows that
\begin{equation}
\Label{E11}
\e\vee|R_t^a-R_t|\le |R_0^a-R_0| +\e+2\a t+\f12 \|\nabla\D R\|_{L^\infty(F_0)}\int_0^t\sup_{r\le s}\rho(Y_r^a,Y_r)\,ds.
\end{equation}
So 
\begin{equation}
\begin{split}\Label{E12}
\E\left[\sup_{t\le \tau\wedge\tau^a\wedge T}\left(R_t^a-R_t\right)^2 \right]\le& 4\E\left[(R_0^a-R_0)^2\right]+4\e^2+8\a^2t^2+\\&
2\|\nabla\D R\|_{L^\infty(F_0)}\int_0^T\E\left[\sup_{s\le \tau\wedge\tau^a\wedge t}\rho^2(Y_s^a,Y_s)\right]\,dt.
\end{split}
\end{equation}
Before continuing with the estimate above, we estimate $\bar \rho( \pi(Y_t^a,\pi(Y_t))$.
The distance function $\bar \rho$ is not smooth on $\partial M\times \partial M$. So we will consider an isometric embedding $\imath  :\partial M\to \R^{m'}$ (in fact since $\partial M$ is compact any embedding would do)
 and instead of proving the second limit in ~\eqref{E5-E6} we will prove that 
 \begin{equation}
 \Label{E5bis}
 \lim_{a\to 0}\E\left[\sup_{t\le \tau^a\wedge \tau\wedge T}\left(\imath (\pi(Y_t^a))-\imath (\pi(Y_t))\right)^2\right]=0
 \end{equation}
We extend $\imath $ to $F_0$ to obtain  $\tilde \imath (y)=(\imath \circ \pi)(y)$, then
\begin{equation}
\begin{split}\Label{E10}
\tilde \imath (Y_t^a)-\tilde \imath (Y_t)
&=\tilde \imath (Y_0^a)-\tilde \imath (Y_0)+\int_0^t\langle \s^\ast\nabla \tilde \imath (Y_s^a)-\s^\ast\nabla \tilde \imath (Y_s),dB_s\rangle\\
&+\f12\int_0^t\left(\D \tilde \imath (Y_s^a)-\D \tilde \imath (Y_s)\right)\,ds+\int_0^t d\tilde \imath (A^a(y_s^a))ds
-\int_0^t  d\tilde \imath  (A(y_s)) dL_s.
\end{split}
\end{equation}
Since $d\pi (A)=0$ and $d\pi(A^a)=0$, the last two terms vanish. By standard calculation, 
\begin{equation}
\begin{split}\Label{E13}
&\E\left[\sup_{t\le T\wedge\tau\wedge \tau^a}\left\|\tilde \imath (Y_t^a)-\tilde \imath (Y_t)\right\|^2\right]\\&
\le 4\E\left[\|\tilde \imath (Y_0^a)-\tilde \imath (Y_0)\|^2\right]+16\left\|\nabla\s^\ast\nabla \tilde \imath \right\|^2_{L^\infty(F_0)}\int_0^T
\E\left[\sup_{s\le t\wedge\tau\wedge\tau^a}\rho^2(Y_s^a,Y_s)\right]\,dt\\
&+2\|\nabla\D \tilde \imath \|_{L^\infty(F_0)}\int_0^T\E\left[\sup_{s\le \tau\wedge\tau^a\wedge t}\rho^2(Y_s^a,Y_s)\right]\,dt.
\end{split}
\end{equation}
Since $\partial M$ is compact, $F_0$ is compact. The quantities $\nabla \sigma^*$ and $\nabla \tilde i=
\nabla \pi(\nabla i)$ are bounded. Similarly $\|\nabla\D \tilde \imath \|_{L^\infty(F_0)}$ is finite.
For $x\in F_0$, set 
\begin{equation}
\Label{Hx}
H(x)= (\tilde \imath (x), R(x))\in \R^{m'+1}.
\end{equation}
 Let $C_H >0$ be a constant such that for all $x,x'\in F_0$, 
\begin{equation}
\Label{E14}
\f1{C_H}\|H(x)-H(x')\|\le\rho(x,x')\le C_H\|H(x)-H(x')\|.
\end{equation}
Define 
\begin{equation}
\Label{E16}
C=\left(16\left\|\nabla\s^\ast\nabla \tilde \imath \right\|^2_{L^\infty(F_0)}+2\|\nabla\D H \|_{L^\infty(F_0)}\right)(C_H)^2.
\end{equation}
From (\ref{E13}),   using Gronwall lemma we obtain
that if $a<a(\e)$,
 \begin{equation}
 \begin{split}\Label{E15}
 &\E\left[\sup_{s\le T\wedge \tau\wedge \tau^a}\left\|H(Y_s^a)-H(Y_s)\right\|^2 \right]
 \le 4\left(\left\|H(Y_0^a)-H(Y_0)\right\|^2+\e^2+ \a^2T^2 \right)e^{CT}.
 \end{split}
 \end{equation}
 Since $\e$ and $\a$ can be chosen as small as we like and $C$ is independent of $\e,\ \a,\ a$, 
 $Y_0^a\to Y_0$ and $H$ is bounded, we obtain that 
 \begin{equation}
 \Label{E16.1}
 \lim_{a\to 0}\E\left[\sup_{s\le T\wedge \tau\wedge \tau^a}\left\|H(Y_s^a)-H(Y_s)\right\|^2 \right]= 0.
 \end{equation}
Together with (\ref{E12}), we see that
 \begin{equation}
 \Label{E17}
\lim_{a\to 0} \E\left[\sup_{s\le T\wedge \tau\wedge \tau^a}\rho^2(Y_s^a,Y_s) \right]=0.
 \end{equation}
This implies that $\tau^a\wedge \tau\to \tau$ almost surely, and since the distance is bounded, 
 \begin{equation}
 \Label{E18}
\lim_{a\to 0} \E\left[\sup_{s\le T\wedge \tau}\rho^2(Y_s^a,Y_s) \right]= 0.
 \end{equation}
 This completes the proof for the convergence of $Y^a$ to $Y$ in UCP,  and also in $\scrS_p$ for compact manifold $M$.
\end{proof}
It would be interesting to use the method in \cite{Li-flow} to study whether there exists a global smooth solution flow to the SDEs.
It is also worth noting that if $\tau^U(Y)$ (resp. $\tau^U(Y^a)$) is the exit time of $Y_\cdot$ (resp. $Y_\cdot^a$) from a relatively compact open set $U$, then
$$
\liminf_{a\to 0}\tau^U(Y^a)\ge \tau^U(Y).
$$
\begin{corollary}
\Label{C1}
Let $S_1$, $S_2$ be stopping times such that $S_1<S_2$ and $a_0$ a positive constant.  Suppose that 
$Y_t^a\in F_0$ for $a\in (0,a_0]$ and $t\in [S_1,S_2]$. 
Then on the interval  $[S_1,S_2]$, $\lim_{a\to 0} \pi(Y_\cdot^a)=\pi(Y_\cdot)$. The convergence is in the semi-martingale topology.
If moreover $M$ is compact and $S_2$ is bounded, then the convergence holds in $\SH_p$ for all $p\in [1,\infty)$.
\end{corollary}
\begin{proof}
Since the drifts $A$ and $A^a$ belong to the kernel of the differential $T\pi$, we
obtain with It\^o formula the following equations:
 \begin{align*}d(\pi(Y_t^a))&= \sum_{k=1}^mT\pi \sigma_k (Y_t^a) \circ dB_t^k + T\pi \circ\sigma_0 (Y_t^a) \,dt\\
 d(\pi(Y_t))&=\sum_{k=1}^m\ T\pi\circ\sigma_k (Y_t) \circ dB_t^k+T\pi\circ \sigma_0 (Y_t^a) \, dt. \end{align*}
Since $\frac{1}{2}\sum_{k=1}^m \nabla_{\sigma_k} \sigma_k+\sigma_0=0$,  we only need to be concerned with the
following term from the It\^o correction:
$\frac{1}{2}\sum_{k=1}^m \nabla T\pi(\cdot) (\sigma_k, \sigma_k)$. 
 By Theorem~\ref{T2},  both $T\pi \circ\sigma_k (Y_t^a)$ and $\nabla T\pi(Y_t^a)
 (\s_k,\s_k)$ converge in the UCP topology, and they are locally uniformly bounded.
 The limits are respectively $T\pi \sigma_k (Y_t)$ and $ \nabla T\pi(Y_t) (\s_k,\s_k)$.
 By Theorem~2 in~\cite[M. Emery]{Emery:79}, see \cite[M. Arnaudon and A. Thalmaier]{Arnaudon-Thalmaier:98} for the manifold case, $\pi(Y_\cdot^a)$ converges to $\pi(Y_\cdot)$ in the semi-martingale topology.
\end{proof}
Define $L_t^a=\int_0^t\f{ds}{a\sinh\left(\f{2R_s^a}{a}\right)}$. Then in the tubular neighbourhood, 
$$A^a(Y_t^a) =\nabla R(Y_t^a)  \frac{d} {dt} L_t^a.$$

\begin{corollary}\Label{convLat}
Suppose that $M$ is compact. Then for all $p\ge 1$ and $T>0$, 
$$\lim_{a\to 0}\E\left( \sup_{s\le T} |L_s^a-L_s|^p\right)=0.$$
Moreover, letting $L^0=L$,  for all $\lambda>0$, there exists $C(T,\lambda)$ such that for all $a\in [0,1]$, 
\begin{equation}
\Label{expLat}
\E\left[e^{\lambda L_T^a}\right]\le C(T,\lambda).
\end{equation}
\end{corollary}
\begin{proof}
Firstly we take $Y_0^a, Y_0$ in $F_{2/3}$, the $\f 2 3$ tubular neighbourhood of the boundary.
Let $\tau=\inf \{ R(Y_t)=2 \delta_0\}$ and $\tau^a=\inf \{ R(Y_t^a)=2 \delta_0\}$ be respectively the first exit times of $Y$ and $Y^a$ from $\{x:R(x)<2\delta_0\}\subset F_{\f 23}$.
On $\{t<\tau^a\wedge \tau\}$ we have  \eqref{E9.1}:
\begin{equation}\Label{E9.1bis} L_t^a-L_t=-R_0^a+R_0+R_t^a-R_t-\f12\int_0^t\left(\D R(Y_s^a)-\D R(Y_s)\right)\,ds.
 \end{equation}
By the convergence of $Y^a$ to $Y$ in $\scrS_p([0,T])$,
$$\E \sup_{t<\tau^a\wedge \tau} |L_t^a-L_t|^p<\infty.$$
Outside of the $2/3$ tubular neighbourhood $F_{\frac {2}{3}}$,
$\frac{1}{a \sinh( \frac{2 R(x)}  {a})}$ converges to $0$ uniformly in $x$ and $L_t$ vanishes. 
Note that $\lim_{a\to 0} \frac{1}{a \sinh( \frac{2 r}  {a})}=0$ for any $r>0$. The required convergence result follows.

To prove~\eqref{expLat} we write for $a\in[0,1]$
$$
L_t^a=R_t^a-R_0^a +\int_0^t \alpha_s^a \,dZ_s^a+\int_0^t\beta^a_s\,ds 
$$
where for all $a$, $Z_t^a$ is a real valued Brownian motion, and $|R_t^a-R_0^a|$, $\alpha_t^a$ and $\beta_t^a$ are uniformly bounded independently of $a$. The result immediately follows.
 \end{proof}

\section{Convergence of the parallel transports}
\Label{section4}
Let $(Y_t)$ and $(Y_t^a)$ be respectively the solutions of (\ref{Yt}) and (\ref{Yta}).
The parallel transport along $(Y_t)$ and $(Y_t^a)$ are respectively 
the solution to the canonical horizontal stochastic differential equations
on the orthonormal frame bundle with drift the horizontal lift of  the drift vector fields $A$  and $A^a$ respectively.

Denote by $\parals_t^a$ the parallel transport along $Y_t^a$, $\parals_t$ the parallel transport along $Y_t$.
Recall that $\sigma^*_y: T_yM\to \R^m$ is the right inverse to $\sigma_y$.
Take $v^a\in T_{Y_0^a}M$ and $v\in T_{Y_0}M$ with the property that $\lim_{a\to 0} \sigma^*(Y_t^a)(v^a) = \sigma^*(Y_t^a)(v)$. Let  $U$ be  a continuous vector field. Then
$$\<\parals_t^a v^a, U(Y_t^a)\>
=\<\sigma^*(Y_t^a) (\parals_t^av^a), \sigma^*(Y_t^a)(U(Y_t))\>.$$
Since $Y_t^a\to Y_t$ as $a\to 0$, so does $\sigma^*(Y_t^a)U(Y_t^a)$ to $\sigma^*(Y_t)U(Y_t)$.
We prove below that $\sigma^*(Y_t^a) (\parals_t^a v^a)\to \sigma^*(Y_t) (\parals_tv)$.

\begin{proposition}\Label{cvpt}
Let $v^a\in T_{Y_0^a}M$ and $v\in T_{Y_0}M$. Suppose that 
$\sigma_{Y_0^a}^*v^a$ converges to $\sigma_{Y_0}^*v$ in probability as $a\to 0$.
Then $\lim_{a\to 0} \s^{\ast}(Y_t^a)(\parals_t^a v^a)=\s^{\ast}(Y_t)(\parals_t v)$, with convergence in the semi-martingale topology.
Also  $\lim_{a\to 0} \parals_t^a v^a\stackrel{UCP}{\to} \parals_tv$. 

If $M$ is compact, the convergences hold respectively in $\SH_p([0,T])$ and $\scrS_p([0,T])$ for all $T>0$ and $p\ge 1$.  
\end{proposition}
\begin{proof}
Since
 $\nabla_{\circ dY_t^a} (\parals_t^a  v^a)=0$ by the definition, the stochastic differential of $\parals_t^a  v^a$ satisfies the following equation:
\begin{align*}
 d(\s^\ast(Y_t^a) \parals_t^a  v^a)&= \nabla_{\circ dY_t^a}\s^\ast(Y_t^a)\parals_t^a v^a
 =\sum_{j=2}^m (\nabla_{\s_j(Y_t^a)}\s^\ast)\parals_t^a  v^a \circ dB_t^j
\end{align*}
where we used the fact that $\nabla_\nu\s^\ast=0$. Upon converting the Stratonovich integral on the right hand 
side we see that
\begin{align*}
 d(\s^\ast(Y_t^a)\parals_t^a  v^a)
 =&\sum_{j=2}^m(\nabla_{\s_j( Y_t^a)}\s^\ast) \parals_t^a  v^adB_t^j
 +\f12\sum_{j=2}^m d (\nabla_{\s_j( Y_t^a)}\s^\ast) \parals_t^a  v^adB_t^j\\
 =&\sum_{j=2}^m( \nabla_{\s_j( Y_t^a)}\s^\ast) \parals_t^a  v^adB_t^j
 +\f12\sum_{j=2}^m ( \nabla_{\s_j( Y_t^a)}\nabla_{\s_j( Y_t^a)}\s^\ast) \parals_t^a  v^adt.
\end{align*}
This can be rewritten as 
\begin{align*}
  d(\s^\ast(Y_t^a)\parals_t^a  v^a)
 =& \sum_{j=2}^m \left(   \nabla_{\s_j( Y_t^a)}\s^\ast\right) \s( Y_t^a) \left(\s^\ast( Y_t^a) \parals_t^a  v^a\right) dB_t^j\\
 &+\f12\sum_{j=2}^m \left(  \nabla_{\s_j( Y_t^a)}\nabla_{\s_j( Y_t^a)}\s^\ast\right)\s ( Y_t^a) \left( \s^\ast ( Y_t^a)\parals_t^a v^a\right)  dt.
\end{align*}
Since the coefficients of the SDE converges as $a\to 0$ uniformly in probability,
we get that $\s^\ast (Y_t^a) \parals_t^a  v^a$ converges to  $\s^\ast(Y_t) \parals_t v$ in semi-martingale topology, see Theorem~2 in~\cite[M. Emery]{Emery:79}.


Finally, since the linear maps $\s(Y_t^a): \R^m \to T_{Y_t^a}M$ 
converge to $ \s(Y_t): \R^m \to T_{Y_t}M$ in UCP topology and  $ \parals_t^a =\s(Y_t^a)\s^\ast(Y_t^a)\parals_t^a $, we
see that $\parals_t^a $ converges to $\parals_t$ in the same topology.

\end{proof}

\section{Convergence of the Damped Parallel Translations}
\Label{section5}
Let $(Y_t)$ be the reflected Brownian motion and $(Y_t^a)$ the approximate reflected Brownian motions,
constructed by (\ref{Yt}) and (\ref{Yta}) respectively. Let  $A^a=\nabla \ln \tanh(\frac{R}{a})$.
Denote $(W_t^a)$ the damped parallel translations $(Y_t^a)$, solving the equation
\begin{equation}
\label{Wta}\frac{DW_t^a}{dt}=-\frac{1}{2}\Ric^{\#} (W_t^a) +\nabla_{W_t^a} A^a, \qquad W_0^a={\rm Id}.
\end{equation}
Let $(W_t)$ the the damped parallel translation along $(Y_t)$.
We take the version constructed by Theorem~\ref{T3} so
$(W_t)$ is an adapted right continuous stochastic process such that $\di \lim_{\e \to 0} W_t^\e=W_t$
in UCP where $(W_t^\e)$ are solutions to the equations (\ref{We-100}).

 Our aim is to prove that  $W_t^a$ converges to $W_t$.
It is fairly easy to see the convergence when $Y_t^a$ and $Y_t$ are in $M^0$. When they are in a
a neighbourhood of $\partial M$,  we use the pathwise construction for $W$.  Let $(\e_n)_{n\ge 0}$ be a sequence of positive numbers converging to $0$. As soon as a continuous version of $(Y_t)$
and parallel translations along $(Y_t)$ are chosen, each $W^{\e_n}$ is constructed pathwise. 
Moreover  $W^{\e_n}$ converges to $W$ locally
in $\scrS_2$ and there exists a subsequence of $\e_{n_k}$ such that $W^{\e_{n_k}}(\omega)$ converges locally uniformly for almost surely all $\omega$.

Denote ${\mathfrak L}(\omega)$ the set of times $Y_t$ spend on the boundary. Let $F_0$ be a tubular neighbourhood of $\partial M$. On $\{Y_t\in F_0\}$,
we write\begin{equation}
\Label{E22}
W_t=W_t^T+f(t) \;\nu_{Y_t},
\end{equation}
where $f(t)$ is its component along $\nu_{Y_t}$ and $W_t^T$ its orthogonal complement. 

The proposition below is a local result. We prove the following two ways of removing the normal part  from the damped parallel translation are equivalent.
(1) During an excursion $(l_\alpha, r_\alpha)$, evolve $(W_t)$ with the continuous damped parallel translation equation,
then remove the normal part at the touching down time $r_\alpha$;
(2) at the beginning of every excursion remove the normal part of $W_t$ and then evolve $(W_t)$ with the continuous damped parallel translation equation during an excursion.  
This equivalence is due to the fact that every beginning of excursion is the right limit of ends of excursions and 
every end of excursion is the left limit of beginning of excursions. 
Notice the integral with respect to the local time is well explained by the approximation by $W_t^\varepsilon$ and later 
by the approximation by $W_t^a$, but is absent of the description here. 

\begin{proposition}
\Label{P2}
Let $S_1,S_2$ be stopping times and $t\in [S_1(\omega),S_2(\omega)]$. Let $\zeta=\inf\{ t>0: Y_t\in \partial M\}$.
We assume the following 
conditions.
\begin{enumerate}
\item The Ricci curvature and the shape operator are bounded on $E_\delta$.
\item $Y_t(\omega)\in F_0$ whenever $t\in [S_1(\omega),S_2(\omega)]$.
\end{enumerate}
Then for almost surely all $\omega$, $W_t=W_t^T+f(t) \;\nu_{Y_t}$ where $f(t, \omega)$ is a right continuous real-valued
 process  vanishing on  $ [S_1(\omega), S_2(\omega)]\cap {\mathfrak L}(\omega)$.
Furthermore,
\begin{equation}
\Label{E26}
f(t)=\left\{
\begin{array}{cc}
r_t\quad&\hbox{if}\quad  t< \zeta\\
r_t-r_{\a_t}\quad &\hbox{if}\quad  t\ge \zeta,
\end{array}
\right.
\end{equation}
where $\a_t(\omega)=\sup\{s\le t,\ Y_s(\omega)\in\partial M\}=\sup \left([S_1(\omega), t\wedge S_2(\omega))\cap {\mathfrak L}(\omega)\right)$,
and
\begin{equation}
\Label{E24}
r_t=\langle W_{S_1}, \nu(Y_{S_1})\rangle -\f12 \int_{S_1}^t \Ric( W_s,\nu_{Y_s})\,ds
+ \int_{S_1}^t\langle W_s, D\nu_{Y_s}\rangle,
\end{equation}
on $[S_1(\omega), S_2(\omega)]$. Furthermore,
\begin{equation}
\begin{split}\Label{E27}
DW_t^T=&-\f12 \Ric^\sharp(W_t)^T \,dt-\mcS(W_t^T)\,dL_t
-\langle W_t, D\nu_{Y_t}\rangle \nu_{Y_t}\\
&-\langle W_t, \nu_{Y_t}\rangle D\nu_{Y_t}
-{\frac12} \sum_{k=2}^m\<W_t, \nabla_{\sigma_k(Y_t)} \nu(Y_t)\>\nabla_{\sigma_k(Y_t)} \nu dt.
\end{split}
\end{equation}
Conversely, if a right-continuous  $L(T_{Y_0}M,T_{Y_t}M)$-valued process $W'_t$ satisfies (\ref{E22}-\ref{E27}), then it satisfies~\eqref{W}.
\end{proposition}
As a local result, this can be reduced to the half plane model, the latter was dealt with in ~\cite[N. Ikeda and S. Watanabe]{Ikeda-Watanabe}. 
Our global description and the proof we give below will  be used for our approximation result (Theorem~\ref{T4} and 
Corollary~\ref{T4bis}).

\begin{proof}
Denote $f_t=f(t)=\langle W_t,\nu_{Y_t}\rangle$. The formulas below in the proof are interpreted and obtained as following:
we first prove the corresponding identity for $W_t^\e$ and then take $\e\to 0$.
Firstly we compute the stochastic differential of $f_t$:
\begin{equation*}
df_t=-{\frac12}\<\Ric^\sharp(W_t), \nu_{Y_t}\>dt+\<W_t, D\nu_{Y_t}\>
-\1_{\{t\in \SR(\omega)\}}\langle W_{t-},\nu_{Y_{t}}\rangle,
\end{equation*}
for which we used the fact that $\mcS(W_t^T)=\mcS(W_t)$ is orthogonal to $\nu_{Y_t}$.
Then from $W_t^T=W_t-f_t\nu_{Y_t}$,  
$$DW_t^T =DW_t-df_t \nu_{Y_t} -f_t D\nu_{Y_t}-\f12D[f_\cdot, \nu_{Y_\cdot}]_t$$ where the covariant 
square bracket $D[f_\cdot, \nu_{Y_\cdot}]_t$ is the martingale bracket including the jump part.  
The jump part of the bracket disappears since
$\nu_{Y_t}$ is a sample continuous process. Thus
$$DW_t^T=-{\frac12}(\Ric^\sharp(W_t))^Tdt -\mcS(W_t^T)\,dL_t
-\<W_t, D\nu_{Y_t}\>\nu_{Y_t} -f_t D\nu_{Y_t}-{\frac12}D\<f_\cdot, \nu_{Y_\cdot}\>_t,$$
where $D\<f_\cdot, \nu_{Y_\cdot}\>_t$ is the continuous part of the martingale bracket. The martingale part of
$\nu_{Y_t}$ is $\sum_{k=2}^m\nabla_{\sigma_k(Y_t)} \nu\,dB_t^k$; while the martingale part of
$\<W_t, D\nu_{Y_t}\>$ is $$\sum_{k=2}^m\<W_t, \nabla_{\sigma_k(Y_t)} \nu(Y_t)\>dB_t^k.$$
This means that 
$$D\<f_\cdot, \nu_{Y_\cdot}\>_t  =\sum_{k=2}^m\<W_t, \nabla_{\sigma_k(Y_t)} \nu(Y_t)\>\nabla_{\sigma_k(Y_t)} \nu dt, $$
 concluding ~\eqref{E27}.
 
For $t<\zeta$, ~\eqref{E26} clearly holds. We prove it holds also for $t>\zeta$. 
If $t\in \SR(\om)\equiv \{ r_\alpha(\omega)\}$,  $f_t=0$ by the definition. This agrees with ~\eqref{E26}: $\alpha_t=t$ and $r_t-r_{\alpha_t}=0$.

On $[S_1,S_2]$ the process $R_t$ is equivalent in law to a reflected Brownian motion see Lemma~\ref{EqRBM}.
So for every   $t\in {\mathfrak L}(\omega)\backslash \SR(\om)$, there exists an increasing sequence $(t_n)_{n\in \N}$ 
of elements of  $\SR(\om)$ converging to~$t$. For all $n\in \N$ we have $f(t_n)=0$ and 
\begin{align*}
f(t)=f(t_n)+\int_{t_n}^t df(s)
=0+\int_{t_n}^t\langle DW_s,\nu_{Y_s}\rangle +\int_{t_n}^t\langle W_s,D\nu_{Y_s}\rangle.
\end{align*}
This formula makes sense by choosing a continuous version of the integral $\int_{S_1}^\cdot \langle W_s,D\nu_{Y_s}\rangle$ and by 
remarking that $(W_t)$ is the pathwise solution to equation~\eqref{W}. 

So we have 
\begin{align*}
f(t)^2=&\int_{t_n}^t2f(s)\,df(s)+\int_{t_n}^t df(s)\,df(s)\\
=& \int_{t_n}^t 2f(s) \left(\langle DW_s,\nu_{Y_s}\rangle +\langle W_s,D\nu_{Y_s}\rangle\right)
+\int_{t_n}^t{\rm trace}\langle W_s,\nabla_\cdot \nu\rangle\langle W_s,\nabla_\cdot \nu\rangle\\
&+\sum_{s\in ]t_n, t]\cap \SR(\om)}\langle W_s,\nu_{Y_s}\rangle^2\\
=&\int_{t_n}^t 2f(s)\left(\left\langle-\frac12\Ric^\sharp(W_s)ds-\mcS(W_s)dL_s, \nu_{Y_s}\right\rangle
+\langle W_s,D\nu_{Y_s}\rangle\right)\\
&+\int_{t_n}^t{\rm trace}\langle W_s,\nabla_\cdot \nu\rangle\langle W_s,\nabla_\cdot \nu\rangle
-\sum_{s\in ]t_n, t]\cap \SR(\om)}f(s)^2.
\end{align*}
Notice that the last term combines the jump term from $\<DW_s, \nu_{Y_s}\>$ and from $\<W_s, \nu_{Y_s}\>^2$. It is the sum:
$$
-2\sum_{s\in ]t_n, t]\cap \SR(\om)}f(s)\langle W_s,\nu_{Y_s}\rangle+
\sum_{s\in ]t_n, t]\cap \SR(\om)}\langle W_s,\nu_{Y_s}\rangle^2.
$$
Since the jumps are all non-positive and $ \langle \mcS(W_s),\nu_{Y_s}\rangle =0$, we get 
\begin{align*}
 f(t)^2\le &\int_{t_n}^t 2f(s)\left(-\frac12\Ric^\sharp(W_s \nu_{Y_s})ds
+\langle W_s,D\nu_{Y_s}\rangle\right) \\&+\int_{t_n}^t{\rm trace}\langle W_s,\nabla_\cdot \nu\rangle\langle W_s,\nabla_\cdot \nu\rangle.
\end{align*}
But $W_s$ is pathwise bounded in compact intervals,  and  
$$ \int_{u}^t f(s)\langle W_s,D\nu_{Y_s}\rangle,\quad \int_{u}^t{\rm trace}\langle W_s,\nabla_\cdot \nu\rangle\langle W_s,\nabla_\cdot \nu\rangle$$ are continuous in $u$. 
So the right hand side converges 
to $0$ as $n\to \infty$. 

This implies that $f(t)=0$ for all $t\in {\mathfrak L}(\omega)$.
In particular for all $t> \xi$, $f(\a_t)=0$ 
and the second equality of~\eqref{E26} is valid.

Conversely let $W_t'$ be a right-continuous process satisfying the conditions of Proposition~\ref{P2}. 
Clearly $W_t'$ satisfies~\eqref{W} when $Y_t\in M^0$. On the other hand $f(t)$ vanishes on left hand sides of excursion, 
it is right continuous, and all right hand times of excursions are limits of decreasing sequences of left hand times of 
excursions, again by Lemma \ref{EqRBM}. So it also vanishes on $\SR$, and consequently $W_t'=W_t$.


\end{proof}

We can now state the representation theorem for the heat equation on differential 1-forms, c.f.  (\ref{form}),  with the absolute boundary conditions $\phi(\nu)=0$ and $d\phi(\nu)=0$.
\begin{theorem}\Label{T6.5}
Suppose that the tubular neighbourhood of $\partial M$ has positive radius, the curvatures $\Ric^{\#}$ and  $\mcS $
are  bounded from below respectively on $M^0$ and on $\partial M$. If $\phi_t$ is a solution to the heat equation on differential 1-forms, (\ref{form}), with the absolute boundary conditions, then for any $v\in T_{Y_0}M$,
$\phi_t(v)=\E \phi_{Y_t}(W_t(v))$.
\end{theorem}

\begin{proof}
It is clear that the reflected Brownian motion $(Y_t)$  is globally defined.
Let $\psi$ be a $C^2$ differential 1-form. Since $\psi_x(w)$ is linear in $w\in T_xM$ and $ {\frac12} \sum_{k=1}^m \nabla _{\sigma_k}\sigma_k+\sigma_0=0$, we see~that
\begin{equation*}
\begin{split}
\psi(W_t)=&\psi(W_0)
+\sum_{k=1}^m \int_0^t (\nabla_{ \sigma_k(Y_s)} \psi) (W_{s-}) dB_s^k+{\frac12} \int_0^t  \Delta^1 \psi(W_{s-}) ds
\\
&+\int_0^t \nabla_{\nu(Y_{s-})}  \psi(W_{s-}) dL_s-\int_0^t \psi \left( \mcS(W_{s-}) \right)dL_s\\
&-\sum_{s\in \SR(\omega)\cap[0,t]}
 \left(\<W_{s}, \nu_{Y_s} \>- \<W_{s-}, \nu_{Y_s} \>\right) \psi ( \nu_{Y_{s}}).
\end{split}
\end{equation*}
We used Weitzenb\"ock formula $\Delta^1\psi={\rm trace} \nabla^2 \psi-\psi(\Ric^{\#} )$.
By Palais's formula for two vector fields 
$\nu$ and $V$:  
\begin{equation}\Label{dphi}
d\psi(\nu, V)=L_V(\psi(\nu))-L_\nu (\psi(V))-\psi([\nu, V])=(\nabla_\nu \psi)(V)-(\nabla_V\psi)(\nu).
\end{equation}
Hence we may commute the directions in $ \nabla_{\nu(Y_{s-})}  \psi(W_{s-})$. 

Suppose that $\psi$ satisfies the additional condition:
$\psi(\nu)=0$ and $d\psi (\nu, \cdot)=0$ on the boundary. Since $Y_s$ is continuous, $\psi ( \nu_{Y_{s-}})=0$ at the ends of an excursion, the last line vanishes. 
For any vector $w$ in
the tangent space of the boundary, $L_w(\psi(\nu))=0$ and so
$(\nabla_w \psi)(\nu)-\psi(\mcS(w))=0$.  Since $\nu_{Y_s}\<W_{s-},\nu_{Y_s}\>$ vanishes on the boundary and 
$L_s$ increases only on the boundary,
 $$\int_0^t (\nabla_\nu \psi)(\nu_{Y_s}\<W_{s-},\nu_{Y_s}\>)dL_s=0.$$
Together with the earlier argument we see the sum of the terms in the second line vanishes:
\begin{equation*}
\begin{split}
&\int_0^t \nabla_{\nu(Y_{s-})}  \psi(W_{s-}) dL_s-\int_0^t \psi \left( \mcS(W_{s-}) \right)dL_s\\
&=\int_0^t \left(\nabla_{\nu(Y_{s-})}  \psi(W_{s-}^T)+\nabla_{\nu(Y_{s-})}  \psi(\<W_{s-}, \nu(Y_{s-})\> \nu(Y_{s-})) -\psi \left( \mcS(W_{s-}) \right)  \right) dL_s=0.
\end{split}
\end{equation*}
The last identity follows from the fact that $\nabla_\nu \nu$ vanishes. The above argument should be interpreted in the following way: we first replace $W_t$ by $W_t^\varepsilon$ everywhere for 
$\varepsilon$~fixed and let $\varepsilon\to 0$ as in Theorem~\ref{T3}.

If $\phi(t, \cdot)$ is the solution to the heat equation on 1-forms with absolute boundary conditions and initial value $\phi$, on  a neighbourhood of the boundary,
\begin{equation} \label{heateq-10}
\begin{split}
\phi(W_t)=&\phi(t,W_0)
+\sum_{k=1}^m \int_0^t (\nabla_{ \sigma_k(Y_s)} \phi(t-s,  W_{s-}) dB_s^k.
\end{split}
\end{equation}
Note that $\phi$ is bounded and $\di \E \sup_{s\le t} |W_s|^2$ is finite, c.f. Lemma \ref{L7.6},  we take expectations of both sides of (\ref{heateq-10}) to obtain $\di \phi(t,v)=\E[\phi(W_t(v))]$.
\end{proof}

Let $T>0$. If $F(t,x)$ is a real valued function on $[0,T]\times M$, we denote by $dF(t,x)$ its differential in the second variable
and $\nabla F(t,x)$ the corresponding gradient.
\begin{theorem}
 \Label{T3bis}
 Let $(W_t)$ be the solution of (\ref{W}).  
If $F : [0,T]\times M\to \RR$ is a $C^{1,2}$ function such that $F(t,Y_t)$ is a continuous local martingale 
(or equivalently $F$ solves~\eqref{E19} below),  then 
$ d F(t,Y_t) (W_t)$ is also a local martingale. 
\end{theorem}

\begin{remark}
The statements in Theorem~\ref{T3bis}, also in Corollary~\ref{BF} and Theorem~\ref{T6.5},  are valid with $W_t$ replaced by $W_t^\varepsilon$. But they are more powerful (and more intrinsic) with $W_t$, for the reason that $|W_t|$ is expected to be smaller than $|W_t^\varepsilon|$. 
\end{remark}

\begin{proof}
It is clear that, on $\{Y_t\in M^0\}$,  $d(\langle \nabla F(t,Y_t),W_t\rangle)$ is the differential of a local martingale, hence we only need to prove the result on $\{Y_t\in F_0\}$.
We write the It\^o formula for $F(t,Y_t)$, the It\^o differential $d(F(t,Y_t))$ satisfies the following identity:
\begin{equation}
  \begin{split}\Label{ItoF}
 d(F(t,Y_t))=&\langle \nabla F(t,Y_t),\s(Y_t)\, dB_t\rangle \\&+\left(\partial_t+\f12\D\right)F(t,Y_t)\,dt+ \langle d F(t,Y_t),\nu_{Y_t}\rangle
 \,dL_t.
 \end{split}
\end{equation}
By the local martingale property of $F(t,Y_t)$ the last two terms vanishes and
\begin{equation}
 \begin{split}\Label{E19}
 &\left(\partial_t+\f12\D\right)F(t,y)=0, \ (t,y)\in [0,T]\times M^0,\\
 &\nu_y\in \ker d F(t,y), \ (t,y)\in [0,T]\times \partial M.
\end{split}
 \end{equation}
Since $W_t$ has finite variation on the set $\{Y_t\not\in \partial M\}$ there is no covariation term between
$dF(t,Y_t)$ and $W_t$. Writing an It\^o formula for $\langle d F(t,Y_t),W_t\rangle$ yields
\begin{multline*}
 d\left \langle \nabla F(t,Y_t),W_t\right\rangle=\nabla d F(t,Y_t)(\s(Y_t)dB_t, W_t)+\nabla d F(t,Y_t)(\nu_{Y_t}, W_t)\,dL_t
 \\+\left(\partial_t+\f12\trace \nabla^2\right)dF(t,Y_t)(W_t)\,dt\\
 -\f12\left\langle \nabla F(t,Y_t), \Ric^\sharp(W_t)\right\> dt- \left \langle \nabla F(t,Y_t) \mcS(W_t)\right\rangle \,dL_t.
 \end{multline*}
 where in the last term we used  (\ref{E19}).
We note that $\Delta^1=\trace \nabla^2-\f12\Ric^\sharp$ and $\Delta^1 d=d\Delta$.
This together with  ~\eqref{E19}, $ \left(\partial_t+\f12\Delta^1\right)dF(t,y)=0$, yields
\begin{align*}
 d\langle \nabla F(t,Y_t),W_t\rangle=&\nabla d F(t,Y_t)(\s(Y_t)dB_t, W_t)\\&+\nabla d F(t,Y_t)(\nu_{Y_t}, W_{t-})\,dL_t
 -\left\langle \nabla F(t,Y_t),\mcS(W_t)\right\rangle \,dL_t.
 \end{align*}
 Now for $y\in \partial M$ and $w\in T_yM$, since $\nu(y)\in \ker d F(t,y)\in T_y\partial M$ we have 
 \begin{align*}
  -\langle \nabla F(t,y),\mcS(w)\rangle&= \langle \nabla F(t,y),\nabla_w\nu \rangle\\
  &=-\langle \nabla_w d F(t,y),\nu_y\rangle=-\nabla dF(t,y)(\nu_y,w).
 \end{align*}
 For the second equality we used the fact that $\nu(y)\in \ker d F(t,y)$. Putting all the calculations together we finally get 
 $$
 d\langle d F(t,Y_t),W_t\rangle=\nabla d F(t,Y_t)(\s(Y_t)dB_t, W_t),
 $$
 which proves that $\langle \nabla F(t,Y_t),W_t\rangle$ is a continuous local martingale.

\end{proof}
Applying  this theorem to $F(t,y)=\E[f(Y_{T-t}(y))]$ where $(Y_.(y))$ is reflected Brownian motion started at $y\in M$,
$f$ is a smooth function on $M$ with $ df (\nu)=0$ on the boundary, (under this condition $F$ is $C^{1,2}$,  
see e.g. \cite[F.-Y. Wang]{Wang:15}),
we immediately get the following Bismut type formula:
\begin{corollary}
\Label{BF}
Assume that $M$ is compact. 
Let $f:M\to \R$ be a smooth bounded function with $\langle df, \nu\rangle=0$ on the boundary and $T>0$. Let $Q_t$ be the semi-group associated to the reflected Brownian motion on $M$. Let $y\in M$, $v\in T_yM$ and $(Y_t)$ a reflected Brownian motion started at $y\in M$, constructed as in Theorem~\ref{T2}.  Then
$$d(Q_T f)(v)=\frac{1}{T} \E \left[f(Y_T) \int_0^T \<W_s(v), \sigma(Y_s) dB_s\>\right].$$
\end{corollary}
For the analogous formula for manifold without boundary, see    \cite[Li]{Li:thesis} and
\cite[K. D. Elworthy and X.-M. Li]{Elworthy-Li-formula}. Such results are also  obtained in  \cite[L. Zambotti]{Zambotti} 
and  \cite[T. Funaki and K.  Ishitani]{Funaki-Ishitani}.

Let $T$ and $a$ be positive numbers.
Recall that the damped parallel translation along a sample continuous stochastic process $(Y^a_t)$
 is the solution to the stochastic covariant differential equation with initial value $W_0^a={\rm Id}_{T_{Y_0^a}M}$,
\begin{equation}\Label{Wa}
 DW_t^a=\left(\nabla_{W_t^a}A^a-\f12\Ric^\sharp(W_t^a)\right)\,dt. 
\end{equation}
The following Theorem will be proved in Section \ref{section7}
\begin{theorem}\Label{T4}
Let $M$ be a compact Riemannian manifold. Let $$A^a(x)=\nabla \ln \tanh\left(\f{R(x)}{a}\right).$$
Let $(Y^a_t, t\in [0,T])$ and $(Y_t, t \in [0,T])$ be the stochastic processes defined in Theorem~\ref{T2}. 
Let $W_t^a$ denote the damped parallel translation along $Y_t^a$. Then for all $p\in [1,\infty)$ and  for any $C^2$ differential $1$-form  $\phi$ vanishing on the normal bundle $\nu(\partial M)$,
$$\lim_{a\to 0} \sup_{s\le t}  \E\left[|\phi( W_s^a)-\phi( W_s)|^p\right] =0.
$$ 
\end{theorem}

For a non-compact manifold, we have the following result, see Appendix \ref{preliminary}.
\begin{corollary}\Label{T4bis}
Let $M$ be a Riemannian manifold, not necessarily compact. 
Then for any $C^2$ differential $1$-form  $\phi$ such that $\phi(\nu)=0$ in  $\partial M$, 
$\phi( W^a)$ converges to $\phi( W)$ in UCP topology.
\end{corollary}

\section{Damped parallel translation as a derivative flow}
\Label{section6}
\setcounter{equation}0

In this section $M$ is a smooth compact manifold with boundary. We prove that the damped parallel translation $W_t$ along
reflected Brownian motion is the weak derivative of a flow which we explain below. 

Let $\gamma:  [0,1]\to M^o$ be a $C^1$ map. Here again for $a>0$,  $\di A^a(x)=\nabla \ln \tanh\left(\frac{R(x)}{a}\right)$. We will built a family of Brownian flows with drift $A^a$ starting at $\gamma(u)$ whose derivative with respect to $u$ is  locally uniformly bounded for a.s. $\omega$. Let us return to (\ref{Yta}).
$$dY_t^a= \sum_{k=1}^m\sigma_k (Y_t^a)\circ dB_t^k
 +\sigma_0(Y_t^a)dt +A^a (Y_t^a) \,dt.$$
 Let us consider its solution flow $\Psi^a$. Let $Y_t^a(0)=\Psi^a(\gamma(0))$.
For $u\in(0,1]$, let $Y_t^a(u)$ denote the solution to the following It\^o  equation:
\begin{equation}
\Label{Ytau}
\begin{split}
dY_t^a(u)&= \parals_{0,u}({Y_t^a(\cdot)}) dY_t^a(0)+A^a(Y_t^a(u))dt,\\
Y_0^a(u)&=\gamma(u)
\end{split}
\end{equation}
where $ \parals_{0,u}({Y_t^a(\cdot)})$ denotes parallel translation along the $C^1$ path $u\mapsto Y_t^a(u)$.
 Recall that the It\^o differentials $dY_t^a(u)$ in~\eqref{Ytau} are defined by
\begin{equation}
\Label{Itodiff}
dY_t^a(u)=\parals_{0,t}d\left(\int_0^\cdot\parals_{0,s}^{-1}\circ dY_s^a(u) \right)_t
\end{equation}
where $\parals_{0,t}$ is parallel transport along $t\mapsto Y_t^a(u)$, and they formally are tangent vectors. Notice that the first differential in the right is an It\^o differential in a fixed vector space and the second one is a Stratonovich differential in a manifold.                
More precisely, putting~\eqref{Yta} in It\^o form $\di dY_t^a= \sum_{k=1}^m\sigma_k (Y_t^a) dB_t^k
  +A^a (Y_t^a) \,dt$
	we have 
\begin{equation}
\Label{Ytaubis}
\left\{
\begin{array}{cc}
dY_t^a(u)&=\sum_{k=1}^m \parals_{0,u}({Y_t^a(\cdot)})\left( \sigma_k (Y_t^a(0))\right) dB_t^k + A^a(Y_t(u))\,dt\\
Y_0^a(u)&=\gamma(u).\hfill
\end{array}
\right.
\end{equation}
The existence of a solution should follow from an iteration method. A proof is given in  \cite[M. Arnaudon, K. A. Coulibaly and A. Thalmaier]{Arnaudon-Coulibaly-Thalmaier:06}, where an approximation procedure with iterated parallel couplings is used to obtain a Cauchy sequence in $H_2$. The advantage is that at each step and each value of $u$ we have a diffusion with the same  generator $ \frac12\Delta+A^a$, and  as the mesh goes to $0$ all problems with cut locus disappear. 
The solution curves $u\mapsto Y_t^a(u)$ are almost surely differentiable and that their derivatives $\partial_uY_t^a(u)$ are locally uniformly bounded for almost surely all $\omega$ and \begin{equation}
\Label{derhordiff}
\partial_uY_t^a(u)=W_t^{a,u} (\dot\gamma(u)),
\end{equation}
where $W_t^{a,u} $ is  the damped parallel translation along $Y^a(u)$. This is, to our knowledge, the only known construction for $\partial_uY_t^a(u)$ a.s. locally uniformly bounded. Our aim is to obtain a similar property for reflected Brownian motion. For this we will let $a\to 0$ in ~\eqref{derhordiff} and obtain a limiting identity in a weak sense. However we believe that our construction indeed yields~\eqref{derhordiff} for $a=0$ in a strong sense.

\begin{proposition}
\Label{Tightness}
The family $\{(Y_t^a(u))_{0\le t\le T,\ 0\le u\le 1},\  a\in (0,1]\}$ of two parameter stochastic processes  is tight.
\end{proposition}
\begin{proof}
We will use the Kolmogorov criterion. For $t_1, t_2, u_1,u_2$ satisfying $0\le t_1<t_2\le T$ and $0\le u_1<u_2\le 1$ and $p\ge 1$,
\begin{align*}
&\E\left[\rho^p\left(Y_{t_1}^a(u_1),Y_{t_2}^a(u_2)\right)\right]\\
&\le 2^{p-1}\left(\E\left[\rho^p\left(Y_{t_1}^a(u_1),Y_{t_2}^a(u_1)\right)\right]+\E\left[\rho^p\left(Y_{t_2}^a(u_1),Y_{t_2}^a(u_2)\right)\right]\right)\\
&\le 2^{p-1}\left(\E\left[\rho^p\left(Y_{t_1}^a(u_1),Y_{t_2}^a(u_1)\right)\right]+\E\left[\left(\int_{u_1}^{u_2}|W_{t_2}^a(u)|\cdot|\dot\gamma(u)|\,du\right)^p\right]\right)\\
&\le 2^{p-1}\left(\E\left[\rho^p\left(Y_{t_1}^a(u_1),Y_{t_2}^a(u_1)\right)\right]+(u_2-u_1)^{p-1}\|\dot\gamma\|_\infty\int_{u_1}^{u_2}\E\left[|W_{t_2}^a(u)|^p\right]\,du\right)\\
&\le 2^{p-1}\left(\E\left[\rho^p\left(Y_{t_1}^a(u_1),Y_{t_2}^a(u_1)\right)\right]+(u_2-u_1)^{p}\|\dot\gamma\|_\infty\sup_{u\in[0,1],\ t\in[0,T]}\E\left[|W_{t}^a(u)|^p\right]\right)\\
&\le 2^{p-1}\left(\E\left[\rho^p\left(Y_{t_1}^a(u_1),Y_{t_2}^a(u_1)\right)\right]+C'(p,T)(u_2-u_1)^{p}\|\dot\gamma\|_\infty\right),
\end{align*}
where $C'(p,T)$ is a constant. We used an estimate on $|W_t^{a,u}|$ given in ~\eqref{BoundWap} below. Here and several time in the sequel, we use the equality in law of the processes $(Y^a(u), W^a(u))$, for each fixed $u$,  and $(Y^a, W^a)$. The latter process was constructed in Sections ~\ref{section3} and ~\ref{section5}. 

For the first term on the right hand side we again use the fact that $u_1$ is fixed and use estimates for $Y^a$, from Theorem~\ref{T2}. Since $M$ is compact we can replace the distance $\rho(x,y)$ on $M$ by the equivalent distance $\|H(x)-H(y)\|$ where $H : M \to \R^d$ is an embedding. We can also assume that $H=(\imath, R)$ is an extension of the construction in~\eqref{Hx} around the boundary. In particular we can assume that the image of $\partial M$ by $H$ is included in $\{R=0\}$. Then we easily check that
\begin{itemize}
\item
  the drift of $\tilde\imath(Y_t^a(u_1))$  is bounded,
	\item 
	the drift of $\left(R(Y_t^a(u_1))-R(Y_{t_1}^a(u_1))\right)^4$ is bounded on $\{R(Y_t^a(u_1))\ge \delta_0\},$
	\item the drift of $\left(R(Y_t^a(u_1))-R(Y_{t_1}^a(u_1))\right)^4$ is negative  on $$\{R(Y_t^a(u_1))\le \delta_0\}\cap \{R(Y_t^a(u_1))\le \{R(Y_{t_1}^a(u_1))\} ,$$
	\item the drift of $\left(R(Y_t^a(u_1))-R(Y_{t_1}^a(u_1))\right)^4$ is positive  on $$\{R(Y_t^a(u_1))\le \delta_0\}\cap \{R(Y_t^a(u_1))\ge \{R(Y_{t_1}^a(u_1))\} $$
	and bounded above by $b\left(R(Y_t^a(u_1))-R(Y_{t_1}^a(u_1))\right)^2$   where $b>0$ is independent of~$a$ (this is a consequence of~\eqref{ERat}). 
	\end{itemize}
	This implies, by a standard calculation, that for some constant $C'>0$,
\begin{equation}
\Label{HY}
\E\left[\left\|H(Y_{t_2}^a(u_1))-H(Y_{t_1}^a(u_1))\right\|^4\right]\le C'|t_2-t_1|^2.
\end{equation} 
Finally, for some positive constant $C$,
\begin{equation}
\Label{rho4}
\E\left[\rho^4\left(Y_{t_1}^a(u_1),Y_{t_2}^a(u_2)\right)\right]\le 8C|t_2-t_1|^2+8C'(4,T)(u_2-u_1)^{4}\|\dot\gamma\|_\infty.
\end{equation}
This concludes the required tightness.
\end{proof}
With this result at hand we  construct our limiting process.
\begin{theorem}
\Label{weakflow}
There is a two parameter continuous  process $(Y_t(u))_{0\le t \le T,\ 0\le u\le 1}$ with the following properties: \begin{itemize}
\item[(1)] for every $u\in[0,1]$, $Y_t(u)$ is a reflected Brownian motion on $M$ started at $\gamma(u)$;
\item[(2)] for every $p\in [1,\infty)$ there exists a number $C'(p,T)$ s.t. for all $0\le u_1< u_2\le 1$,
\begin{equation}
\Label{rhopY2}
\sup_{0\le t\le T}\E\left[\rho^p(Y_t(u_1), Y_t(u_2))\right]\le C'(p,T)\|\dot\gamma\|_\infty (u_2-u_1)^p.
\end{equation}
\end{itemize}
\end{theorem}
\begin{proof}
By Proposition~\ref{Tightness}, there exists a sequence $a_k\to 0$ such that the two parameter family of stochastic processes $Y_\cdot^{a_k}(\cdot)$ converges in law whose limit we denote by $Y_\cdot(\cdot)$. 

Let us fix $u\in [0,1]$. Since the convergence considered is in the weak topology, 
we are allowed to use another construction of $Y_t^a(u)$, namely part (2) in Theorem~\ref{T2}, in which the convergence is stronger. The limit is reflected Brownian motion started at $\gamma(u)$. This yields (1). 

Let us then  take $t_1=t_2=t$
in the computation for tightness in Proposition~\ref{Tightness}.
Then take $k\to \infty$ to obtain (2).
\end{proof}


For each $u\in[0,1]$ fixed, the stochastic processes $Y_t^{a_k}(u)$ converges in law to $Y_t(u)$. 
So the damped parallel translations $W_t^{a_k}(u)$, as stochastic processes on $[0, T]$, converge in law to $W_t(u)$ in the following sense: if $\phi$ is a $C^2$ differential 1-form such that $\phi|\partial M(\nu)=0$, 
then $\phi(W_t^{a_k}(u))$ converges in law to $\phi(W_t(u))$. This is due to the fact that $W^a$ is a functional of $Y^a$, $W$ is a functional of $Y$, so we can apply Corollary~\ref{T4bis}.

Unfortunately this argument does not allow us to prove the convergence of
 $W_t^{a_k}(u)$ converges  to $W_t(u)$, which would yield $\partial_uY_t(u)=W_t(u)(\dot\gamma(u))$. However the following theorem asserts this equality in a weak sense. 
\begin{theorem}
\Label{Weakderivative}
Let $M$ be compact.
For all $f\in C^2(M)$ satisfying $\langle \nabla f,\nu\rangle=0$ on the boundary, then
\begin{equation}
\Label{weakder}
\E\left[f(Y_t(u_2))-f(Y_t(u_1))-\int_{u_1}^{u_2}\langle df(Y_t(u)),W_t(u)(\dot\gamma(u))\rangle\,du\right]=0.
\end{equation} 
\end{theorem}
\begin{proof}
By~\eqref{derhordiff},
$$
f(Y_t^{a_k}(u_2))-f(Y_t^{a_k}(u_1))-\int_{u_1}^{u_2}\langle df(Y_t^{a_k}(u)),W_t^{a_k}(u)\dot\gamma(u)\rangle\,du=0.
$$
Since all the terms are integrable we can take the expectation,
So 
$$
\E\left[f(Y_t^{a_k}(u_2))\right]-\E\left[f(Y_t^{a_k}(u_1))\right]-\int_{u_1}^{u_2}\E\left[\langle df(Y_t^{a_k}(u)),W_t^{a_k}(u)\dot\gamma(u)\rangle\right]\,du=0.
$$
Now by Theorem~\ref{T2} and Corollary \ref{T4bis} and dominated convergence theorem:
$$
\E\left[f(Y_t^{a_k}(u_2))\right]-\E\left[f(Y_t^{a_k}(u_1))\right]-\int_{u_1}^{u_2}\E\left[\langle df(Y_t^{a_k}(u)),W_t(u)\dot\gamma(u)\rangle\right]\,du=0.
$$
Finally we use Fubini-Tonelli Theorem to obtain~\eqref{weakder}.
\end{proof}
 
\section{Proof of Theorem~\ref{T4}}
\Label{section7}
\setcounter{equation}0

We first reduce the proof of Theorem~\ref{T4} to the class of $C^2$ differential 1-forms $\phi$ vanishing in a 
neighbourhood of the boundary, this is the content of Section \ref{Subsection9.1}.
  We then prove the convergence of the tangential part of the parallel transport in the topology of UCP, followed by the convergence of its normal part $f_t^a$. By the latter we mean that for all  smooth 
 $\phi : M\to \RR_+$ vanishing in a neighbourhood of $\partial M$, 
 $\phi(Y_t^a)f_a(t)\to \phi(Y_t)f(t)$ in the UCP topology. See Sections~\ref{TP} and~\ref{NP}. 

 We describe briefly the strategy and the main difficulties. Thanks to the convergence of parallel transports
 established in Proposition \ref{cvpt} we only need to prove that $w_t^a\to w_t$ where $w_t^a=(\parals_t^a)^{-1}W_t^a$
 and $w_t=(\parals_t)^{-1}W_t$, more precisely that, writing
\begin{equation}
 w_t^a=w_t^{a,T}+f_a(t)n_t^a
  \Label{E33}
 \end{equation}
with $n_t^a=(\parals_t^a)^{-1}\nu(Y_t^a)$ and $w_t^{a,T}$ orthogonal to $n_t^a$, 
$w_t^{a,T}\to w_t^T$ and  
  for any $C^2$ map $\phi: M\to \R$ vanishing in a neighbourhood of $\partial M$, $\phi(Y_t^a)f_a(t)n_t^a\to \phi(Y_t)f(t)n_t$.

Firstly, the integral equation for $(f_a, w^{a,T})$ has the following form
 \begin{equation}\Label{matrix-a}
  \left(\begin{array}{c}
         f_a(t)\\w_t^{a,T}
        \end{array}
\right)=\left(\begin{array}{c}
         f_a(0)\\w_0^{a,T}
        \end{array}
\right)+\int_0^tdM_s^a\left(\begin{array}{c}
         f_a(s)\\w_s^{a,T}
        \end{array}
\right)+\left(\begin{array}{c}
         f_a(0)e^{-\tilde C_a(t)}\\0
        \end{array}
\right),
 \end{equation}
where $(M_t^a)$ is a matrix valued process for the following form
 \begin{equation}\Label{matrix-a2}
  M_t^a=\left(\begin{array}{cc}
         0&\tilde u_t^a\\v_t^{a,N}&v_t^{a,T}
        \end{array}
\right)
\end{equation}
whose components are to be specified later. Also,
\begin{equation}\Label{matrix}
  \left(\begin{array}{c}
         f(t)\\w_t^{T}
        \end{array}
\right)=\left(\begin{array}{c}
         f(0)\\w_0^{T}
        \end{array}
\right)+\int_0^tdM_s\left(\begin{array}{c}
         f(s)\\w_s^{T}
        \end{array}
\right)+\left(\begin{array}{c}
         f(0)e^{-\tilde C(t)}\\0
        \end{array}
\right)
 \end{equation}
where $M_t$ is of the following form:
 \begin{equation}\Label{matrix2}
  M_t=\left(\begin{array}{cc}
         0&\tilde u_t\\v_t^{N}&v_t^{T}
        \end{array}
\right).
\end{equation}

If $(V_t)$ is a vector valued stochastic process, denote
 $$\|V\|_{\scrS_p([0,T])}=\E \left(\sup_{0\le s\le T} \E|V_s|^p\right).$$
We will see that the components of $M^a$ converge to the corresponding components of $M$ in $\SH_p([0,T])$ for all $p\ge 1$, $T>0$, with the exception $\tilde u^a$ which contains $e^{-\tilde C_a(t)}$ and $v_t^{T}$ which contains local time of the distance to boundary. 
The main difficulty is the convergence of $f_a(0)e^{-\tilde C_a(t)}$ to $f(0)e^{-\tilde C(t)}$. The convergence is only in $L^p(dt\times \P)$, see Corollary~\ref{C2}. Also the convergence of $L_t^a$ to $L_t$: it is in $\scrS_p([0,T])$ but not in $\SH_p([0,T])$ and
 will require several integrations by parts.
We note in the last term in equation~\eqref{matrix-a}, the tangential and the normal part decouples.
The matrix~\eqref{matrix-a2} is in the lower triangular form. It is therefore possible
 to split the proof into  the convergence of $w^{a,T}$ to $w^T$ and the convergence of $f_a(t)$ to $f(t)$. This procedure is essential for our proof to work.

 
\subsection{Localisation}
\Label{Subsection9.1}

\begin{lemma}
Let $S_1$ and $S_2$ be stopping 
times such that for $a$ sufficiently small, $Y_t\in E_0$ and $Y_t^a\in E_0$ on $\{\omega: S_1(\omega) \le t \le S_2(\omega)]$.
\begin{enumerate}
\item If  Theorem~\ref{T4} holds for the class of $C^2$ differential 1-forms $\phi$ with $\phi(\nu)$ vanishing in a neighbourhood of the boundary, then it holds for all $C^2$ 1-form $\phi$ such that $\phi(\nu)=0$ on $\partial M$.
\item It is sufficient to prove that for $t\in [S_1,S_2]$,
$\phi( W^a)$ converges to $\phi( W)$ in UCP.  
\item If $\phi( W^a)$ converges to $\phi( W)$ in the UCP topology and $M$ is compact, then
$$\lim_{a\to 0}\E \sup_{s\le t} \left|\phi( W^a)-\phi( W)\right|^p\to 0.$$
\end{enumerate}
\end{lemma}
Without loss of generality 
we will assume that $S_1=0$ and let $S_2=S$.

\begin{proof}
Let $\phi$ be a $C^2$ differential 1-form such that $\phi(\nu)=0$ on $\partial M$. Then there exists a family of 
$C^2$ differential $1$-forms $\phi^\e$  such that $\langle \phi^\e,\nu\rangle=0$ in a neighbourhood of $\partial M$ and $\sup_{x\in M} \|\phi^\e(x)-\phi(x)\|<\e$: choose for instance $\phi^\e(u)=\phi(u)-\langle u,\nu\rangle f^\e(\pi(u)) \phi(\nu)$ where $f^\e$ is a smooth function on $M$ satisfying 
\begin{itemize}
\item $f^\e=0$ on $\{R\ge \alpha\}$,
\item $f^\e=1$ on $\{R\le \alpha/2\}$,
\item $0\le f\le 1$ on $M$,
\end{itemize}
and $\alpha\in (0,\delta_0)$ is chosen in such a way that $| \phi(\nu)| <\e$ on $\{R< \alpha\}$.

By the assumption, 
  $\phi^\e(W^a)\to \phi^\e(W)$ in $\scrS_p([0,T])$. On the other hand 
 \begin{align*}
 |\phi(W_t^a)-\phi(W_t)|&\le  |\phi(W_t^a)-\phi^\e(W_t^a)|+ |\phi^\e(W_t^a)-\phi^\e(W_t)|+ |\phi^\e(W_t)-\phi(W_t)|\\
 &\le \e |W_t^a|+|\phi^\e(W_t^a)-\phi^\e(W_t)|+\e|W_t|.
 \end{align*} 
 If $a$ is sufficiently small then $\|\phi^\e(W^a)-\phi^\e(W)\|_{\scrS_p([0,T])}<\e$ by Theorem \ref{T4}. 
Using lemma~\ref{Wbdd} we get
 $$
 \|\phi(W^a)-\phi(W)\|_{\scrS_p([0,T])}<\e(1+2C)
 $$
 for $a$ sufficiently small. Taking $\e\to 0$ we obtain  $\|\phi(W^a)-\phi(W)\|_{\scrS_p([0,T])}\to 0$.
 
 (2) We note that $Y_t^a$ converges to $Y_t$ in UCP topology and inside $M^0$ the coefficients for $W_t^a$ converge smoothly and uniformly to the coefficients of the equation for $W_t$.
If for any $t\in [S_1,S_2]$ where $S_1,S_2$ are stopping  times such that for all $t\in [S_1, S_2]$ and $a$ sufficiently small, $Y_t\in E_0$ and $Y_t^a\in E_0$, then $\phi( W^a)$ converges to $\phi( W)$ in UCP.

(3) 
 By Lemma \ref{L1} below, $\int_0^t\f{2}{a\sinh\left(\f{2R(Y_s^a)}{a}\right)}\,ds$  converges to $L_t$ in $\scrS_p$ for all $p\in[1,\infty)$.
 Let $t>0$.  Since $M$ is compact, by corollary~\ref{C4} and lemma~\ref{Wbdd},  if $\phi( W^a)$ converges to $\phi( W)$ in UCP topology,
  $$\lim_{a\to 0} \E \left( \sup_{s\le t} \left| \phi( W_s^a)-\phi( W_s)\right|^p\right)= 0.$$ 
\end{proof}

 \subsection{Preliminary Computations}
\Label{Subsection4.2}

Let  $a>0$.
The damped parallel translation along $Y_t^a$ satisfies the following equations:
 \begin{align*}
 {D}  W_t^a=&\nabla_{W_t^a} A^a \; dt-\f12\Ric^\sharp_{Y_t^a}(W_t^a)\; dt.
 \end{align*}
 In the tubular neighbourhood $F_0$ on which the approximating SDEs were constructed, take $y\in M^0$ and $w\in T_yM$.
 Then
 \begin{align*}
 \nabla_w A^a=&\nabla_w\nabla \ln\tanh \left(\f{R}{a}\right)=\nabla_w\left(\f{2\nabla R}{a\sinh\left(\f{2R}{a}\right)}\right)\\
 =&-\f{4}{a^2}  \f{\cosh}{\sinh^2} \left(\f{2R}{a}\right) 
  \langle w,\nu_y\rangle \nu_y
 -\f{2}{a\sinh\left(\f{2R}{a}\right)}\mcS(w),
 \end{align*}
 Let us define  $R_t^a=R(Y_t^a)$, $R_t=R(Y_t)$,
 \begin{equation}\label{ca}
 c_a(t)=\f{4}{a^2}\frac {\cosh} {\sinh^2}    \left(\f{2R_t^a}{a} \right),
 \end{equation}
and \begin{equation}\label{E-30-1}
 f_a(t)=\langle w_t^a,n_t^a\rangle=\langle W_t^a,\nu(Y_t^a)\rangle.
 \end{equation}
 We also denote $W_t^{a,T}$ the tangential part of $W_t^a$: 
 \begin{equation}
  \Label{E30}
 W_t^{a,T}=W_t^a-f_a(t)\nu(Y_t^a).
 \end{equation} 
 \begin{definition}
 Let
 \begin{equation}
 L_t^a=\int_0^t\f{2}{a\sinh\left(\f{2R(Y_s^a)}{a}\right)}\,ds.
\end{equation}
 \end{definition}
 Below $\parals_t^{-1}$ is shorthand for $\parals_t^{-1}(Y_\cdot^a)$. For $x\in M$, denote
 $\|\nabla \nu(x)\|^2=\sum_k \|\nabla _{\sigma_k} \nu\|_x^2$. The latter is the Hilbert-Schmidt norm of the linear operator
 $\nabla \nu :T_xM\to T_xM$. In the following formulas we should consider the integrals are in It\^o form. 
 Hence the equation for $W_t^{a,T}$  should be interpreted as for $\parals_t^{-1}W_t^{a,T}$.
\begin{lemma}
\label{covariant-formula}
In the tubular neighbourhood the following  formulae hold.
\begin{align*}
DW_t^a=-\f12\Ric^\sharp(W_t^a)\,dt -c_a(t)  f_a(t) \nu(Y_t^a) \; dt
- \mcS(W_t^a)\;  dL_t^a,\\
d \parals_t^{-1}\nu(Y_t^a) dt= \parals_t^{-1}\sum_k\nabla_{\sigma_k} \nu(Y_t^a) dB_t^k +
\frac{1}{2}  \parals_t^{-1}\trace \nabla^2 \nu (Y_t^a) dt,
\end{align*}
Finally the stochastic differential of the tangential part of $W_t^a$ has the following tangential and normal decomposition
\begin{align*}
DW_t^{a,T} =&-\f12 \left(\Ric^\sharp(W_t^{a,T})\right)^T\, dt-\f12 f_a(t)\left(\Ric^\sharp(\nu(Y_t^a))\right)^T\, dt- \mcS(W_t^a)\;  dL_t^a
\\&-f_a(t) \nabla_{\sigma_k} \nu(Y_t^a) dB_t^k -
\frac{1}{2}  f_a(t)\trace \nabla^2 \nu (Y_t^a) dt\\&
-\frac 12 \langle W_t^{a,T},\nabla_{\sigma_k} \nu (Y_t^a) \rangle \nabla_{\sigma_k} \nu (Y_t^a)  \;dt
  \\
&-\sum_k\<W_t^{a,T}, \nabla_{\sigma_k} \nu(Y_t^a) \> \nu(Y_t^a)dB_t^k -
 \<W_t^{a,T},  \trace \nabla^2 \nu (Y_t^a)  \>\nu(Y_t^a)dt\\
&   +f_a(t) \|\nabla \nu(Y_t^a)\|^2\nu(Y_t^a) dt,
\end{align*}
\begin{align*}
df_a(t)=-c_a(t) f_a(t)dt -\frac 12 \Ric(W_t^a, \nu(Y_t^a))dt
 +\sum_k\<W_t^{a,T},  \nabla_{\sigma_k} \nu (Y_t^a)\>\;dB_t^k\\
   + \f12\<W_t^{a,T}, \trace \nabla^2 \nu(Y_t^a)\>dt-\f12 f_a(t)\|\nabla \nu(Y_t^a)\|^2.
\end{align*}
\end{lemma}
 \begin{proof}
 The first formula is clear, the second is straight forward after applying It\^o's formula to 
 the equation for $(Y_t^a)$:
 \begin{equation*}
  (\parals_t^a)^{-1} D\nu(Y_t^a)=(\parals_t^a)^{-1}\left(\langle \nabla \nu, \circ dY_t^a\rangle +\f12 \trace\nabla^2 \nu(Y_t^a)\,dt\right).
 \end{equation*}
Since $\< \nabla \nu, A^a\>=0$,  this yields
 \begin{align*}
  (\parals_t^a)^{-1} D\nu(Y_t^a)
  =&(\parals_t^a)^{-1}\left(\sum_k\left \langle \nabla \nu, \sigma_k(Y_t^a)\right\rangle 
	\,dB_t^k+\f12 \trace \nabla^2 \nu(Y_t^a)\,dt\right)
  \\&+(\parals_t^a)^{-1}\nabla _{A^a(Y_t^a)} \nu dt.
 \end{align*}
 Note that $\nu=\nabla R$. If $\gamma_t$ is the geodesic from $x$ to $\pi(x)$ then $\dot \gamma(t)=\nabla R(\gamma(t))$
 and hence $\nabla_\nu \nu=0$ and the second formula follows. 
 We note also that 
 \begin{align*}
 \<W_t^a, \circ D(\nu(Y_t^a))\>&= \sum_k\<W_t^a, \nabla_{\sigma_k} \nu(Y_t^a) \>dB_t^k 
 + \<W_t^a,\frac{1}{2} \trace \nabla^2 \nu (Y_t^a) \>dt\\
 &=\sum_k \<W_t^a, \nabla_{\sigma_k} \nu(Y_t^a) \>dB_t^k 
 -\frac{1}{2} f_t\|\nabla \nu(Y_t^a)\|^2\\
 & +\frac{1}{2}  \<W_t^{a,T}, \trace \nabla^2 \nu (Y_t^a) \>dt.
 \end{align*}
 Note that the left hand side is in Stratonovich form and the right hand side in It\^o form.
 We work on the third  equation:
 \begin{align*}
 df_t=&\<D W_t^a, \nu(Y_t^a)\>+\<W_t^a,  D(\nu(Y_t^a))\> \\
 =&-c_a(t) f_a(t)dt -\frac 12 \Ric(W_t^a, \nu(Y_t^a))dt
 +\sum_k\<W_t^a,  \nabla_{\sigma_k} \nu(Y_t^a)\>\;dB_t^k\\
  & + \f12\<W_t^a, \trace \nabla^2 \nu(Y_t^a)\>dt.
 \end{align*}
 All stochastic integrals in the above formula are in It\^o form.
  The required identity follows from the observation below:
   $$\< \sum_k \nabla ^2 \nu (\sigma_k, \sigma_k), \nu\>=-\sum_k \<\nabla_{\sigma_k}\nu, \nabla_{\sigma_k}\nu\>=-\|\nabla\nu\|^2.$$

Next we compute the tangential part of the damped parallel translation.
\begin{align*}
DW_t^{a,T} &= DW_t^a-D(f_t\nu(Y_t^a))\\
=&-\f12\Ric^\sharp(W_t^a)\,dt -c_a(t)  f_a(t) \nu(Y_t^a) \; dt
- \mcS(W_t^a)\;  dL_t^a -D(f_t\nu(Y_t^a)).
\end{align*}
 For the normal part of the damped parallel transport, we use product rule
\begin{align*}
D(f_t\nu(Y_t^a))=& \nu(Y_t^a) df_t +f_t D(\nu(Y_t^a))
+  df_tD(\nu(Y_t^a))\\
=&\nu(Y_t^a) \left(-c_a(t) f_a(t)dt -\frac 12 \Ric(W_t^a, \nu(Y_t^a))dt\right)\\
 &+\nu(Y_t^a)\left( \sum_k\<W_t^a,  \nabla_{\sigma_k} \nu(Y_t^a)\>\;dB_t^k
   + \f12\<W_t^a, \trace \nabla^2 \nu(Y_t^a)\>dt.
\right)\\
&+f_t \left( \sum_k\nabla_{\sigma_k} \nu(Y_t^a) dB_t^k +
\frac{1}{2}  \parals_t^{-1}\trace \nabla^2 \nu (Y_t^a) dt\right)
\\
&+\sum_k\langle W_t^a,\nabla_{\sigma_k(Y_t^a)} \nu  \rangle \nabla_{\sigma_k(Y_t^a)} \nu  \;dt.\\
\end{align*}
 Since $\<\nabla _\cdot  \nu, \nu\>$ vanishes,
 $\<\nabla_{W_t^a}  A^a, \nu(Y_t^a)\>dt =  \mcS(W_t^a)dL_t^a.$
  Finally, we bring the above formula back to the equation for $W_t^{a,T}$ and observe that
 the cancellation of the term involving $f_a(t)$. 
\begin{align*}
DW_t^{a,T}
=&-\f12\Ric^\sharp(W_t^a)\,dt - \mcS(W_t^a)\;  dL_t^a
+\frac 12 \Ric(W_t^a, \nu(Y_t^a)) \nu(Y_t^a)dt\\
&-\sum_k \<W_t^a, \nabla_{\sigma_k} \nu(Y_t^a) \> \nu (Y_t^a)dB_t^k 
 +\frac{1}{2} f_t\|\nabla \nu\|^2\nu (Y_t^a) \\
 &-\frac{1}{2}  \<W_t^{a,T}, \trace \nabla^2 \nu (Y_t^a) \>\nu (Y_t^a)dt
-f_t \sum_k\nabla_{\sigma_k} \nu(Y_t^a) dB_t^k\\
& -\frac{1}{2} f_t\trace \nabla^2 \nu (Y_t^a) dt    
-\sum_k \langle W_t^a,\nabla_{\sigma_k(Y_t^a)} \nu   \rangle \nabla_{\sigma_k(Y_t^a)} \nu  \;dt.
\end{align*}
 Following this up and observing  that $$-\f12\Ric^\sharp(W_t^a)+ \frac 12 \Ric(W_t^a, \nu(Y_t^a)) \nu(Y_t^a)
=-\f12(\Ric^\sharp(W_t^a))^T,$$ we see
 \begin{align*}
DW_t^{a,T} =&-\f12(\Ric^\sharp(W_t^a))^T\,dt
- \mcS(W_t^a)\;  dL_t^a 
-\sum_k \<W_t^{a,T}, \nabla_{\sigma_k} \nu(Y_t^a) \> dB_t^k \\
&
 +\frac{1}{2} f_t\|\nabla \nu\|^2 -\frac{1}{2}  \<W_t^{a,T}, \trace \nabla^2 \nu (Y_t^a) \>dt
-f_t \sum_k\nabla_{\sigma_k} \nu(Y_t^a) dB_t^k
\\& -
f_t\trace \nabla^2 \nu (Y_t^a) dt    
-\sum_k \langle W_t^a,\nabla_{\sigma_k(Y_t^a)} \nu   \rangle \nabla_{\sigma_k(Y_t^a)} \nu  \;dt.
\end{align*}
This completes the proof.

 \end{proof}
 
 \begin{lemma}
 Let   $n_t^a=(\parals_t^a)^{-1}\nu(Y_t^a)$. Then
  $\lim_{a\to 0} n^a=n$, in the topology of semi-martingales.
\end{lemma}
\begin{proof}
By the definition, (\ref{E28}), the stochastic differential $D\nu(Y_t^a)$ is essentially $n_t^a$:
   \begin{equation}
  \Label{E41}
dn_t^a
  =(\parals_t^a)^{-1}\left(\langle \nabla \nu, \sigma_k(Y_t^a)\rangle \,dB_t^k+\f12 \trace \nabla^2 \nu(Y_t^a)\,dt\right).
 \end{equation}
 By the same computation,
  \begin{equation}
	\Label{E41.1}
dn_t
  =(\parals_t)^{-1}\left(\langle \nabla \nu, \sigma_k(Y_t)\rangle \,dB_t^k+\f12 \trace \nabla^2 \nu(Y_t)\,dt\right). \end{equation}
We recall that $Y^a\to Y$ and $ \paral^a(Y^a)\to \paral(Y)$. Let $$\Delta_{\partial M} \nu: =\trace \nabla^2 \nu,$$ where the trace
is taken in the vector space orthogonal to $\nu$.
It follows that
 $$ \nabla \nu(Y^a)\to\nabla \nu(Y), \quad
 \s(Y^a)\to\s(Y), \quad \Delta_{\partial M} \nu(Y^a)\to \Delta_{\partial M} \nu(Y),$$
all in the topology of UCP.  This implies that $n^a$ converges to $n$ in the topology of semi-martingales. 

\end{proof}

Let us define
 $$\underline{\Ric}(x)= \inf_{v\in T_{x}M, |v|=1} \left\{\Ric(v, v)\right\},\  x\in M\quad\hbox{and}\quad\underline{\Ric}=\inf_{x\in M}\underline{\Ric}(x).$$ We also define
$$ \underline {{\mcS}}(x)=\inf_{v\in T_{x}M, |v|=1, \<v, \nabla R\>=0} \left\{  {\mcS}(v,v)\right\},\  x\in M\quad\hbox{and}\quad\underline{{\mcS}}=\inf_{x\in M}\underline{{\mcS}}(x).$$

 \begin{lemma}\Label{Wbdd}
 For any $t>0$ and $a$, 
\begin{equation}\Label{BoundWa2}
|W_t^a|^2\le  |W_0^a|^2\; e^{-\int_0^t \underline{\Ric}(Y_s^a) \,ds-2\int_0^t\underline {{\mcS}}(Y_s^a) dL_s^a}.
\end{equation}
 Suppose that $\underline\Ric$ and $\underline{\mcS}$ are bounded. Then $\sup_{a} \|W^a\|_{\scrS_p([0,T])}$ is finite.
Furthermore 
\begin{equation}\Label{BoundWap}
\E\sup_{t\le T} | W_t^a|^p
\le |W_0^a|^pe^{-p\underline{\Ric}T}C(T,-p\underline{\mcS})
\end{equation}
where $C(T,\lambda)$ is defined in~\eqref{expLat}
 \end{lemma}
 \begin{proof}
 
 We begin with $W^a$. Firstly, 
 \begin{align*}
| W_t^a|^2=&|W_0^a|^2-\int_0^t\langle \Ric^\sharp (W_s^a),W_s^a\rangle\,ds-2\int_0^t\langle \mcS(W_s^a),W_s^a\rangle\,dL_s^a\\&-\int_0^tc_s^a\langle W_s^a,\nu_{Y_s^a}\rangle^2 \,ds
 \end{align*}
where $c_s^a$ is the scalar normal part of $W_t^a$, see (\ref{ca}). It is easy to see that $c_s^a>0$. So 
 \begin{align*}
|W_t^a|^2\le& |W_0^a|^2-\int_0^t \Ric (W_s^a,W_s^a)\,ds-2\int_0^t\langle \mcS(W_s^a),W_s^a\rangle\,dL_s^a\\
 \le&|W_0^a|^2 -\int_0^t|W_s^a|^2 \underline{\Ric}(Y_s^a) \,ds-2\int_0^t|W_s^a|^2\underline{\mcS}(Y_s^a) dL_s^a.
 \end{align*}
 This implies~\eqref{BoundWa2} and~\eqref{BoundWap} immediately follows.

 \end{proof}
\begin{lemma}\label{L7.6}
We also have  
\begin{equation}\Label{BoundWp}
\|W\|_{\scrS_p([0,T])}^p\le |W_0|^pe^{-pT\underline{\Ric}}C(T,-p\underline{\mcS})
\end{equation}
\end{lemma}
\begin{proof}
 Since $W^\e$ converges to $W$, a similar computation holds for $W^\e$, the conclusion for $(W_t)$ follows.
\end{proof}

\subsection{The Local Time}

 Let us recall the notation
 \begin{equation}
  \Label{E28}
  n_t^a=(\parals_t^a)^{-1}\nu(Y_t^a),\ \ric_t^a=(\parals_t^a)^{-1}\Ric^\sharp(\parals_t^a(\cdot)),\ \rms_t^a=(\parals_t^a)^{-1}\mcS (\parals_t^a(\cdot)).
 \end{equation}
 Denote $w_t^a=\parals_t^{-1}(Y^a) W_t^a$. Then
$$ dw_t^a= -c_a(t) f_a(t) n_t^a dt -\rms_t^a(w_t^a)dL_t^a-\f12\ric_t^a(w_t^a)\,dt.$$

%

Recall $c_a $ is a real valued stochastic process defined in (\ref{ca}). 
Let us define a new stochastic process \begin{equation}
\Label{E43}
\tilde c_a(t)=c_a(t)+\f12\|\nabla\nu(Y_t^a)\|_{\rm H.S.}^2+\f12\langle\ric_t^a(n_t^a),n_t^a\rangle,
\end{equation}
and also
\begin{equation}\Label{E31}
  C_a(t)=\int_0^tc_a(s)\,ds, \quad \quad \tilde C_a(s,t)=\int_s^t\tilde c_a(s)\,ds.
\end{equation}

In the tubular neighbourhood,  the following holds.
\begin{lemma}
\Label{L0}
Let $S$ be a bounded stopping time such that $Y_t\in F_0$ on $\{t<S\}$. 
For $0\le s<t\le S$, define
\begin{equation}
\Label{E45}
\begin{split}
\tilde r_t^a=&\int_0^t\langle w_s^{a,T},\na_{\s_k(Y_s^a)}\nu\rangle dB_s^k
-\f12 \int_0^t \langle \ric_s^a(w_s^{a,T}),n_s^a\rangle \,ds\\
&+\f12\int_0^t\langle w_s^{a,T},(\parals_s^a)^{-1}\trace \nabla^2\nu(Y_s^a)\rangle\,ds
\end{split}
\end{equation}
Then
\begin{equation}
 \Label{E46}
 f_a(t)=f_a(0)e^{-\int_0^t \tilde c_a(r)dr}+\int_0^te^{-\int_s^t \tilde c_a(r)dr}\,d\tilde r_s^a.
\end{equation}
 \end{lemma}

\begin{proof}
By Lemma \ref{covariant-formula}, the function $f_a(t)$ is a solution to the following equation,
\begin{align*}
df_t=-c_a(t) f_a(t)dt -\frac 12 \Ric(W_t^a, \nu(Y_t^a))dt
 +\sum_k\<W_t^{a,T},  \nabla_{\sigma_k} \nu (Y_t^a)\>\;dB_t^k\\
   + \f12\<W_t^{a,T}, \trace \nabla^2 \nu(Y_t^a)\>dt-\f12 f_a(t)\|\nabla \nu(Y_t^a)\|^2.
\end{align*}
Split the $W_t^a$ terms into its tangential and normal parts:
\begin{align*}
df_t=&-c_a(t) f_a(t)dt -\frac 12 f_a(t)\Ric(\nu(Y_t^a), \nu(Y_t^a))dt 
-\frac12 f_a(t)\|\nabla \nu\|^2dt\\
&-\frac 12 \Ric(W_t^{a,T}, \nu(Y_t^a))dt
 +\sum_k\<W_t^{a,T}, \nabla_{\sigma_k(Y_t^a)} \nu \> dB_t^k \\&
 +\frac{1}{2} \<W_t^{a,T}, \trace \nabla^2 \nu (Y_t^a) \>dt .
\end{align*}
The required identity follows from
 the variation of constant method .
 
%
\end{proof}

\begin{lemma}
 \Label{L1}
 Let $S$ be a stopping time as in Lemma~\ref{L0}.
 Define for $t\in [0,S]$,
 \begin{equation}
  \Label{E36}
  L_t^a=\int_0^t\f{2\,ds}{a\sinh\left(\f{2R_s^a}{a}\right)}.
 \end{equation}
 Let $p,q\in [1,\infty]$ and $r$ defined by $\di \f1r=\f1p+\f1q$.
Let $Z_t^a$ and $Z_t$ be continuous real semimartingales defined on $[0,S]$. Then
\begin{equation}
 \Label{E37}
 \begin{split}
 &\left\|\int_0^\cdot Z_s^a\; d L_s^a-\int_0^\cdot Z_s \; dL_s\right\|_{\scrS_r}\\&
 \le \left\|Z^a-Z\right\|_{\scrS_p}\left\|L_S^a\right\|_q+\left\|L^a-L\right\|_{\scrS_p}\left(\|Z\|_{{\scrS}_q}+\|Z\|_{{\SH}_q}\right).
 \end{split}
\end{equation}

\end{lemma}
\begin{proof}
 We have for $t\in[0,S]$
 \begin{align*}
  &\int_0^tZ_s^a\,dL_s^a-\int_0^tZ_s\,dL_s\\
  &=\int_0^t(Z_s^a-Z_s)\,dL_s^a+\int_0^tZ_s\,d(L_s^a-L_s)\\
  &=\int_0^t(Z_s^a-Z_s)\,dL_s^a+\int_0^t(L_s^a-L_s)\,dZ_s+Z_t(L_t^a-L_t).
 \end{align*}
Since $L^a$ is nondecreasing we have 
 $$\|L_\cdot^a\|_{{\SH}_q([0,S])}^q=\E\left( \sup_{s\le S} |L_s^a|^q\right)=\|L_\cdot^a\|_{{\scrS}_q([0,S])}^q,$$
so we get  by ~\eqref{HSH}, 
$$\left\| \int_0^\cdot(Z_s^a-Z_s)\,dL_s^a\right\|_{{\SH}_r([0,S])}
\le \|Z_\cdot^a-Z_\cdot\|_{\scrS_p([0,S])}  \|L_\cdot^a\|_{{\scrS}_q([0,S])}^q.$$
Similar estimates holds for the last two terms on the right hand side of the identity. This concludes the proof.
\end{proof}
Let $S$ be a stopping time such that $Y_t\in F_0$ on $\{t<S\}$. 
\begin{lemma}
\Label{L2}
Let $\alpha_t=\sup_{s\le t} \{s\le t: Y_s \in \partial M\}$. Suppose that $t\not \in \SR(\omega)$.
  For all $s,t\in [0,S]$ satisfying $s<t$, \begin{equation}
\Label{E38}
\begin{split}
\lim_{a\to 0} e^{-\int_s^tc_a(u)\,du}= 1&\quad \hbox{if}\quad s>\a_t\\
\lim_{a\to 0}  e^{-\int_s^tc_a(u)\,du}=0&\quad \hbox{if}\quad s<\a_t.
\end{split}
\end{equation}
The convergence is in probability.
As a consequence, for all $p\ge 1$, 
\begin{equation}
 \Label{E38bis}
\lim_{a\to 0}  \E\left[\int_0^S\left|e^{-C_a(t)}-\1_{\{s>\a_t\}}\right|^p dt\right]=0;
\end{equation}
\begin{equation}
 \Label{E38ter}
\lim_{a\to 0}  \E\left[\int_0^S\left(\int_0^t\left|e^{-\int_s^tc_a(u)\,du}-\1_{\{s>\a_t\}}\right|^p ds\right)dt\right]=0.\end{equation}
\end{lemma}
\begin{proof}
From~\eqref{E38} it is easy to get~\eqref{E38bis} and~\eqref{E38ter} using 
the fact that $e^{-\int_s^tc_a(u)\,du}$ and $\1_{\{s>\a_t\}}$ are uniformly bounded and Fubini-Tonelli theorem. 
 
 So let us prove~\eqref{E38}.
Write 
\begin{align*}
\int_s^tc_a(u)\,du=\int_s^t\f{2\cosh\left(\f{2R_u^a}{a}\right)}{a\sinh\left(\f{2R_u^a}{a}\right)}
\times \f2{a\sinh\left(\f{2R_u^a}{a}\right)}\,du.
\end{align*}
If $s>\a_t$ then there exists $\e(\om)>0$ such that for $u\in [s,t]$, $R_u>\e(\om)$. Since $R^a\to R$ in UCP topology, $\sup_{u\in [s,t]} c_a(u)$ converges to $0$ in probability, and this implies that  $\di e^{-\int_s^tc_a(u)\,du}\to 1$.

If $s<\a_t$ then $L_t-L_s>0$. Indeed, this would be true if $R_t$ was a reflected Brownian motion. But by Girsanov transform we obtain that the law of $R_t$ is equivalent to the one of a reflected Brownian motion (Lemma~\ref{EqRBM}). So this is true. 

Now we have 
\begin{align*}
\int_s^t\f{2\cosh\left(\f{2R_u^a}{a}\right)}{a\sinh\left(\f{2R_u^a}{a}\right)}
\times \f2{a\sinh\left(\f{2R_u^a}{a}\right)}\,du>&\int_s^t\f{2}{a}
\times \f2{a\sinh\left(\f{2R_u^a}{a}\right)}\,du\\
=&\f2a(L_t^a-L_s^a).
\end{align*}
Since $R^a\to R$ in UCP topology we have that $L^a\to L$ in UCP topology. So for all $a_0>0$ the $\liminf$ of the right hand side is larger than $\di \f2{a_0}(L_t-L_s)$. This yields 
$$
\limsup_{a\to 0}e^{-\int_s^tc_a(u)\,du}<e^{-\f2{a_0}(L_t-L_s)}
$$
in probability. Letting $a_0\to 0$ we get 
$$
\lim_{a\to 0}e^{-\int_s^tc_a(u)\,du}=0 \quad \hbox{in probability}.
$$
\end{proof}

From this result we get the following 
\begin{corollary}
\Label{C2}
Define 
\begin{equation}
\Label{E43.1}
\tilde c(t)=-\f12\|\nabla\nu(Y_t)\|_{\rm H.S.}^2-\f12\langle\ric_t(n_t),n_t\rangle
\end{equation}
where $\ric_t$ is defined in~\eqref{wt0}.
For $s,t\in [0,S]$ satisfying $s<t$ we define
\begin{equation}
\Label{E44.1}
\tilde C(s,t)=\left\{ \begin{array} {ll}\int_s^t\tilde c(s)\,ds, \quad &\hbox{if}\quad s>\a_t\\
+\infty,  \quad  & \hbox{if}\quad s\le \a_t
\end{array}\right.
\end{equation}
Then the following convergence holds in probability for $\tilde C_a(s,t)$ defined in~\eqref{E31}:
\begin{equation}
\Label{E47}
\begin{split}
\lim_{a\to 0} e^{-\tilde C_a(s,t)}= e^{-\tilde C(s,t)}.
\end{split}
\end{equation}
Consequently,
for all $p\ge 1$, 
\begin{align}
 \Label{E47bis}
\lim_{a\to 0}  \E\left[\int_0^S\left|e^{-\tilde C_a(t)}-e^{-\tilde C(t)}\right|^p dt\right]= 0,\\
 \lim_{a\to 0} \E\left[\int_0^S\left(\int_0^t\left|e^{-\tilde C_a(s,t)}-e^{-\tilde C(s,t)}\right|^p ds\right)dt\right]= 0.
 \Label{E47ter}
 \end{align}
 With these notations Equation~\eqref{E26} rewrites as
\begin{equation}
 \Label{E26.1}
 f(t)=f(0)e^{-\tilde C(t)}+\int_0^te^{-\tilde C(s,t)}\,d\tilde r_s
\end{equation}
where
\begin{equation}
\Label{E45-2}
\begin{split}
\tilde r_t=&\int_0^t\langle w_s^{T},\na_{\s_k(Y_s)}\nu\rangle dB_s^k
-\f12 \int_0^t \langle \ric_s(w_s^{T}), n_s\rangle \,ds\\
&+\f12\int_0^t\langle w_s^{T},\parals_s^{-1}\trace \nabla^2\nu(Y_s)\rangle\,ds.
\end{split}
\end{equation}

\end{corollary}
\begin{proof}
The convergences are obvious. For equation (\ref{E45-2}),  we see if $t<\zeta$,
$e^{-\tilde C(s,t)}\to 0$ for any $s\ge 0$. Hence
$$f(t)=f(0)+\int_0^t\,d\tilde r_s=f(0)+\tilde r_t.$$
If $t\ge\zeta$, 
$$\int_0^te^{-\tilde C(s,t)}\,d\tilde r_s=\int_{\alpha(t)}^te^{-\tilde C(s,t)}\,d\tilde r_s$$
and so
$$f(t)=\int_{\alpha(t)}^t d\tilde r_s=\tilde r(t)-\tilde r(\alpha(t)).$$
\end{proof}

The new expression  \eqref{E26.1} for $f(t)$ is the same form as the equation for $f_a(t)$:
 \begin{align*}f_a(t)=&f_a(0)e^{-\int_0^t \tilde c_a(r)dr}+\int_0^te^{-\int_s^t \tilde c_a(r)dr}\,d\tilde r_s^a\\
 =&f_a(0)e^{-\tilde C_a(t)}+\int_0^te^{-\tilde C_a(s,t)}\,d\tilde r_s^a.\end{align*}
We observe also that $c_a(t)$ does not converge to a finite stochastic process, hence we only expect
that $f_a(t)$ converges to $f(t)$ in a weak sense.
Especially it is only for a set of $t$ of full measure that $f_a(t)\to f(t)$. This will be made precise in part~\ref{NP}

\subsection{Convergence of the tangential parts}\Label{TP}

We will see that tangential parts of $W_t^a$ converges in UCP topology, as for normal parts we have to exclude the boundary times. But both of them converge in $L^p([0,T]\times \Omega)$, this will be proved at the very end of the proof. Let us begin with the first convergence.

\begin{lemma}
\Label{L4}
As $a\to 0$,   $W_\cdot^{a,T}\to W_\cdot^T$ in UCP topology.
\end{lemma}
\begin{proof}


Since $\parals_t^a\to \parals_t$ in the UCP topology it is sufficient to prove that  $w_\cdot^{a,T}\to w_\cdot^T$ in the UCP topology.
We recall from Lemma \ref{covariant-formula}, the term involving $c_a(t)$ cancels and we have

\begin{align*}
DW_t^{a,T} =&-\f12 \left(\Ric^\sharp(W_t^{a,T})\right)^T\, dt-\f12 f_a(t)\left(\Ric^\sharp(\nu(Y_t^a))\right)^T\, dt- \mcS(W_t^a)\;  dL_t^a
\\&-f_a(t) \nabla_{\sigma_k} \nu(Y_t^a) dB_t^k -
\frac{1}{2}  f_t\trace \nabla^2 \nu (Y_t^a) dt
\\&-\frac 12 \langle W_t^{a,T},\nabla_{\sigma_k} \nu (Y_t^a) \rangle \nabla_{\sigma_k} \nu (Y_t^a)  \;dt
  \\
&-\sum_k\<W_t^{a,T}, \nabla_{\sigma_k} \nu(Y_t^a) \> \nu(Y_t^a)dB_t^k -
 \<W_t^{a,T},  \trace \nabla^2 \nu (Y_t^a)  \>\nu(Y_t^a)dt\\
&   +f_t \|\nabla \nu(Y_t^a)\|^2\nu(Y_t^a) dt.
\end{align*}
Hence
\begin{equation}
\Label{E60}
\begin{split}
dw_t^{a,T}=&-\f12 \ric_t^a(w_t^{a,T})\,dt-\f12 f_a(t)\left(\Ric^\sharp(\nu(Y_t^a))\right)^T\, dt
-\rms_t^a(w_t^{a,T})\,dL_t^a\\
&-f_a(t) (\parals_t^a)^{-1} \nabla_{\sigma_k} \nu(Y_t^a) dB_t^k -
\frac{1}{2}  f_a(t)  (\parals_t^a)^{-1}  \trace \nabla^2 \nu (Y_t^a) dt\\
&
-\frac 12 \langle w_t^{a,T},(\parals_t^a)^{-1} \nabla_{\sigma_k} \nu (Y_t^a) \rangle (\parals_t^a)^{-1} \nabla_{\sigma_k} \nu (Y_t^a)  \;dt
  \\
&-\sum_k\<w_t^{a,T}, (\parals_t^a)^{-1} \nabla_{\sigma_k} \nu(Y_t^a) \> n_t^a\,dB_t^k \\&-
 \<w_t^{a,T},  \trace (\parals_t^a)^{-1} \nabla^2 \nu (Y_t^a)  \>n_t^a\;dt\\
&   +f_a(t) \|\nabla \nu(Y_t^a)\|^2n_t^a dt.
\end{split}
\end{equation}
We define the processes $v_t^{a,T}$, $v_t^{a,\n}$ :

\begin{equation}
\Label{E49}
\begin{split}
v_u^{a,T}(\cdot)=&-\f12 \int_0^u\ric_t^a(\cdot)\,dt-\int_0^u\rms_t^a(\cdot)\,dL_t^a\\
&
-\frac 12 \int_0^u\langle \cdot,(\parals_t^a)^{-1} \nabla_{\sigma_k} \nu (Y_t^a) \rangle (\parals_t^a)^{-1} \nabla_{\sigma_k} \nu (Y_t^a)  \;dt
  \\
&-\sum_k\int_0^u\<\cdot, (\parals_t^a)^{-1} \nabla_{\sigma_k} \nu(Y_t^a) \> n_t^a\,dB_t^k \\&-
\int_0^u  \<\cdot,  \trace (\parals_t^a)^{-1} \nabla^2 \nu (Y_t^a)  \>n_t^a\;dt.\end{split}
\end{equation}
Also,

\begin{equation}
\Label{E51}
\begin{split}
v_u^{a,\nu}=&-\int_0^u \f12 (\parals_t^a)^{-1}\left(\Ric^\sharp(\nu(Y_t^a))\right)^T\, dt
-\int_0^u  (\parals_t^a)^{-1} \nabla_{\sigma_k} \nu(Y_t^a) dB_t^k\\
& -\frac{1}{2} \int_0^u   (\parals_t^a)^{-1}  \trace \nabla^2 \nu (Y_t^a) dt
   +\int_0^u \|\nabla \nu(Y_t^a)\|^2n_t^a dt.
\end{split}
\end{equation}
With these notations and the expression for $f_a(t)$ in formula ~\eqref{E46} we have 
$$dw_t^{a,T}=dv_t^{a,T}(w_t^{a,T})+f_a(t)dv_t^{a,\nu}$$ and so
\begin{equation}
\Label{E53}
dw_t^{a,T}=dv_t^{a,T}(w_t^{a,T})+\left(f_a(0)e^{-\tilde C_a(t)}+\int_0^te^{-\int_s^t\tilde c_a(u)\,du}\,d\tilde r_s^a\right)dv_t^{a,\nu}.
\end{equation}
We also have 
\begin{equation}
\Label{E55}
dw_t^T=dv_t^T(w_t^T)+\left(f(0)e^{-\tilde C(t)}+\int_{0}^te^{-\tilde C(s,t)}\,d\tilde r_s\right)dv_t^{\nu},
\end{equation}
(recall that $\di e^{-\tilde C(s,t)}=0$ if $s<\a_t$), where
\begin{equation}
\Label{E52}
\begin{split}
v_u^{\nu}
=&-\f12 \int_0^u  \ric_t(n_t)^T\,dt+\int_0^u \|\na\nu(Y_t)\|^2n_t\,dt
-\int_0^u (\parals_t)^{-1} \nabla_{\s(Y_t)\,dB_t} \nu\\
&-\f12 \int_0^u (\parals_t)^{-1}\D^{h,T}\nu(Y_t)\,dt,
\end{split}
\end{equation}
\begin{equation}
\Label{E50}
\begin{split}
 v_u^{T}(\cdot)
=&-\f12 \int_0^u \ric_t(\cdot)\,dt-\int_0^u\rms_t(\cdot)\,dL_t
-\int_0^u\langle \cdot,(\parals_t)^{-1}\na_{\s^T(Y_t)\,dB_t}\nu\rangle n_t\\
&
-\f12\int_0^u \langle \cdot,(\parals_t)^{-1}\D^{h} \nu(Y_t)\rangle n_t\,dt\\&
-\sum_{j\ge 2}\int_0^u \langle \cdot,\na_{\s_j(Y_t)}\nu\rangle(\parals_t)^{-1} \nabla_{\s_j(Y_t)} \nu \,dt.
\end{split}
\end{equation}
We investigate further (\ref{E53})

\begin{align*}
dw_t^{a,T}=&dv_t^{a,T}(w_t^{a,T})+f_a(0)e^{-\tilde C_a(t)}dv_t^{a,\nu}\\
&+ \left( \int_0^te^{-\int_s^t\tilde c_a(u)\,du} \langle w_s^{a,T},\na_{\s_k(Y_s^a)}\nu\rangle dB_s^k\right)dv_t^{a,\nu}\\
&-\f12 \left( \int_0^te^{-\int_s^t\tilde c_a(u)\,du} \langle \ric_s^a(w_s^{a,T}),n_s^a\rangle \,ds\right)dv_t^{a,\nu}\\
&+\f12  \left( \int_0^te^{-\int_s^t\tilde c_a(u)\,du} \langle w_s^{a,T},(\parals_s^a)^{-1}\trace \nabla^2\nu(Y_s^a)\rangle\,ds
\right) dv_t^{a, \nu}.
\end{align*}
From this the required convergence should follow: when $a$ approaches zero,
$v_t^{a,\nu}$ approaches $v_t^{a}$ and  $v_t^{a,T}$ approaches $v_t^{T}$.
If furthermore if $f_a(0)\to f(0)$, then
$$\lim_{a\to 0} f_a(0)e^{-\tilde C_a(t)}
=f(0)e^{-\tilde C(t)}.$$
Hence the components of $w_t^{a,T}$ is the solution to a  system of non-Markovian stochastic differential equations
whose coefficients converge, and furthermore $v_t^{a,T}(w_t^{a,T})$ converges only in UCP, not in $\SH_p([0,T])$.
These factors explain why the proof below is long given this simple explanation. 
To prove that $v_t^{a,T}\to v_t^T$ in UCP topology, c.f. \eqref{E49} and \eqref{E50}, we only need to prove that 
$\di \int_0^\cdot \rms_t^a(\cdot)\,dL_t^a\to \int_0^\cdot \rms_t(\cdot)\,dL_t$ in UCP topology. 
This is a consequence of Lemma~\ref{L1}, together with the facts that UCP topology is equivalent to local convergence 
in ${\scrS}_p$ and that the random variables $L_S^a$ are uniformly bounded in $L^2$.

To make the rest of the proof more transparent let us define
\begin{equation}
\Label{E57}
\begin{split}
\tilde u_t=&\int_0^te^{-\tilde C(s,t)}\langle \cdot,\na_{\s^T(Y_s)\,dB_s}\nu\rangle\\
&-\f12 \int_0^t e^{-\tilde C(s,t)}\langle \ric_s(\cdot),n_s\rangle \,ds+\f12\int_0^te^{-\tilde C(s,t)}\langle \cdot, \parals_s^{-1}\D^{h,T} \nu(Y_s)\rangle\,ds;
\end{split}
\end{equation}
\begin{equation}
\Label{E56}
\begin{split}
\tilde u_t^a=&\int_0^te^{-\tilde C_a(s,t)}\langle \cdot,\na_{\s^T(Y_s^a)\,dB_s}\nu\rangle\\
&-\f12 \int_0^t e^{-\tilde C_a(s,t)}\langle \ric_s^a(\cdot),n_s^a\rangle \,ds+\f12\int_0^te^{-\tilde C_a(s,t)}\langle \cdot, \parals_s^{-1}\D^{h,T} \nu(Y_s^a)\rangle\,ds
\end{split}
\end{equation}
Then, by Lemma \ref{covariant-formula}, we may write
\begin{eqnarray}
\Label{E58}
f_a(t)&=&f_a(0)e^{-\tilde C_a(t)}+\int_0^t d\tilde u_s^a(w_s^{a,T}),\\
f(t)&=&f(0)e^{-\tilde C(t)}+\int_0^t d\tilde u_s(w_s^{T}).
\Label{E59}
\end{eqnarray}
Take these equalities back to equations (\ref{E53}) and (\ref{E55}), we see
\begin{equation*}
dw_t^{a,T}=dv_t^{a,T}(w_t^{a,T})+\left(f_a(0)e^{-\tilde C_a(t)}+\int_0^td\tilde u_s^a(w_s^{a,T})\right)dv_t^{a,\nu}.
\end{equation*}
We also have 
\begin{equation*}
dw_t^T=dv_t^T(w_t^T)+\left(f(0)e^{-\tilde C(t)}+\int_{0}^td\tilde u_s(w_s^{T})\right)dv_t^{\nu}.
\end{equation*} 
Let us simply compute the difference of the two matrices:
\begin{equation}
\Label{E61}
\begin{split}
&d(w_t^{a,T}-w_t^T)\\=&d(v_t^{a,T}-v_t^T)(w_t^{a,T})+dv_t^T(w_t^{a,T}-w_t^T)+f_a(t)\,d(v_t^{a,\nu}-v_t^\nu)\\
&+dv_t^\nu \left(f_a(0)e^{-\tilde C_a(t)}-f(0)e^{-\tilde C(t)}\right)\\
&+dv_t^\nu\int_{0}^td(\tilde u_s^a-\tilde u_s)(w_s^{a,T})+ dv_t^\nu \int_{0}^td\tilde u_s(w_s^{a,T}-w_s^T).
\end{split}
\end{equation}

Now we recall that convergence in UCP topology is implied by local convergence in ${\scrS}_1$. 
For a stopping time ${S'}$ smaller than $S$ we have 
\begin{equation}
\Label{E62}
\left\|\left(\int_0^\cdot dv_t^T (w_t^{a,T}-w_t^T\right)^{S'}\right\|_{{\scrS}_1}\le \|(v^T)^{S'}\|_{{\SH}_\infty}
\cdot \|(w_\cdot^{a,T}-w_\cdot^T)^{S'}\|_{{\scrS}_1}
\end{equation}
Since $v_0^T=0$ and  $v^T$ has locally bounded ${\SH}_\infty$ norm we can split the time interval and we only have to make the proof on $[0,S']$ where ${S'\le S}$ is  a stopping time  so that 
\begin{equation}
\Label{E62.1}
\|(v^T)^{S'}\|_{{\SH}_\infty}<1.
\end{equation}
Then using an argument analogous to that for (\ref{E62}) we see
\begin{equation}
\Label{E63}
\begin{split}
&\left\|\left(\int_0^\cdot dv_t^\nu\int_0^t d\tilde u_s (w_s^{a,T}-w_s^T\right)^{S'}\right\|_{{\scrS}_1}\\&\le \|(v^\nu)^{S'}\|_{{\SH}_\infty}
\cdot \left\|\left(\int_0^\cdot d\tilde u_t (w_t^{a,T}-w_s^T\right)^{S'}\right\|_{{\scrS}_1}\\
&\le \|(v^\nu)^{S'}\|_{{\SH}_\infty}\cdot \|(\tilde u)^{S'}\|_{{\SH}_\infty}\cdot  \|(w_\cdot^{a,T}-w_\cdot^T)^{S'}\|_{{\scrS}_1}
\end{split}
\end{equation}
Since  $v^{\nu}$ and $\tilde u$ have locally bounded ${\SH}_\infty$ norms, with the same argument we can take ${S'}$ so that \begin{equation}
\Label{E63.1}
\|(v^\nu)^{S'}\|_{{\SH}_\infty}\cdot \|(\tilde u)^{S'}\|_{{\SH}_\infty}<1.
\end{equation}
We want to prove that $\|(w_\cdot^{a,T}-w_\cdot^T)^{S'}\|_{{\scrS}_1}\to 0$ as $a\to 0$. Using 
(\ref{E61}-\ref{E62.1}), \eqref{E63}, \eqref{E63.1} and Gronwall lemma, it is sufficient to prove that 
\begin{equation}
\Label{E64}
\lim_{a\to 0}\left\| \left(\int_0^\cdot d(v_t^{a,T}-v_t^T)(w_t^{a,T})\right)^{S'}\right\|_{{\scrS}_1}= 0,
\end{equation}
\begin{equation}
\Label{E65}
\lim_{a\to 0}\left\|\left(\int_0^\cdot dv_t^\nu \left(f_a(0)e^{-\tilde C_a(t)}-f(0)e^{-\tilde C(t)}\right)\right)^{S'}\right\|_{{\scrS}_1}
=0.
\end{equation}
and
\begin{equation}
\Label{E66}
\lim_{a\to 0}\left\|\left(\int_0^\cdot dv_t^\nu\int_{0}^td(\tilde u_s^a-\tilde u_s)(w_s^{a,T})\right)^{S'}\right\|_{{\scrS}_1}
=0.
\end{equation}

For~\eqref{E64} we write 
\begin{equation}
\Label{E67}
\int_0^t d(v_s^{a,T}-v_s^T)(w_s^{a,T})=w_t^{a,T}(v_t^{a,T}-v_t^T)-\int_0^t
v_s^{a,T}-v_s^T)dw_s^{a,T}.
\end{equation}
From~\eqref{E60} and Lemma~\ref{Wbdd} we see that the processes $w^{a,T}$ are uniformly bounded in ${\SH}_2$. Since$v_t^{a,T}\to v_t^T$ in UCP topology, $v_t^{a,T}\to v_t^T$ locally in ${\scrS}_\infty$. We have
$$
\left\|\left(\int_0^\cdot
v_s^{a,T}-v_s^T)dw_s^{a,T}\right)^{S'}\right\|_{{\scrS}_2}\le \|(v^{a,T}-v^T)^{S'}\|_{{\scrS}_{\infty}}\cdot \|w_s^{a,T}\|_{{\SH}_2}
$$
and 
\begin{align*}
\|w_\cdot^{a,T}(v_\cdot^{a,T}-v_\cdot^T)\|_{{\scrS}_2}&\le \|(v^{a,T}-v^T)^{S'}\|_{{\scrS}_{\infty}}\cdot \|w_s^{a,T}\|_{{\scrS}_2}\\&\le 3\|(v^{a,T}-v^T)^{S'}\|_{{\scrS}_{\infty}}\cdot \|w_s^{a,T}\|_{{\SH}_2}.
\end{align*}
From this, ~\eqref{E67} and the fact that ${\scrS}_1$ norm is smaller than ${\scrS}_2$ norm, we obtain ~\eqref{E64}.

For  \eqref{E65} it is sufficient to compute the ${\SH}_2$ norm of 
$$
\left(\int_0^\cdot dv_t^\nu \left(f_a(0)e^{-\tilde C_a(t)}-f(0)e^{-\tilde C(t)}\right)\right)^{S'}
$$
and to use the dominated convergence theorem.

Finally let us prove \eqref{E66}. This can be done  by modifying ${S'}$, using the facts that the processes $W^{a,T}$
have uniformly bounded ${\scrS}_2$ norms and $\tilde u^a\to \tilde u$ in UCP topology. For this last point, use~\eqref{E47ter} 
in Corollary \ref{C2} and 
Corollary~\ref{C3}.
\end{proof}

\subsection{Convergence of the normal parts}\Label{NP}
\begin{lemma}
\Label{LpProd}
For all $p\in[1,\infty)$ and $T>0$, 
\begin{equation}
\Label{LpProdeq}
 \E\left[\int_0^T|f_a(t)-f(t)|^p\,dt\right]\to 0.
 \end{equation}
\end{lemma}
\begin{proof}
Write 
\begin{align*}
f_a(t)-f(t)&=\left(f_a(0)-f(0)\right)e^{-\tilde C_a(t)}+f(0)\left(e^{-\tilde C_a(t)}-e^{-\tilde C(t)}\right)\\
&+\int_0^t\left(e^{-\tilde C_a(s,t)}-e^{-\tilde C(s,t)}\right)d\tilde r_s^a
+\int_0^te^{-\tilde C(s,t)}\,d\left(\tilde r_s^a-\tilde r_s\right).
\end{align*}
The first term in the right converges to $0$ in $L^p([0,T]\times \P)$ due to the positiveness of $\tilde C_a(t)$. The second term in the right converges to $0$ due to~\eqref{E47bis}. For the last term in the right we use boundedness of $e^{-\tilde C(s,t)}$ and the fact that $\tilde r^a_t\to \tilde r_t$ in $\SH_p([0,T])$ due to~\eqref{E45} and~\eqref{E45-2} together with~\eqref{HSH} and Lemmas~\ref{L4} and~\ref{Wbdd}
which allow to prove that $w_t^{a,T}\to w_t^T$ in  $\scrS_q([0,T])$, $q\in[1,\infty)$.

We are left to prove that 
$$
\int_0^t\left(e^{-\tilde C_a(s,t)}-e^{-\tilde C(s,t)}\right)d\tilde r_s^a\to 0\quad \hbox{in}\quad L^p([0,T]\times \P.
$$
Here it is easier to replace $S$ by $T\ge S$ which is deterministic.
We have 
\begin{align*}
&\E\left[\int_0^T\left|\int_0^s\left(e^{-\tilde C_a(u,s)}-e^{-\tilde C(u,s)}\right)d\tilde r_u^a\right|^p\,ds\right]\\
&=\int_0^T\E\left[\left|\int_0^s\left(e^{-\tilde C_a(u,s)}-e^{-\tilde C(u,s)}\right)d\tilde r_u^a\right|^p\right]\,ds\\
&\le C(p,T)\int_0^T\E\left[\int_0^s\left|e^{-\tilde C_a(u,s)}-e^{-\tilde C(u,s)}\right|^p\,du\right]\,ds.
\end{align*}
The last inequality comes from the fact that the identity map from $\scrS_p([0,s])$ to $\SH_p([0,s])$ is continuous and bounded by $C(p,s)$ satisfying $0<C(p,s)\le C(p,T)$. Notice that the fact that $u\mapsto e^{-\tilde C_a(u,s)}-e^{-\tilde C(u,s)}$ is not adapted is not a problem since in $d\tilde r_u^a$ there is no integration with respect to $B^1$. We conclude with~\eqref{E47ter} which is easily seen to be true with $S$ replaced by $T$.
\end{proof}
\begin{lemma}
\Label{L3}
For all $p\in [1,\infty)$, $T>0$ and all  smooth 
 $\phi : M\to \RR_+$ vanishing in a neighbourhood of $\partial M$, 
 $\phi(Y_t^a)f_a(t)\to \phi(Y_t)f(t)$ in $\scrS_p([0,T])$.
\end{lemma}

\begin{proof}
Since $\phi$ is bounded and the processes $f_a(t)$ are uniformly bounded in $\scrS_p([0,T])$ independently of~$a$, it is sufficient to prove convergence in UCP topology. 

We have 
\begin{align*} \phi(R_t^a)f_a(t)-\phi(R_t)f(t)=\left(\phi(R_t^a)-\phi(R_t)\right)f_a(t)+\phi(R_t)\left(f_a(t)-f(t)\right).
\end{align*}
Since the processes $f_a(t)$ are uniformly bounded in $\scrS_p([0,T])$ independently of~$a$ and $R_t^a\to R_t$ in $\scrS_p([0,T])$, the fist term in the right 
converges to $0$ in UCP topology. Let us consider the second term:
\begin{align*}
 d\left(\phi(R_t)(f_a(t)-f(t))\right)=&(f_a(t)-f(t))d\phi(R_t)+ \phi(R_t)d\left(f_a(t)-f(t)\right)\\&
 +d\phi(R_t)d\left(f_a(t)-f(t)\right).
\end{align*}
The integral of the first term in the right converges to~$0$ in UCP topology, due to~\eqref{E47ter} and the fact that $\phi(R_t)$
has uniformly bounded absolutely continuous local characteristics.

On the other hand 
\begin{align*}
 &\phi(R_t)d\left(f_a(t)-f(t)\right)\\=&-\tilde c_a(t)\phi(R_t)(f_a(t)-f(t))\,dt +\phi(R_t)(\tilde c(t)-\tilde c_a(t))f(t)\,dt\\
 &+\phi(R_t)(f_a(t)-f(t))d\tilde r_a(t)+\phi(R_t) f(t)d(\tilde r_t^a-\tilde r_t).
\end{align*}

From subsection~\ref{TP} together with~\eqref{E45} and~\eqref{HSH} we get that $\tilde r^a\to \tilde r$ in semimartingale topology. 

So due to the presence of $\phi(R_t)$ which vanishes in a neighbourhood of $\partial M$ all the terms behave nicely, with the help of~\eqref{LpProdeq}.

Finally the covariance term can be treated with similar methods.
\end{proof}
With this we completed the proof of Theorem~\ref{T4} and close this section.
  

\appendix
\section{\\The half line example}
\Label{section-half-line}

On the half line we select a reflected Brownian motion with `good' sample path properties. 
To begin with, we consider two reflected Brownian motions: the solution to the Skorohod problem associated with a Brownian motion $x+B_t$ and the solution to the Tanaka problem associated with
$x+\int_0^t \sgn(x+B_s) dB_s$. The first is a stochastic flow, see Lemma \ref{reflected-BM} below, while the second is not. 

The solution and the derivative flow to the Skorohod problem for $x+B_t$ is approximated by solutions 
 and derivative flows to a family of SDEs with explicit drifts. Furthermore,  its derivative flow is shown to coincide with the damped parallel translation introduced in Appendix \ref{section:damped}.

Denote the space of real valued continuous function with $f(0)=0$ by $C_0(\R; \R)$ 
and its subset of non-negative valued functions by $C_0(\R; \R_+)$. To each $x\ge 0$ and $f\in C_0(\R, \R)$
there exists a unique nondecreasing  function
$h\in C_0(\R, \R_+)$ such that, for $g(t):=x+f(t)+h(t)$,  $\int_0^t \1_{\{0\}} (g(s))  dh(s) =h(t)$.
The pair $(g,h)$ is the solution to the Skorohod problem associated to $(x, f)$ and is denoted by
\begin{equation}
\label{SkPr}
\Phi_\cdot(x, f)=(g,h).
\end{equation}
 It is well known that $h(t)=-\inf_{0\le s \le t} \{(x+f(s))\wedge 0\}$.
 
If $B_t$ is a standard real valued Brownian motion, then the Skorohod problem 
defines the pair of stochastic processes $(X_t(x), L_t(x))$, and $L_t(x)$ is called the local time at $0$ of $X_t(x)$
and \begin{equation}
X_t(x)=x+B_t+L_t(x).
\end{equation}

On the other hand, the process $(|x+B_t|)$ is also a reflected Brownian motion. In fact, by Tanaka's formula, $
|x+B_t|=x+\beta_t+2\l_t^0(\omega)$
where $\beta_t=\int_0^t \sgn(x+B_s) dB_s$ is a Brownian motion and $\l_t^0$ is the local time of $x+B_t$.
The local time $\l_t^0$ is also the boundary time, i.e. the total time spent by $x+B_t$ on the boundary $\{0\}$ before time $t$.
The local time of a brownian motion was introduced by P. L\'evy (1940) as a Borel measurable function 
$ \Omega\times \R_+\times \R\to \R_+$ such that (1) for all $f\in \B_b(\R;\R)$,
 $\int_0^t f(x+B_s)ds=\int_\R f(a) \l_t^ada$, and  (2)  $(t,a)\mapsto \l^a_t(\omega)$ is continuous  a.s..
It is also well known that $\l_t^0=\lim_{\e \to 0}\frac{1}{ \e} \int_0^t \1_{[0, \e)}(x+B_s)ds= \lim_{\e \to 0}\frac{1}{ 2\e} \int_0^t \1_{(-\e, \e)}(x+B_s)ds$. 
It is clear that  $(|x+B_t(\omega)|, 2\l_t^0(\omega))$ is the solution
to the Skorohod problem associated with $x+\beta_\cdot(\omega)$, and $|x+B_t|$ is not a stochastic flow.

It turns out that $X_t(x)=x+B_t+L_t(x)$ has many nice properties.
Despite that the probability distribution of $X_t(x)$ is that of a reflecting Brownian, on a sample path level it is not at all
the reflected path! It is rather, a lifted path, by `the lower envelope' curve. The lower envelope curve is the 
unique continuous decreasing curve that is below the given curve $(B_t)$. Let $0<x<y$.  Let $\tau(y)=\inf \{t>0: X_t(y)=0\}$. 
It is clear that  $X_t(y)-X_t(x)=y-x$ until $X_t(x)$ reaches zero and
 the two stochastic processes coalesce  when $X_t(y)$ reaches zero.  
 If we compensate $x$ by $L_t(x)$, the two processes  $X_t(x+L_t(x))$ and $X_t(x)$ are equal for all $t$.

In Lemma \ref{reflected-BM} we summarise the sample properties of $X_t(x)$ and discuss differentiability of $X_t(x)$ with respect to $x$. 
These properties are elementary and not surprising. It is perhaps more surprising that these elementary properties of $X_t(x)$ are passed to the reflected Brownian motion on a manifold with boundary.
We should mention that differentiability with respect to the initial value was studied in   \cite[K. Burdzy]{Burdzy:09} 
and   \cite[S. Andres]{Andres:11} for domains in~$\R^2$ and polygons. 

For $s<t$ define $\theta_s B=B_{s+\cdot}-B_s$.  Let $\xi$ be an $\F_s$ measurable random variable and
$(X_{s,t}(\xi, \theta_s B), L_{s,t}(\xi, \theta_sB))$ the solution to the Skorohod problem for $(\xi, \theta_s B)$,
$$X_{s,t}(\xi, \theta_s B)=\xi+ (\theta_s B)_{t-s}+L_{s,t}(\xi, \theta_sB).$$
Define $L_{s,t}(\xi, \theta_sB)=0$ for $0\le t \le s$. For simplicity we also omit  $B$ in the flow, and 
write $X_t(x)$ for $X_t(x,B)$.
Let $T(x,y)=\inf \{t>0, \, X_t(x)=X_t(y)\}$ be the first time $X_t(x)$ and $X_t(y)$ meet.

\begin{lemma}\Label{reflected-BM}
The following statements hold pathwise.
\begin{enumerate}
\item For all $0\le s <t$, $x\in \R$,
 $$X_{s,t} (X_s(x, B), \theta_s B )=X_t(x,B), \quad L_t(x,B)= L_s(x,B)+L_{s,t}(X_s(x, B),\theta_s B).$$
 \item Let  $0<x<y$,  then $X_t(x)$ and $X_t(y)$ coalesce
 at the finite time $T(x,y)$. Furthermore $T(x,y)=\tau(y)$ and
$L_{\tau(y)}(x)=y-x$.
\item For all $t\ge 0$ and $x>0$, $X_t(x+L_t(x))=X_t(x)$ .
\item  For all $x\ge 0$ and $t\ge 0$, $$\partial_x X_t(x) =\left\{\begin{aligned}
&1,  \qquad t<\tau(x) \\
&0, \qquad   t>\tau(x)\end{aligned} \right..$$
\end{enumerate}
\end{lemma}

\begin{proof}
For part (1), we observe that,
$$X_{s,t}(X_s(x, B), \theta_sB) =x+B_t+ L_s(x, B)+L_{s,t}(X_s(x), \theta_sB).$$
Define $\tilde L(r)=L_r(x, B)$ when $r\le s$ and
$\tilde L(t)=L_s(x, B)+L_{s,t}(X_s(x), \theta_sB)$ for $t>s$. Then $\tilde L\in  C_0(\R_+, \R_+)$,
and $(X_{s,t}(X_s(x, B), \theta_sB) , \tilde L)$ solves the Skorohod problem
for $(x, B)$. By the uniqueness of the Skorohod problem, 
$X_t(x, B)=X_{s,t}(X_s(x, B), \theta_sB) $ and
$\tilde L(t)=L_t(x, B)$.

Part  (2).  From the construction of the solution of the Skorohod problem,
 it is easy to see that
$\tau(x) <\tau(y)$ and $X_t(y)-X_t(x)=y-x$ on $\{t<\tau(x)\}$, and $X_{\tau(y)}(y)=X_{\tau(x)}(x)$ on $\{t =\tau(y)\}$.
By the flow property, $0\le X_t(x)\le X_t(y)$ a.s. for all time. 
In other words,  the two curves $\{ X_s(x), s\le t\}$ and $\{X_s(y), s\le t\}$ are parallel on $\{t<\tau(x)\}$,
 until the lower curve hits zero after which the distance between the two curves decreases until  $X_t(y)$ reaches zero,
upon which point the two curves meet. The accumulated upward lift that $X_t(x)$ receives up to $\tau(y)$ is
 $$-\inf_{0\le s \le \tau(y, \omega)}\{ (x-y+y+B_s(\omega))  \wedge 0\}=y-x.$$
This shows that $X_{\tau(y)}(x)=0$ and together with the flow property we see the coalescence. We completed the proof that 
$T(x,y)=\tau(y)$ and 
$L_{\tau(y)}(x)=y-x$.

Part (3). On $\{t<\tau(x)\}$, $X_t(x+L_t(x))=X_t(x)$ trivially. If $t\ge \tau(x+L_t(x))$, 
\begin{align*}
X_t (x+L_t(x))&=x+L_t(x)+B_t(x)-\inf_{0\le s\le t}(  (x+L_t(x)+B_s )\wedge 0)\\
&=x+L_t(x)+B_t(x)-\inf_{0\le s\le t}(  (x+B_s )\wedge 0) -L_t(x)=X_t(x).
\end{align*}
If $\tau(x)\le t< \tau(x+L_t(x))$,
$X_t(x+L_t(x))=x+L_t(x)+B_t$ while $X_t(x)$ receives the kick of the size $L_t(x)$: $X_t(x)=x+B_t+L_t(x)$.

Part (4).  Take $t<\tau(x)$. Then  $t<\tau(x+\e)$ for $\e>t-\tau(x)$  and $X_t(x+\e)=X_t(x)+\e$, consequently $\partial_x X_t(x)=1$. Suppose $t> \tau(x)$. Then by part (2),
 $X_t(x)=X_t(x-\e)$ for any $\e<0$. 
If $0<\e<L_t(x)$,
$0\le X_t(x+\e)-X_t(x)\le X_t(x+L_t(x))-X_t(x)=0$. We used part (3) in the last step. Hence
$\partial_x X_t(x)=0$ for $t> \tau(x)$. This completes the proof.
\end{proof}

 A consequence of Lemma \ref{reflected-BM} is the following. If we pick up a time $t>\tau(x)$,  then
$X_t(x+L_t(x))$ must reach $0$ between $\tau(x)$ and $t$.

In the following we construct a family of stochastic processes $\{ X_\cdot^a(x), a> 0\}$ 
with the properties stated below illustrating the general construction.
(1) For each $a$, $X_\cdot^a$ is a stochastic flow and $x\mapsto X_t^a(x)$ is a diffeomorphism on its image; (2) they approximate the reflected Brownian motion;
(3) their derivatives approximate $\partial _x X_t(x)$.

Let $\phi(x)=\int_0^{x } e^{-\frac{y^2}{ 2}}dy$. For $x>0$ and $a>0$ let   $$u^a(x)=P(\tau(x)>a)
=\sqrt{\frac{2}{ \pi}}\phi\left({\frac{x}{ \sqrt a}}\right)=2P_a \1_{(0, x)}(0),$$ 
where $P_t$ denotes the heat semigroup.
Thus $\partial_x u^a=2p_a$ where $p_a$ is the Gaussian kernel.
 Formally  $u^0( 0)=P(\tau(0)>0) =0$ and for $x>0$, $u^0( x)=P(\tau(x)>0)=1$,
and $\frac{\partial } { \partial x} \ln u^0(x)=\frac{\partial }{ \partial x} \1_{(-\infty, x)}= \delta_0(x)$,  the Dirac mass at $0$.  
 Note that $\ln u^a$ is a concave function with positive gradient:
\begin{eqnarray}
 \Label{E2}
 \partial_x\ln u^a&=&\frac{1}{\sqrt a} (\ln \phi)'(\frac{x}{\sqrt a})=\f{1} {\sqrt a}  \f{e^{-\frac{x^2}{2a}}}{\phi(\frac{x} {\sqrt a})}>0;\\
 \partial_x^2(\ln u^a)&=&\frac{1}{  a} (\ln \phi)''(\frac{x}{ \sqrt a})
= - \frac{xe^{- \frac{x^2}{ 2a}}}{ a^{ \frac {3}{ 2}} \phi( \frac{x}{ \sqrt a})  }
 - \frac{ e^{- \frac {x^2}{ 2a}}  } { a\phi( \frac{x}{ \sqrt a}) } <0.
  \Label{E1}
\end{eqnarray}

\begin{proposition}\Label{P1}
Let $X_t^a(x)$ be the solution to
\begin{equation}
\Label{line}
X_t^a(x)=x+B_t+\int_0^t \partial_x \ln u^a(X^a_s(x))ds.
\end{equation}
Then $x\mapsto X_t^a(x)$ is an increasing function, $a\mapsto X_t^a(x)$ decreases as $a$ decreases to zero. For every $(t, x, \omega)$, $\lim_{a \downarrow 0} X_t^a(x)$ exists.
For every $x\ge 0$, the following holds for almost surely all $\omega$:
 $\lim_{a \downarrow 0} X_t^a(x) =X_t(x)$ for all $t$.
 \end{proposition}

\begin{proof}
That $ X_t^a(x)$ increases with $x$ follows from the comparison theorem one dimensional SDEs.
We also observe that the drift $\partial_x \ln u^a(x)$ in (\ref{line}) increases with $a$. 
$$\partial_a\partial_x\ln u^a=-\frac{1}{2 a^{ \frac {3}{ 2}}} (\ln \phi)'(\frac{x}{ \sqrt a})
- \frac{x}{ 2\sqrt a a^{ \frac {3}{ 2}}} (\ln \phi)''(\frac{x}{ \sqrt a})>0.$$
For $y>0$, define
$$ F(y)=  y (\ln \phi)''(y)+(\ln \phi)'(y).$$
It is clear that $F(y)$ is negative for $y$ sufficiently large.  
By the comparison theorem,
  $ X_t^a(x)$ increases with $a$ and  $\bar X_t(x)=\lim_{a\downarrow 0}X_t^a(x)$ exists
  for every $t,x, \omega$.
Consequently  $$A_t^x:=\lim_{a\downarrow 0}\int_0^t \partial_x \ln u^a(X_s^a(x))ds$$
  exists and
   $$\bar X_t(x)=x+B_t+A_t^x.$$
Let $f(t,a)=  \int_0^t \partial_x \ln u^a(X_s^a(x))ds$, which is positive and increasing with $t$. Thus $A_t^x$  is non-negative and nondecreasing in $t$. 
 
  Note that $\lim_{a\to 0}  \partial_x \ln u^a(x)=0$ for $x>0$, but the convergence is not uniform in $x$.  For $x\in (0, \sqrt a]$,
$$\partial_x \ln u^{a}\ge \frac{1}{ \sqrt a} \frac{e^{-\frac{x^2}{ 2a}} }{ {\frac{x}{ \sqrt a}}}\ge \frac{1-\frac{x^2}{ 2a}}{ x} >\frac{1}{ 2x}.$$
By comparison with the Bessel square process $Bes^2$ or standard criterion for diffusion process, for almost surely all $\omega$, $X_t^a(x)$ cannot reach~$0$.
  Next we observe that, $X_t(y)$ is a flow, 
   $X^a_t(y)>X^a_t(x)$ whenever $y>x$.
   Thus $X_t^a(x)>0$ for all $a>0$.  
The limiting process $\bar X_t(x)$ has the property:
$$\bar X_0(x)=x, \qquad \bar X_t(x) \ge 0.$$

Let $S<T$ be random times with $\bar X_t(x,\omega)>0$ for  $t \in [S, T]$. Let 
$$\delta(\omega)=\inf \{\bar X_t(x, \omega),\ t\in [S(\omega), T(\omega)]\}>0.$$
The function  $x\mapsto \partial _x \ln u^a(x)$ decreases,
$$ \int_S^T \partial_x \ln u^a(\bar X_s(x))ds \le  \int_S^T \partial_x \ln u^a(\delta(\omega))ds.$$
Then, since $X_s^a(x)\ge \bar X_s(x)$,
 \begin{align*}A_T(\omega)-A_S(\omega) &= \lim_{a\downarrow 0}   \int_S^T \partial_x \ln u^a(X_s^a(\omega))ds\\
&\le  \lim_{a\downarrow 0}   \int_S^T \partial_x \ln u^a(\bar X_s(\omega))ds\\
&\le \lim_{a\downarrow 0}   \int_S^T \partial_x \ln u^a(\delta(\omega))ds=0.
\end{align*}
 This implies that
$$\int_0^t \1_{\{\bar X_s(x)>0\}}d A_s^x=0$$
and $(\bar X_t(x), A_t^x)$ solves the Skorohod problem associated to $x+B_t$.
\end{proof}
%
%
%
%
\begin{lemma}
\Label{Lemma-u3}  For all $x>0$ and $a>0$,
 $\partial_x^3 \ln u^a>0$.
\end{lemma}
\begin{proof}
It is clearly sufficient to consider the case $a=1$.
$$
A(x)=\partial_x\ln u^1(x)=\phi'(x).
$$
We have from~\eqref{E1}
$$
A'(x)=-xA(x)-A^2(x)
$$
and this implies
\begin{equation}\Label{E4}
\phi'''(x)=A''(x)=(x^2-1)A(x)+3xA^2(x)+2A^3(x).
\end{equation}
It is clearly positive when $x\ge 1$. For $0\le x<1$,
$$
A(x)=\f{e^{-x^2/2}}{\int_0^xe^{-y^2/2}\,dy}>\f{1-x^2/2}{x}>\f{1-x^2}{3x}.
$$
Hence $$\phi'''(x)\ge  A(x) \left( (x^2-1)+3xA(x)\right)>0.$$
This completes the proof.
\end{proof}

Since  $\ln u_t^a(x)$ is smooth, the derivative flow
$V_t^a(x)=\partial_x X_t^a(x)$ exists and satisfies the linear equation $\dot V_t^a=( \partial^2_x \ln u^a) V_t^a$.
We prove that $V_t^a$ converges to $1$ when $t<\tau(x)$ and converges to $0$ when $t>\tau(x)$.
In the sequel, by  $V_{\tau (x+h)}^a (x)$ we mean $\partial_x X_s^a (x) |_{s={\tau (x+h)}} $,
and  $\tau(x+h)$ is not differentiated.

\begin{theorem}\Label{T1}
Let $X_t^a(x)$ be the solution to (\ref{line}). Let $V_t^a(x)=\partial_x X_t^a(x)$. Then  the following holds.
\begin{enumerate}
\item For all positive $a$ and $t$,  $x\mapsto V_t^a(x)$ is increasing and
$t\mapsto V_t^a(x)$ decreases.  For any $y>x$,
$$\lim_{a \to 0} V_{\tau (y)}^a (x)=0.$$
\item  For almost all $\omega$ the following holds for all $x>0$ and  $t\ge 0$ such that $t\not =\tau(x)$:
$$\lim_{a\downarrow 0} V_t^a(x) = \partial_x X_t(x).$$
Furthermore,
$$\lim_{a \to 0}\E\int_0^T \left| V_t^a(x)- { \partial_ x} X_t(x)\right| dt=0.$$
\end{enumerate}

\end{theorem}

\begin{proof} 
We observe that 
$$\frac{d}{ dt}V_t^a(x)=\partial_x ^2 (\ln u_t^a)(X_t^a(x))V_t^a(x),$$
and $V_0^a(x) =1$, leading to the formula,
\begin{equation}\Label{E3}
V_t^a(x)=e^{\int_0^t \partial^2_x( \ln u^a)(X_s^a(x)) ds}. 
\end{equation}
(1) Since $ \partial^2_x( \ln u^a)<0$, $V_t^a(x)$ decreases with $t$.
 We differentiate ~\eqref{E3} to see that
$$
\partial_xV_t^a(x)=V_t^a(x)\int_0^t  (\partial^3_x \ln u^a)(X_s^a(x)) V_s^a(x) ds.
$$
 Firstly, letting $V_0^a=1$. By 
Lemma \ref{Lemma-u3}, $\partial_x^3 \ln u^a>0$, so  $x\mapsto V_t^a(x)$ is increasing. 
 
Let $x, \omega$ be fixed.  Let $t\not =\tau(x, \omega)$ be a non-negative number and $h>0$.
There is a number $\theta(\omega)\in [0,1]$ s.t.
$$\frac{X_t^a(x+h)-X_t^a(x) }{ h}
 =  \partial_x X_t^a(x+\theta h)\ge  \partial_x X_t^a(x).$$
Since $\tau(x+h, \omega) \not= \tau(x,\omega)$ for a.e. $\omega$, for almost all $\omega$ we may set $t=\tau(x, \omega)$:
$$ 0\le \partial_x X_t^a (x)|_{t=\tau(x+h)}
 \le \frac{X_{\tau(x+h)} ^a(x+h)-X_{\tau(x+h)}^a(x) }{ h}
 \le\f{X_{\tau(x+h)} ^a(x+h)}{h}.$$
Take $h=L_t(x)$. By Proposition~\ref{P1}
$$\lim_{a\to 0}  X_{\tau(x+L_t(x))} ^a(x+L_t(x))=X_{\tau(x+L_t(x))} (x+L_t(x))=0  .$$
Thus for any $h>0$,  $$\lim_{a\downarrow 0}V^a_{\tau(x+h)} (x)= 0.$$

(2)  Let $x>0$. By Lemma \ref{reflected-BM}, $X_t(x+L_t(x))=X_t(x)$ for all $t\ge 0$.
So if $t>\tau(x)$, then $t\ge \tau(x+L_t(x))$. 
Suppose that $t>\tau(x)$. Since $V_t^a(x)$ decreases with $t$,
$$ 0\le V_t^a(x) \le V_{\tau(x+L_t(x))}^a(x).$$
By the conclusion of part (1), the right hand side converges to $0$ as $a\to 0$.
%

If $t<\tau(x)$, $X_t^a(x)>X_t(x)>0$ by comparison theorem for SDEs. Also $ \partial^2_x \ln u^a<0$,
$$1\ge \lim_{a\downarrow 0}\exp\left(\int_0^t \partial^2_x \ln u^a(X_s^a(x)) ds\right)
\ge \exp\left( \lim_{a\downarrow 0}\int_0^t \partial^2_x \ln u^a(X_s(x)) ds\right).$$
On the other hand for every $y$, $\partial^2_x \ln u^a(y)\to 0$ and 
$\inf_{s\in [0, t]}X_s(x)>0$ for $t<\tau(x)$. This concludes that $\lim_{a\downarrow 0}\exp\left(\int_0^t \partial^2_x \ln u^a(X_s^a(x)) ds\right)=1$. Note that $V_t^a(x)$ is uniformly bounded to conclude the convergence in $L^1$.
\end{proof}


\section{\\Convergence in $\scrS^p$ and in $\SH^p$}
\label{preliminary}
Let $a_0>0$ and let  $\{ (Y^a_t,\  t <\xi^a), {a\in [0,a_0)}\}$ be a family of continuous semi-martingales with values 
in a manifold $M$. If $U$ is an open domain in $M$, let $\tau^{U,a}$ denote the exit times:
$$\tau^{U,a}=\inf\{ t>0: Y_t^a\not \in U\}.$$

 \begin{definition}\Label{UCP} 
 \begin{enumerate}
 \item[(1)] We say that $Y^a$ converges to $Y^0$ in the topology of uniform convergence in probability on compact time sets (UCP) if 
 \begin{itemize}
 \item[(1a)] for all relatively compact open domain $U\subset M$, 
 $$\di\liminf_{a\to 0}\tau^{U,a}\ge \tau^{U,0},$$ 
 \item[(1b)]   for all  $t>0$, the following convergence holds in probability:
 $$\lim_{a\to 0}\di \sup_{s\le t\wedge \tau^{U,a}\wedge \tau^{U,0}}  \rho\left(Y_s^a, Y_s^0\right )\stackrel{(P)}{=} 0$$ \end{itemize}
 \item[(2)]
  \Label{Dp} 
  Let $p\in [1,\infty)$.
  We say that $Y^a$ converges to $Y^0$ locally in ${\scrS}_p$ if there exists an increasing sequence of stopping times 
  $(T_n)_{n\ge 1}$ with $\lim_{n\to \infty} T_n=\xi^0$ such that for some $a_1>0$ and for all  $a<a_1$ 
  and all $n\in \N$,  $T_n<\xi^a$ a.s.   and 
  \begin{equation}
   \Label{Dpeq}
   \lim_{a\to 0}\E\left[\sup_{t\le T_n}\rho^p(Y_t^a,Y_0^a)\right]=0
  \end{equation}
 \end{enumerate}
 \end{definition}
 
 Notice that  given (1b), condition (1a) is equivalent to  $\di
 \liminf_{a\to 0}\xi^a\ge \xi^0$.

Let $\DD$ denote the space of real-valued  adapted, C\`adl\`ag stochastic processes, defined on some filtered 
probability space $(\Omega,\SF, (\SF_t)_{t\ge 0}, \P)$ satisfying the usual conditions. We are mainly interested
in special semi-martingales from $\DD$. Below an element of $\DD$ is assumed to be also a special semi-martingale.
 
For two real valued semi-martingales $X, Y\in \DD$ we define the distance functions:
 $$r(X,Y)=\sum_{n>0} 2^{-n} \E \left(1\wedge \sup_{0\le t\le n}|X_t-Y_t|\right),$$
 $$\hat r(X,Y)=\sup_{|H|\le 1} r\left(\int_0^t  H_s d(X_s-Y_s) \right),$$ where the supremum is taken over all predictable processes $H$ bounded by $1$. The distance $r$ is compatible with UCP:  $$\sup_{0\le s\le t} |X_s^{(n)}-X_s| \to 0 \quad (\hbox{ in probability }) $$  
 for each $t>0$ if and only if $r(X^n-X)$ converges to $0$. 
 The distance $\hat r$ induces the semi-martingale topology on the vector space of semi-martingales.

Define \begin{align*}
\scrS^p&=\{ X\in \DD:  \|X\|_{\scrS^p} = \| \sup_t | X_t|\|_{L^p}<\infty\}, \\
\SH^p&= \{ X\in \DD:  \|X\|_{\SH^p}=
\inf \left\{  \left| \,|X_0|+ [M,M]_\infty^{\frac{1}{ 2}} +\int_0^\infty |dA_s|\right|_{L^p}<\infty\right\}.
\end{align*} where the infimum is taken over all 
semi-martingale decompositions $X=X_0+M+A$. 
 When the time interval is restricted to a finite time interval $[0,T]$ the notations will be $\scrS^p([0,T])$ and $\SH^p([0,T])$.
 
A semi-martingale is locally in $\scrS^p$ and $\SH^p$ if there exists a sequence of stopping times $T_n$ increasing to infinity such that
$X^{T_n}\chi_{\{T_n>0\}}$ are in these spaces. It is  prelocally in these spaces if all $X^{T_n-}\chi_{\{T_n>0\}}$ are,
where $X^{T-}(t)=X_t \chi_{ [0, T)}(t) +X_{T-} \; \chi_{[T, \infty)}(t)$. Let $(X^{(n)})$ and $X$ be semi-martingales.
Let $1\le p<\infty$.
If $X^{(n)}$ converges to $X$ is the semi-martingale topology, then there exists a subsequence that converges prelocally in 
$\SH^p$. If $X^{(n)}$ converges to $X$ prelocally in $\SH^p$ then it converges in the semi-martingale topology.
 Extension to $\R^k$-valued processes is done by considering the components. 
 
The following estimate of M. Emery is useful:
If $Y$ is a semi-martingale and $H$ a left continuous process with right limit, and $\frac{1}{ p}+\frac{1}{ q}=\frac{1}{ r}$ where
$p, q\in [1, \infty]$, then
\begin{equation}\Label{HSH}
\left\| \int_0^\infty H_s dZ_s \right\|_{\SH^r} \le \|H\|_{\scrS^p} \|Z\|_{\SH^q}.
\end{equation}

We review these convergence in the settings that the semi-martingales may have finite life times or take values in a manifold.
See ~\cite[M. Arnaudon and A. Thalmaier]{Arnaudon-Thalmaier:98} for details.

Let $Z_t=Z_0+M_t+A_t$ be a semi-martingale in $\RR^k$ with lifetime $\xi$ and the canonical 
 decomposition of $Z_t$ into starting point, local martingale $M_t$ starting at~$0$ and a finite variation process $A_t$  starting 
 at~$0$. Define 
 \begin{equation}
  \Label{vZ}
  v(Z)_t=\sum_{i=1}^k\left(|Z_0^i|+{\<}M^i,M^i{\>}_t^{1/2}+\int_0^t |dA^i|_s\right),\quad t<\xi.
 \end{equation}
Let $T>0$. 
 We say that a family of semi-martingales $Z^{(n)}$ converges to $0$ in ${\scrS}_p([0,T])$
 if $\di\E\left[\sup_{s\le T} |Z_s^{(n)}|^p\right]\to 0$.  It converges to $0$  in ${\SH}_p([0,T])$ if 
 $v(Z^{(n)})\to 0$ in  ${\scrS}_p([0,T])$.

To define this for a manifold valued stochastic process, we will use an embedding $\Phi : M\to \RR^k$. The definition will in fact 
 be independent of this embedding.
 
 \begin{definition}
 \Label{SM}
 Let $(Y_t^a)$ be a family of semi-martingales indexed by $a$.
   \begin{enumerate}
  \item We say that $Y^a$ converges to $Y^0$ in semi-martingale topology or in SM topology if the semi-martingale norm of $\Phi(Y^a)-\Phi(Y^0)$,
  $v\left(\Phi(Y^a)-\Phi(Y^0)\right)$, converges to~$0$ in UCP topology.
\item  \Label{Hp}
   Let $p\in [1,\infty)$.
  We say that $Y^a$ converges to $Y^0$ locally in ${\SH}_p$ if the processes $v\left(\Phi(Y^a)-\Phi(Y^0)\right)$
  converge to~$0$ locally in ${\scrS}_p$.  
  \end{enumerate} \end{definition}

  The convergence in the semi-martingale topology is stronger than convergence in the UCP topology. However it is a remarkable fact 
  that they coincide on the subset of martingales in the manifold. 
 The following characterisations of convergence will be very useful (see~\cite[M. Arnaudon and A. Thalmaier] {Arnaudon-Thalmaier:98}).
 
 \begin{proposition}
  \Label{P3}
  \begin{itemize}
   \item If  $Y^a\to Y^0$ as $a\to 0$ in UCP topology then for all $p\in[0,\infty)$ there
   exists a sequence $a_k\to 0$ such that $Y^{a_k}\to Y^0$ as $k\to\infty$ locally in ${\scrS}_p$.
   \item If  $Y^a\to Y^0$ as $a\to 0$ in SM topology then for all $p\in[0,\infty)$ there
   exists a sequence $a_k\to 0$ such that $Y^{a_k}\to Y^0$ as $k\to\infty$ locally in ${\SH}_p$.
   \item If for some $p\in[1,\infty)$ $Y^a\to Y^0$ as $a\to 0$ locally in $\scrS_p$ then $Y^a\to Y^0$ as $a\to 0$
   in UCP topology.
   \item If for some $p\in[1,\infty)$ $Y^a\to Y^0$ as $a\to 0$ locally in ${\SH}_p$ then $Y^a\to Y^0$ as $a\to 0$
   in SM topology.
  \end{itemize}
 \end{proposition}
 As a consequence, a standard way to establish UCP or SM convergence given by the following:
 \begin{corollary}
  \Label{C3}
  \begin{itemize}
   \item $Y^a\to Y^0$ as $a\to 0$ in UCP topology if and only if there exists  $p\in [1,\infty)$ such that for any $a_k\to 0$ there exists a subsequence 
   $a_{k_\ell}$ such that $Y^{a_{k_\ell}}\to Y^0$ locally in ${\scrS}_p$.
   \item  $Y^a\to Y^0$ as $a\to 0$ in SM topology if and only if there exists  $p\in [1,\infty)$ such that for any $a_k\to 0$ there exists a subsequence 
   $a_{k_\ell}$ such that $Y^{a_{k_\ell}}\to Y^0$ locally in ${\SH}_p$.
  \end{itemize}
 \end{corollary}

	For processes which take their values in a compact manifold and which are defined in bounded times, we have the following easy relations.
	\begin{corollary}
	\Label{C4}
	Assume that $M$ is compact and that all processes are defined on some deterministic time interval $[0,T]$.
	The following equivalences hold:
	\begin{itemize}
	\item $Y^a\to Y^0$ as $a\to 0$ in UCP topology;
	\item $Y^a\to Y^0$ as $a\to 0$ in $\scrS_p$ for some $p\in[1,\infty)$;
	\item $Y^a\to Y^0$ as $a\to 0$ in $\scrS_p$ for all $p\in [1,\infty)$.
	\end{itemize}
	Similarly, we have the equivalences
	\begin{itemize}
	\item $Y^a\to Y^0$ as $a\to 0$ in SM topology;
	\item $Y^a\to Y^0$ as $a\to 0$ in $\SH_p$ for some $p\in[1,\infty)$;
	\item $Y^a\to Y^0$ as $a\to 0$ in $\SH_p$ for all $p\in [1,\infty)$.
	\end{itemize}

	\end{corollary}
	
\section{\\ Ikeda and Watanabe's Damped parallel Translation}
\Label{section:damped}
The parallel transport $P_t$ along a semi-martingale $(Z_t)$
 is the semi-martingale with values in  $L(T_{Z_0}M,T_{Z_t}M)$  solving the Stratonovich SDE
\begin{equation}\Label{Pt}
 \circ dP_t={\mathfrak h}_{P_t}(\circ dZ_t), \quad P_0={\rm Id}_{T_{Z_0}M}
\end{equation}
where ${\mathfrak h}_{P_t}$ denotes horizontal lift to the orthonormal frame bundle.
We have identified $\R^d$ with $T_{Z_0}M$.  Parallel transport
is an isometry, a proof for its existence on manifolds 
with boundary can be found in  \cite[N. Ikeda and S. Watanabe]{Ikeda-Watanabe}. For simplicity we also use the notation
 $\paral_t(Z)$.

If $Z_t$ is a diffusion process with generator $\L=\f12 \D+U$, where $U$ is a time dependent vector field,
 remaining in the interior of $M$ for all time (which happens if $M$ has no boundary or if $U$ is sufficiently strong in a neighbourhood of the boundary), then 
the parallel transport $P_t$ along $Z_t$ is the diffusion process whose generator on differential $1$-forms  is $ \f12 \trace\nabla^2+\nabla_U$. If $\Delta^1 =-(d^\ast d +d d^\ast)$ is the Hodge Laplacian, ${\rm trace}\na^2=\Delta^1+\Ric$.
The damped parallel translation $W_t$ along $Z_t$ is the solution to the equation
\begin{equation}\Label{CovderW}
 DW_t=\left(\nabla_{W_t}U-\f12\Ric^\sharp (W_t)\right)\,dt, \quad W_0={\rm Id}_{T_{Z_0}M},
\end{equation}
where the covariant derivative $DW_t$ is defined to be $\parals_t d\left(\parals_t^{-1}W_t\right)$.
The process $(W_t)$  is a diffusion process with generator on $1$-forms $L^W$ :\begin{equation}\Label{GenWt}
\L^W\a= \f12\Delta^1 \a+\nabla_U\a +\a \left( \nabla_\cdot U\right).
\end{equation}

The fundamental property of $\L^W$ is its commutation with differentiation: 
\begin{equation}\Label{commut}
 d( \L f)=\L^W( df),\qquad f\in C^\infty(M).
\end{equation}
As a consequence, if $F\in C^{1,2}([0,T]\times M,\RR)$ is such that $F(t,Z_t)$ is a local martingale, then $dF(t,W_t)$ is also a local
martingale, where $dF$ is the differential of $F$ in the second variable. On the other hand,~\eqref{CovderW} together with the fact that
$P_t$ is an isometry yield estimations on the norm of $W_t$. This allows to estimate the norm of $dF$. 
Another fundamental property is that $W_t$ is the derivative of the flow corresponding to parallel couplings of $L$-diffusions. 


We construct a damped parallel transport along Brownian motion in a manifold with boundary. The covariant derivative $DW_t$ has three components: one coming from the behaviour in $M^0$
(the usual one), one tangential to $\partial M$ absolutely continuous with respect to $dL_t$ and involving the shape operator of 
$\partial M$, and the third is normal to $\partial M$ and has jumps. This is similar to what happens for the half line. 
Concerning the half line case 
the flow corresponding to parallel coupling of reflected Brownian motion is $ \partial_x X_t$, so in 
this case $W_t=\partial_x X_t$ and the study is complete.

We also define the second fundamental form and shape operator 
for level sets of the distance function to the boundary.  Let  $
 S(r)=\{y\in F_0,\ R(y)=r\}.$
Within a tubular neighbourhood $F_0$ of the boundary, $R$ is smooth around $x$. Let $r=R(x)$ and $\nu_x=\s_1(x)=\nabla R(x)$.
For $w\in T_xM$, $w'\in T_xS(r)$ we define $\Pi_x:T_xM\times T_xS(r)\to \R$ and ${\mcS}_x:T_xM\to T_xS(r)$ by
\begin{equation}
 \Label{So}
 \Pi(w,w')= \langle {\mcS}(w),w'\rangle=-\langle \nabla_w\nu,w'\rangle=-\nabla dR(w,w') 
\end{equation}
The bilinear map $\Pi$ is said to be the second fundamental form of $S(r)$ and
${\mcS}_x$ its shape operator or the Weingarten map. 


Assume that there exists $\delta_0>0$ and a tubular neighbourhood of $\partial M$ with radius
  $3\delta_0$. If $D$ is a set denote $\tau_D$  the exit time of $Y$ from $D$. 

\begin{lemma}\Label{EqRBM}
Let $U$ be a relatively compact set of $M$. Let $R_0<\delta_0$, $Y_0\in U$, $Y_t$ the reflected Brownian motion, and
$$\tau_{2\delta_0}=\inf\{t: R_t=2\delta_0\}.$$
Then  under a probability measure equivalent to $P$,
$\{ R_t, t <\tau_{\delta_0}\wedge \tau_U\wedge T\}$ is the solution to a Skorohod problem for a one dimensional Brownian motion on $\R_+$.
\end{lemma} 
\begin{proof}
Within $E_0$, because of Proposition~\ref{construction} (2),
$$R_t=R_0+B_t^1
 +\int_0^t \Delta R(Y_s) ds +L_t. $$
Let $Q$ be the probability measure whose density with respect to $P$ is
the exponential martingale of $-\int_0^t\Delta R(Y_s) dB^1_s$.
Then under $Q$, $\tilde B_t^1:=B_t^1
 +\int_0^t \Delta R(Y_s) ds$ is a Brownian motion. On the other hand,  $L_t$ is nondecreasing and $dL_t$ vanishes when  $R_t\not=0$. Since $R_t\ge 0$ we have $(R,L)=\Phi(0,\tilde B^1)$, the solution to Skorohod problem. See \eqref{SkPr}. 
 By the uniqueness of the Skorohod problem, under $Q$,
$\{ R_t, t <\tau_{2\delta_0}\wedge\tau_U\wedge T\}$ has the law of a one dimensional reflected Brownian motion.
\end{proof}
 

Let ${\mathfrak L}(\omega) =\{t\ge 0: Y_t(\omega) \in \partial M\}$ be the set of times that $Y_t$ spends on the boundary. It has Lebesque measure zero for a.s. all $\omega$ and its complement 
$$(0,\infty)\setminus {\mathfrak L}(\omega)=\cup _\alpha (l_\alpha(\omega), r_\alpha(\omega))$$ is the union of countably many disjoint open intervals,  the excursion intervals.
Denote the set of right end times of excursions by $\SR(\om)$:
$$\SR(\om)=\cup_\alpha\{ r_\alpha(\omega)\}.$$
  
 We are interested in defining a damped parallel translation $W_t$ along $Y_t$, which agrees with the usual one
 during an excursion, and pick up a change of direction when exiting the boundary. The normal direction 
 on the boundary is zero: we remove $-\<W_t, \nu(Y_t)\> \nu(Y_t)$ upon the process entering the boundary. 
We would have liked to define a stochastic processes $(W_t)$, if it were possible, with values in $n\times n$ matrices, satisfying

\begin{align*}
(\parals_t )^{-1}W_t =&(\parals_{r_\alpha} )^{-1}W_{r_{\alpha}} 
-\f12\int_{r_{\alpha}}^t (\parals_s )^{-1} \Ric^\sharp(\parals_s-) ((\parals_s )^{-1} W_s)\,ds\\
&- \int_{r_{\alpha}}^t (\parals_s )^{-1} \mcS(\parals_s-)\left( ( \parals_s )^{-1} W_s^\e\right)\,dL_s, \quad 
t\in (r_{\alpha}, r_{\alpha+1}) 
\end{align*}
\begin{align*}
&(\parals_{r_{\alpha+1}})^{-1}W_{r_{\alpha+1}}\\&
=(\parals_{(r_{\alpha+1})-} )^{-1}W_{(r_{\alpha+1})-}
-\langle W_{(r_{\alpha+1})-},\nu_{Y_{(r_{\alpha+1})-}}\rangle (\parals_{r_{(\alpha+1})-})^{-1} \nu_{Y_{(r_{\alpha+1})-}}.
\end{align*}
Given $(l_\alpha, r_\alpha)$, for any $\epsilon>0$ there is an excursion $(l_{\alpha'}, r_{\alpha'})$  such that $0<l_{\alpha'}-r_\alpha<\epsilon$, and so the heuristic definition given above does not make sense. 


We remedy this problem with an approximation adding jumps only on excursions
of size greater or equal to $\epsilon$. We consider the set of excursions of lengths greater or equal to a given size $\e>0$
and define
$$\SR_\e(\omega)=\{ s=r_\alpha(\om)\in\SR(\om): r_\a(\om)-l_\a(\om)\ge \e\},$$
where the excursions of size greater than or equal to $\varepsilon$ are ordered with $l_1$ the first time $Y_t$ hits the boundary and we consider only $\alpha \in \N$.
If $r_\alpha-l_\alpha\ge \varepsilon$,
\begin{align*}
(\parals_t )^{-1}W_t^\e &=(\parals_{r_\alpha} )^{-1}W_{r_{\alpha}}^\e 
-\f12\int_{r_{\alpha}}^t (\parals_s )^{-1} \Ric^\sharp(\parals_s-) ((\parals_s )^{-1} W_s^\e)\,ds\\
&- \int_{r_{\alpha}}^t (\parals_s )^{-1} \mcS(\parals_s-)\left( ( \parals_s )^{-1} W_s^\e\right)\,dL_s, \quad 
t\in (r_{\alpha}, r_{\alpha+1})
\end{align*}
\begin{align*}
&(\parals_{r_{\alpha+1}})^{-1}W_{r_{\alpha+1}}^\e\\&
=(\parals_{(r_{\alpha+1})-} )^{-1}W_{(r_{\alpha+1})-}^{\e}-\langle W_{(r_{\alpha+1})-},\nu_{Y_{(r_{\alpha+1})-}}\rangle (\parals_{r_{(\alpha+1})-})^{-1} \nu_{Y_{(r_{\alpha+1})-}}.
\end{align*}
where $(W^{\epsilon, T})$ denotes the tangential part of $W_t$. This takes into consideration those times slightly before  $(l_{\alpha+1},r_{\alpha+1})$ and is relevant to the
integration with respect to $L_t$.

Since $Y_t$ spends Lebesgue time $0$ on the boundary, for integration with respect to a continuous
process we could ignore the boundary process. We would like to simply remove the normal part of 
$W_t$ upon it touches down to the boundary. We are lead to the following alternative description. Let $v\in T_{Y_0}M$
and $t>0$, for almost surely all $\omega$, the following folds,

{\begin{align*}
(\parals_t )^{-1}W_t^\e  =& Id
-\f12\int_{0}^t (\parals_s )^{-1} \Ric^\sharp(\parals_s-) ((\parals_s )^{-1} W_s^\e )\,ds\\
&- \int_{0}^t (\parals_s )^{-1} \mcS(\parals_s-)\left( ( \parals_s )^{-1} W_s^\e\right)\,dL_s, \quad 
t\not \in \SR_\e(\omega) \\
(\parals_t)^{-1}W_t^\e
=& (\parals_{t-} )^{-1}W_{t-}^\e 
-\sum_{s\le t, s\in \SR_\e(\omega)} \langle W_{t-}^\e,\nu_{Y_{t-}}\rangle (\parals_{t-})^{-1} \nu_{Y_{t-}}.
\end{align*}
In other words, $W_t^\e$ is continuous at any time $t$ that is not an element of $\SR_\e(\omega)$,
and satisfies the following covariant equation
$$DW_t =-\frac{1}{2} \Ric^{\#}(W_t)\,dt-\mcS(W_t)\,dL_t.$$
If $t$ is the right hand side of an excursion, we remove the normal part of its component.
This description will be used in Theorem \ref{T3}.

Define
\begin{equation}
\label{T}
{\mathcal T}
=\{ \a_t \in \SH_\infty([0,T]):  \a_t\in T_{Y_t}^\ast M, \a_t=0 \hbox { on } \{Y_t\not\in D\} \hbox{ for some D} \}.
\end{equation}
where $D$ is relatively compact subset of $M$.
 This is the set of bounded semi-martingale with values in the pull back cotangent bundle by $Y_t$, with the property that there is relatively compact subset $D$ of $M$ such that $\a_t=0$ whenever $Y_t\not\in D$. 
Denote by $\tau^D$ the first exit time from $D$ by $Y_t$.
 
 \begin{definition}
\label{define-w}
The limit process  $W_t \in L(T_{Y_0}M,T_{Y_t}M)$, below in  Theorem \ref{T3},
 is said to be a solution to the following equation
 \begin{equation}\Label{W}
 DW_t=-\f12\Ric^\sharp(W_t)\,dt-\mcS(W_t)\,dL_t-\1_{\{t\in \SR(\om)\}}\langle W_{t-},\nu_{Y_t}\rangle \nu_{Y_t}, \quad W_0={\rm Id}_{T_{Y_0}M}.
\end{equation}
\end{definition} 
\begin{theorem}
 \Label{T3}
Let
 $W_t^\e$ the solution to
\begin{equation}\Label{We}
 DW_t^\e=-\f12\Ric^\sharp(W_t^\e)\,dt-\mcS(W_t^\e)\,dL_t
-\1_{\{t\in \SR_\e(\om)\}}\langle W_{t-}^\e,\nu_{Y_{t}}\rangle \nu_{Y_t}, \quad W_0^\e={\rm Id}_{T_{Y_0}M}.
\end{equation}
There exists an  adapted right continuous stochastic process $W_t$ such that 
$\lim_{\epsilon \to 0} W_t^\epsilon= W_t$ in UCP and in $\scrS_p $ for $M$ compact, and any $p\ge 1$. Furthermore for any $\alpha\in {\mathcal T}$,
 $$\lim_{\epsilon\to 0}\left(\int_0^{\cdot\wedge \tau_D}  \alpha_s (DW_s^\e) \right)\stackrel{\scrS_2} =\left(\int_0^{\cdot\wedge \tau_D} \alpha_s (DW_s) \right).$$
\end{theorem}
The same result but with different formulation can be found in ~\cite[N. Ikeda and S. Watanabe]{Ikeda-Watanabe}. 
We give a proof close to~\cite[N. Ikeda and S. Watanabe]{Ikeda-Watanabe}, which will be used for our approximation result (Theorem~\ref{T4} and 
Corollary~\ref{T4bis}).

\begin{proof}
Since the definition and convergence are local in $Y$, we can assume that $M$ is compact.
Since $Y_t(\om)$ has a finite number of 
excursions larger than $\e$, the process $W_t^\e$ is a well defined right continuous process. We first prove that as $\e\to 0$, 
$(W_t^\e)_{t\in [0,T]}$ converges in $\scrS_p$ to a process which we will call $(W_t)_{t\in [0,T]}$. 

Using the parallel translation process $\parals_t$ along $Y_t$, we reformulate the  equation as an equation in the linear space
$L(T_{Y_0}M, {T_{Y_0}}M)$. Set 
\begin{align}
\Label{wt0}
\ric_t=\parals_t^{-1}\circ \Ric_{Y_t}^\sharp\circ \parals_t, \quad \rms_t=\parals_t^{-1}\circ \mcS_{Y_t}\circ \parals_t,
\quad
\n_t=\parals_t^{-1}(\nu_{Y_t}).
\end{align}
Then $W_t^\epsilon$ is a solution to (\ref{We}) if and only if
 $w_t^\epsilon=(\parals_t )^{-1}W_t^\e$ satisfies the following equations.
  For any  $r_\alpha\in \SR_\e(\om)$,
\begin{align*}
w_t^\e =&-w_{r_{\alpha}}^\e 
-\f12\int_{r_{\alpha}}^t (\parals_s )^{-1} \ric_s (w_s^\e)\,ds
- \int_{r_{\alpha}}^t  \rms_s (w_s^\e)\,dL_s, \quad 
t\in (r_{\alpha}, r_{\alpha+1}) \\
w_{r_{\alpha+1}}^\e =&w_{(r_{\alpha+1})-}^\e 
-\langle w_{(r_{\alpha+1)})-}^\e,\n_{(r_{\alpha+1})-}\rangle \n_{(r_{\alpha+1}-)},
\end{align*}
This means $(w_t^\e)$ satisfies the following equation:
 \begin{equation}\Label{wteps}
dw_t^\e=-\f12\ric_t(w_t^\e)\,dt-\rms_t(w_t^\e)\,dL_t-\1_{\{t\in \SR_\e(\om)\}}\langle w_{t-}^\e,n_t\rangle n_t, \quad w_0={\rm Id}.
\end{equation}
Let $0<\e'<\e$, the difference between $w_t^\e$ and $w_t^{\e'}$ is given by
\begin{align*}
 w_t^\e-w_t^{\e'}=&\int_0^t\left(-\f12\ric_s(w_s^\e)+\f12\ric_s(w_s^{\e'})\right)\,ds+
 \int_0^t\left(-\rms_s(w_s^\e)+\rms_s(w_s^{\e'})\right)\,dL_s\\
 &-\sum_{\{s\in \SR_\e(\om)\cap[0,t]\}}\langle w_{s-}^\e,n_s\rangle n_s
 +\sum_{\{s\in \SR_{\e'}(\om)\cap[0,t]\}}\langle w_{s-}^{\e'},n_s\rangle n_s\\
 =&-\f12\int_0^t\ric_s(w_s^\e-w_s^{\e'})\,ds-\int_0^t\rms_s(w_s^\e-w_s^{\e'})\,dL_s\\
 &-\sum_{\{s\in \SR_{\e}(\om)\cap[0,t]\}}\langle w_{s-}^\e-w_{s-}^{\e'},n_s\rangle n_s
 +\sum_{\{s\in (\SR_{\e'}(\om)\backslash\SR_{\e})\cap[0,t]\}}\langle w_{s-}^{\e'},n_s\rangle n_s.
\end{align*}
Consequently,
\begin{align*}
\|w_t^\e-w_t^{\e'}\|^2=&-\int_0^t\langle\ric_s(w_s^\e-w_s^{\e'}),w_s^\e-w_s^{\e'}\rangle\,ds-2\int_0^t\langle \rms_s(w_s^\e-w_s^{\e'}),w_s^\e-w_s^{\e'}\rangle \,dL_s\\
&-2\sum_{\{s\in \SR_{\e}(\om)\cap[0,t]\}}\langle w_{s-}^\e-w_{s-}^{\e'},n_s\rangle^2 \\&
 +2\sum_{\{s\in (\SR_{\e'}(\om)\backslash\SR_{\e})\cap[0,t]\}}\langle w_{s-}^{\e'},n_s\rangle 
\langle n_s, w_{s-}^\e-w_{s-}^{\e'}\rangle
\end{align*}
which yields 
\begin{align*}
\|w_t^\e-w_t^{\e'}\|^2\le &
2\sum_{\{s\in (\SR_{\e'}(\om)\backslash\SR_{\e})\cap[0,t]\}}\langle w_{s-}^{\e'},n_s\rangle 
\langle n_s, w_{s-}^\e-w_{s-}^{\e'}\rangle
\\&+\int_0^t \|w_s^\e-w_s^{\e'}\|^2\left(-\underline{\Ric}(Y_s) \,ds-2\underline {{\mcS}}(Y_s) dL_s\right)
\end{align*}
and
\begin{equation}
\Label{bdddeltawt}
\|w_t^\e-w_t^{\e'}\|^2\le K_t
+\int_0^t \|w_s^\e-w_s^{\e'}\|^2\left(\rho\,ds+2C dL_s\right)
\end{equation}
where $\rho, C\ge 0$, $-\rho$ is a lower bound for the Ricci curvature, $C$ is an upper bound for the norm of the shape operator,
 and  
\begin{equation}
 \Label{Kt}
 K_t=\sup_{s\le t}\left|2\sum_{\{r\in (\SR_{\e'}(\om)\backslash\SR_{\e})\cap[0,s]\}}\langle w_{r-}^{\e'},n_r\rangle 
\langle n_r, w_{r-}^\e-w_{r-}^{\e'}\rangle\right|.
\end{equation}
So using Gronwall lemma we get 
 \begin{equation}
  \Label{bounddeltawt}
  \sup_{s\le t}\|w_s^\e-w_s^{\e'}\|^2\le K_te^{\rho t +2CL_t}.
 \end{equation}

 On the other hand it is a remarkable but not surprising fact that  each term $$\langle w_{s-}^{\e'},n_s\rangle 
\langle n_s, w_{s-}^\e-w_{s-}^{\e'}\rangle$$
 can be written as a stochastic integral over an interval not containing any excursion 
 of size larger than
 $\e$. This comes from the fact that the normal part of $w_t^{\e'}$ is set to zero at the end of each excursion of size at least $\e'$.
 More precisely,
 $$
 \langle w_{s-}^{\e'},n_s\rangle 
\langle n_s, w_{s-}^\e-w_{s-}^{\e'}\rangle=\int_u^sd\langle w_{r-}^{\e'},n_r\rangle 
\langle n_r, w_{r-}^\e-w_{r-}^{\e'}\rangle
 $$
where $u$ is the last vanishing time of $\langle w_{r}^{\e'},n_r\rangle $ before $s$. Now since we are outside excursions of size larger than $\e'$ the process inside the integral has no jumps, and since the range of $\rms_r$ is orthogonal to $n_r$,
the process inside the
 integral is a continuous semi-martingale whose drift is absolutely continuous 
 with respect to $ds$ with bounded derivative, see~\eqref{wteps} and~\eqref{E41.1}. 
 Consequently letting 
 $$ C_t:=2\sum_{\{s\in (\SR_{\e'}(\om)\backslash\SR_{\e})\cap[0,t]\}}\langle w_{s-}^{\e'},n_s\rangle 
\langle n_s, w_{s-}^\e-w_{s-}^{\e'}\rangle$$
we can write 
\begin{equation}
\Label{Cteq}
C_t=\int_0^t a_s ds+b_sdB_s
\end{equation}
where $a_s$, $b_s$ are uniformly bounded, $B_s$ is a real-valued Brownian motion and $a_s$, $b_s$ vanish outside $U(\e)\cap [0,t]$ where $U(\e)$ is the set of times not 
 contained in excursions larger than~$\e$.
 
 From this we get for $q>1$
\begin{equation}
\Label{Ktq}
\E[K_T^q]^{1/q}\le C_q\E\left[\lambda(U(\e)\cap [0,T])\right]
\end{equation}
for some $C_q>0$, where $\lambda$ is the Lebesgue measure in $\R$.
On the other hand by~\eqref{expLat} in Corollary~\ref{convLat} the random variable $e^{\rho T+2CL_T}$ has finite moments of any order. 
As a consequence, using~\eqref{bounddeltawt} and H\"older inequality, for any $p\ge 1$
 $$
 \E[|\sup_{t\le T}w_t^\e-w_t^{\e'}|^p]^{1/p}\le C_p'\E\left[\lambda(U(\e)\cap [0,T])\right].
 $$
for some $C_p'>0$.
 The left hand side goes to $0$ as $\e\to 0$, so for $\e_n\to 0$, $w^{\e_n}$ is a Cauchy sequence in $\scrS_p$, 
 it converges to some process $w$. Clearly $w$ does not depend on the sequence.  Letting
 $W_t=\parals_tw_t$  then $W_t^\e$ converges to $W_t$ in $\scrS_p$.
 
 Let us now prove that $DW^\e$ converges to $DW$ in the sense given by theorem~\ref{T3}. 
 Let
 $(\a_t)_{t\in [0,T]}$ be a $\SH_\infty$ semimartingale taking its values in $T^\ast_{Y_t}M$, bounded by $1$.
 We have for $0<\e'<\e$
 \begin{align*}
  &\int_0^t\a_s \left( DW_s^\e\right) - \int_0^t\a_s \left( DW_s^{\e'}\right)\\
  &=\a_t \left(W_t^\e-W_t^{\e'}\right)-\int_0^t D\a_s \left(  W_s^\e-W_s^{\e'}\right).
 \end{align*}
 By (\ref{HSH}),  a result of M. Emery, we see that
 \begin{align*}
 &\E\left[\sup_{t\le T}\left\|\int_0^t \a_s \left(DW_s^\e\rangle\right) - 
 \int_0^t \a_s \left(  DW_s^{\e'}\right) \right\|^2\right]\\
  &\le \|\a\|_{\scrS_\infty([0,T])}\|W^\e-W^{\e'}\|_{\scrS_2([0,T])}
  +\|W^\e-W^{\e'}\|_{\scrS_2([0,T])}\|\a\|_{\SH_\infty([0,T])}.
 \end{align*} 

From the first part of the proof and the assumption on $\a$, we get that 
  for $\e_n\to 0$, 
 $\left(\int_0^\cdot \a_s \left( DW_s^{\e_n}\right)\right)_{t\in [0,T]}$ is a Cauchy sequence in $\scrS_2([0,T])$, 
 so it converges to some process which is linear in $\alpha_s$. Consequently 
 we denote it by $\left((\int_0^\cdot \a_s \left( DW_s\right)\right)_{t\in [0,T]}$.
\end{proof}


%
%
%

\providecommand{\bysame}{\leavevmode\hbox to3em{\hrulefill}\thinspace}

\end{document}